%% file: main.tex
\documentclass[12pt]{article}

\title{\textsc{Higher Order Unbounded Rough Drivers}}

\include{preamble}

\begin{document}

\begingroup
\renewcommand{\thefootnote}{\fnsymbol{footnote}}
\footnotetext[2]{ \textsc{Department of Mathematics, University of Agder, Kristiansand, Norway}}
\endgroup

\footnotetext[1]{ Email: \texttt{torstein.nilssen@uia.no}}
\footnotetext[2]{ Email: \texttt{jonas.p.vean@uia.no}}

\author{
Torstein Nilssen\textsuperscript{$\dagger$,1} \\
\and 
Jonas P\hspace{0.01em}.~Vean\textsuperscript{$\dagger$,2}}

\makeatletter
\let\runningtitle\@title
\makeatother
\newcommand{\runningauthor}{\textsc{T.~Nilssen \& J.~P\hspace{0.01em}.~Vean}}

\pagestyle{fancy}
\fancyhf{} 
\renewcommand{\headrulewidth}{0pt}

\setlength{\headheight}{14.5pt}

\fancyhead[C]{\ifodd\value{page}\runningtitle\else\runningauthor\fi}
\fancyfoot[C]{\thepage}

\setcounter{footnote}{1} 
\renewcommand{\thefootnote}{\fnsymbol{footnote}} 
\maketitle
\renewcommand{\thefootnote}{\arabic{footnote}}
\setcounter{footnote}{0} 

\begin{abstract}
    \noindent
    We construct a framework for higher order unbounded rough drivers defined by rough paths over vector fields. This allows us to consider a purely Eulerian perspective for rough transport equations where the driving rough path vector field is allowed to have temporal $\mathfrak{p}$-variation for any $\mathfrak{p} \in (1,\infty)$. We use this framework to show the well-posedness of an equation with two transport terms; one driven by a rough path vector field with aforementioned $\mathfrak{p}$-variation in time, and the other driven by a DiPerna--Lions-type drift vector field. 

    \vspace{0.6cm}

    \vspace{0.5cm}
\end{abstract}

\thispagestyle{empty}

\newpage
\tableofcontents
\thispagestyle{empty}
\newpage

\section{Introduction} 
\label{sec:intro}
\setcounter{page}{1}

The goal of this article is to provide a purely Eulerian framework for rough partial differential equations of the form
\begin{equation} \label{eq:main eq}
    \partial_t u = b \cdot \nabla u + \dot{\mathrm{X}} \cdot \nabla u, \qquad u|_{t=0} = u_0 \in L^2(\R^d)
\end{equation}
where we will assume $b$ to have low spatial regularity in the spirit of the seminal work by DiPerna--Lions on transport equations \cite{diperna1989ordinary} and $\mathrm{X}$ to have low temporal regularity in the spirit of rough path theory as introduced by Terry Lyons \cite{Lyons98}.

We accomplish this by expanding the framework of \emph{unbounded rough drivers} as introduced in \cite{BaiGub2017} in two ways; first, we generalize the temporal regularity of $\mathrm{X}$ to be of finite $\mathfrak{p}$-variation for $\mathfrak{p} \in (1,\infty)$. 
Second, we allow the driving noise $\mathrm{X}$ to be a \emph{rough vector field} in the sense that $\mathrm{X}\colon [0,T] \rightarrow \clW^N(\R^d;\R^d)$ gives rise to a suitable rough path lift in a class $\clW^N$ of vector fields on $\R^d$ with appropriate regularity. Rough vector fields of this form appear naturally in the study of rough flows \cite{BRS, BailleulRiedel2019}, common noise McKean--Vlasov equations \cite{COGHI20211} and recently in the study of slow-fast systems \cite{debussche2024roughanalysisscalesystems, LiSobczak, LuongoTriggiano}.


We first describe how to give an intrinsic notion of a solution of \eqref{eq:main eq} through an iteration procedure as in \cite{BaiGub2017}, which was developed akin to Davie's expansion \cite{davie2008}. For simplicity of exposition, assume $b \equiv 0$ so that we are considering the equation
$$
\partial_t u =  \dot{\mathrm{X}} \cdot \nabla u.
$$
Assume $t \mapsto \mathrm{X}_t$ is of finite $\mathfrak{p}$-variation for $\mathfrak{p} \in (1,\infty)$ and let $N = \lfloor \mathfrak{p} \rfloor$ denote its integer value. We integrate the above equality from $s$ to $t$ to obtain 
\begin{align*}
    \delta u_{st}(x) & = \int_s^t  \dot{\mathrm{X}}_r(x) \cdot \nabla u_r(x) \,\mathrm{d}r \\
    &=  \int_s^t  \dot{\mathrm{X}}_r(x) \cdot \nabla  u_s(x) \,\mathrm{d}r  + \int_s^t  \dot{\mathrm{X}}_{r_1} (x) \cdot \nabla \int_s^{r_1} \dot{\mathrm{X}}_{r_2}(x) \cdot \nabla u_{r_2} (x)\,\mathrm{d}r_2 \,\mathrm{d}r_1 
\end{align*}
where we have inserted the first equality into the integral. Continuing in this way, we obtain the Davie-type expansion
\begin{equation} \label{eq:expansion no drift}
    \delta u_{st}  = \sum_{n=1}^N \mathrm{A}_{st}^{n} u_s + u_{st}^{\natural}
\end{equation}
having defined the formal expressions
\begin{equation} \label{eq:URD formally}
    \mathrm{A}_{st}^{n} \,\phi = \int_s^t \int_s^{r_1} \cdots \int_s^{r_{n-1}} \dot{\mathrm{X}}_{r_1} \cdot \nabla \left(  \dots \left( \dot{\mathrm{X}}_{r_n} \cdot \nabla \phi \right) \dots \right)\,\mathrm{d}r_n  \dots \,\mathrm{d}r_1
\end{equation}
as well as 
\begin{equation} 
\label{eq:remainder formally}
    u_{st}^{\natural} = \int_s^t \int_s^{r_1} \cdots \int_s^{r_{N}} \dot{\mathrm{X}}_{r_1} \cdot \nabla \left(  \dots \left( \dot{\mathrm{X}}_{r_{N+1}} \cdot \nabla u_{r_{N+1}}\right) \dots \right) \,\mathrm{d}r_{N+1}  \dots  \,\mathrm{d}r_1 .
\end{equation}
We have truncated the sum to order $N$ so that by formal power counting we have that $(s,t) \mapsto u^{\natural}_{st}$ is of time variation $\mathfrak{p}/N+1< 1$, and should therefore be negligible on small time-scales.  

Following \cite{BaiGub2017}, we develop an $L^2(\R^d)$-theory for \eqref{eq:main eq} using the notion of \emph{renormalized solutions} as in DiPerna--Lions in \cite{diperna1989ordinary}; a solution $u$ of \eqref{eq:main eq} is called renormalized if $u^2$ is also a solution (the initial condition for $u^2$ is now, of course, $u_0^2$). Uniqueness and stability for $u^2$ can be deduced from a rough Gronwall lemma, using the positivity $u^2 \geq 0$. Thus, the main ingredient in proving uniqueness is proving, under suitable regularity assumptions on $b$ and $\mathrm{X}$, that $u^2$ satisfies \eqref{eq:main eq}. 

The main difficulty in this direction is that the weak formulation \eqref{eq:expansion no drift} includes spatial distributions, so that products are not canonically defined. Instead, this is accomplished by doubling the variables, which effectively amounts to finding the dynamics of $u^{\otimes 2}_t(x,y) :=  u_t(x)\,u_t(y)$ in the rough path formulation. We deduce the dynamics for $u_t^2 = \lim_{\varepsilon\to 0} T^*_{\varepsilon} u_t^{\otimes 2}$ using the so-called \emph{blow-up transformation} $T_{\varepsilon}$ as introduced in \cite{BaiGub2017}. 
We show that both steps -- the doubling of variables and blow-up transformation -- correspond to similar operations on the rough vector fields (or rather, tensor fields), which allows us to track suitable regularity assumptions.

\subsection{Related works} \label{sec:related works}

The first works using an intrinsic notion of a solution of rough path transport equations is found in \cite{BaiGub2017} and \cite{DFS} using somewhat different approaches. Since then, a number of works related to rough transport noise have appeared, ranging from linear equations (\cite{hocquet2018energy, HN2021, FNS, GerasimovicsHocquetNilssen2021, BDFT2021}), regularization by noise \cite{Catellier2016, Nilssen2020}, non-linear equations (\cite{DGHT2019,HocquetNilssenStannat2020, Hocquet2021, HocquetNeamtu2024} as well as applications within fluid equations (\cite{HLN1, HLN2021, CrisanHolmLeahyNilssen22a, CrisanHolmLeahyNilssen22b, CHNR, FHLN, roveri2024wellposednessrough2deuler, GLN, LuongoTriggiano}).


Below we go more in depth on the works closest in spirit to the present paper.

\subsubsection*{Time-space factorization}

From its inception in \cite{BaiGub2017,DGHT2019}, the majority of published works in the unbounded rough drivers framework and rough path transport equations literature assume that the noisy vector field admits a factorization in time and space, namely,
\begin{equation} 
\label{eq:time space factorization}
    \mathrm{X}_t(x) = \sum_{k=1}^K \xi_k(x) \,\mathrm{Z}_t^k
\end{equation}
where the path $\mathrm{Z} = (\mathrm{Z}^1, \dots, \mathrm{Z}^K)$ can be lifted to a $\mathfrak{p}$-variation rough path $\mathbf{Z} = (\mathrm{Z},\mathbb{Z})$ for $\mathfrak{p} \in [2,3)$. Even if not written explicitly, the apparent regularity of $\xi_k$ needed in \cite{BaiGub2017,DGHT2019} is the following: in order to prove the \emph{existence} of a solution, each vector field $\xi_k$ must be such that the differential operator
$$
\Div( \xi_k \,\cdot \,)\colon W^{3,\infty}(\R^d) \rightarrow W^{2,\infty}(\R^d)
$$
is a bounded and continuous linear map, which is satisfied provided $\xi_k \in W^{2,\infty}(\R^d; \R^d)$ and $\Div \,\xi_k \in W^{2,\infty}(\mathbb{R}^d)$ for all $k=1, \dots, K$; for \emph{uniqueness} and renormalizability, the stronger assumption $\xi_k \in W^{3, \infty}(\R^d;\R^d)$ is sufficient. 

We mirror both these assumptions in the present paper by requiring that the rough vector field $\mathrm{X}$ (not necessarily time-space factorized as in \eqref{eq:time space factorization}) satisfies the criteria $\mathrm{X}_t \in W^{N,\infty}(\R^d;\R^d)$, $\Div\, \mathrm{X}_t \in W^{N,\infty}(\R^d)$ for existence and $\mathrm{X}_t \in W^{N+1,\infty}(\R^d;\R^d)$ for uniqueness; the sufficient regularity criteria for the corresponding tensor-valued lift $\mathbf{X}$ are morally similar -- see Sections \ref{sec:existence} and \ref{sec:renormalization} for details.

\subsubsection*{Low temporal regularity}
Rough path transport equations on the form 
\begin{equation} \label{eq:friz eq}
    \partial_t u =  \sum_{k=1}^K \dot{\mathrm{Z}}_t^k\, \xi_k \cdot \nabla u
\end{equation}
where $\mathfrak{p} \geq 3$ was considered in the work \cite{BDFT2021} which uses an intrinsic notion of a solution based on the controlled rough path framework of Gubinelli \cite{gubinelli2004}. Their approach is to show a product formula for the solution of \eqref{eq:friz eq} and a solution of a rough path continuity equation
\begin{equation} \label{eq:continuity}
    \partial_t \rho = \sum_{k=1}^K \dot{\mathrm{Z}}_t^k\, \Div(\xi_k \rho),
\end{equation}
namely that $\langle \rho_t, u_t \rangle = \langle \rho_0, u_0 \rangle$.

Coupled with a forward-backward duality trick, this gives uniqueness of \eqref{eq:main eq} and \eqref{eq:continuity} as long as one can prove \emph{existence} of these solutions. The latter is accomplished by showing that the flow
\begin{equation} \label{eq:friz flow}
    \dot{\phi}_t(x) + \sum_{k=1}^K \dot{\mathrm{Z}}_t^k\, \xi_k(\phi_t(x)) = 0   , \qquad \phi_0(x) = x \in \R^d
\end{equation}
gives rise to explicit solutions of \eqref{eq:friz eq} and \eqref{eq:continuity} through the formulae $u_t(x) = u_0(\phi_t^{-1}(x))$ and $\rho_t  = (\phi_t)_{\sharp} \rho_0$ respectively. Let us remark that giving meaning to \eqref{eq:friz flow} and the above mentioned product formula requires more regularity from the coefficients than in the present work. Moreover, it is not clear how to include a drift term $b$ in this framework since in this case the flow equation is ill-posed.\footnote{One could imagine constructing \emph{regular Lagrangian flows} as in \cite{diperna1989ordinary} to give meaning to these equations by relaxing to almost every initial condition $x \in \R^d$, but this point of view seems to be missing from the literature.}
On the other hand, our approach leverages a purely Eulerian framework to allow the regularity of $b$ and $\mathrm{X}$ to be outside what is needed to give meaning to Lagrangian coordinates. 


\subsection{Main contributions and result}
The present paper presents a systematic generalization of the techniques introduced in \cite{BaiGub2017}. Since we are working solely with the Eulerian formulation, the present paper allows for regularity assumptions on the coefficients outside what is needed to give meaning to Lagrangian coordinates. We give an informal summary of our main result in the following theorem.

\begin{theorem}
    Let $\mathfrak{p} \in (1,\infty)$ and denote by $N = \lfloor \mathfrak{p} \rfloor$ its integer value. Assume the path $\mathrm{X}\colon [0,T] \rightarrow W^{N+1,\infty}(\R^d;\R^d)$ can be lifted to a strong geometric rough path $\bX$ with $W^{N+1,\infty}$-valued components. Assume the drift $b \in L^1([0,T]; L^{\infty} \cap W^{1,1}_{\loc} ( \R^d ;\R^d))$ is such that $\Div \,b \in L^1([0,T];L^{\infty}(\R^d))$. For any $u_0 \in L^2(\R^d)$ the equation
   $$
        \partial_t u = b \cdot \nabla u + \dot{\bX} \cdot \nabla u 
    $$
    with initial condition $u|_{t=0} = u_0$ is well-posed and $\|u_t \|_{L^2} \lesssim \|u_0\|_{L^2}$. 
\end{theorem}

The rough path lift $\bX$ of the vector field $\mathrm{X}$ is defined in Definition \ref{def:sobolev rough path}. Existence of a solution is obtained in Theorem \ref{thm:existence} and uniqueness is proved in Theorem \ref{thm:uniqueness}.

We note that the DiPerna--Lions drift $b$ is chosen as an example, and our techniques and framework are not limited to this setting. We expect that our approach can be used to generalize pre-existing results, mentioned in Section \ref{sec:related works}, that use the unbounded rough drivers framework when $\mathfrak{p} \in [2,3)$ to the setting $\mathfrak{p} \geq 2$.

\subsection{Structure of the paper}

We recall the basic notions of $\mathfrak{p}$-variation rough path theory, the relevant function spaces, and our definition of tensor-valued fields in Section \ref{sec:notation}, as well as establishing conventions and notation that will appear throughout the article. 

In Section \ref{sec:URD} we rigorously give meaning to the Davie-type expansion \eqref{eq:expansion no drift} by defining \emph{unbounded rough drivers}, namely a family of differential operators $\bA =  \{ \mathrm{A}^{n} \}_{n = 1}^N$ satisfying certain natural compatibility conditions; see Definition \ref{def:URD}. Using only this compatibility condition, we show in Subsection \ref{sec:a priori} how to obtain a priori estimates of the remainder term $u^{\natural}$ in terms of the unbounded rough driver $\bA$, the drift term $\mu$ and the solution $u$. 

In Section \ref{sec:constructing URD} we show how to construct the unbounded rough driver $\bA$ from an \emph{admissible tensor field rough path} $\mathbf{X}$ through so-called {\it casting operators}, thus giving rigorous meaning to the notion of a solution of \eqref{eq:main eq}; see Definition \ref{def:main eq}. The construction is performed in two parts: first, by defining and studying the casting operator on generic tensor fields in Section \ref{sec: casting operators}, and secondly, by extrapolating to tensor field rough paths that are defined in Section \ref{sec:sobolev RP} using a Hopf-algebraic framework deduced from examining the algebraic structure of their comprising tensor fields. These tensor field rough paths are the appropriate rough path lifts of rough vector fields $\mathrm{X}$ as discussed. Various technical results are incorporated into this section: broadly speaking, Section \ref{sec: casting operators} discusses various results related to the spatial regularity of $\mathbf{X}$, whereas Section \ref{sec:sobolev RP} highlights aspects regarding its temporal regularity and structure. An example related to stochastic analysis is included as a proof of concept for rough vector fields that do not admit time-space factorization as in Section \ref{sec:related works}.

Section \ref{sec:existence} constructs a solution of \eqref{eq:main eq} when the rough vector field is strong geometric (see Definition \ref{def:strong geometric}). We first show that if $u$ is a \emph{positive} solution, then we necessarily have $\sup_{t \in [0,T]} \|u_t \|_{L^1(\R^d)} \lesssim \|u_0 \|_{L^1(\R^d)}$. We then construct an approximate solution, $u^{\varepsilon}$, from the corresponding equation with smooth coefficients. At the approximate level, the solution is renormalizable which gives $\sup_{t \in [0,T]} \|u_t^{\varepsilon} \|_{L^2(\R^d)} \lesssim \|u_0 \|_{L^2(\R^d)}$. This allows us to fully close the a priori estimates and use a compactness criterion to find a limiting sequence that solves \eqref{eq:main eq}.

In Section \ref{sec:renormalization} we use the technical results of Section \ref{sec:constructing URD} to deduce the equation for $u^2$. This is accomplished through a doubling-of-variables procedure using commutator estimates to evaluate the doubled variables along the diagonal. 

The Appendix contains a technical construction of a family of smoothing maps defined between localized $W^{k,\infty}$-spaces in Appendix A, as well as a compactness criterion from \cite{GLN} in Appendix B.

\noindent
\textbf{Declaration of AI use} \newline
Following the guidelines put forth by the Leiden Declaration of 2026, the authors declare that no artificial intelligence in the form of LLMs or proof assistants has been used to generate any of the text in the accompanying work. The authors wish to declare that the following items were inspired by, but not directly copied from, output generated by ChatGPT-5.6 Sol:
\begin{itemize}
    \item[(i)] the definitions of tail-/head derivatives and their corresponding representation formulae found in Section \ref{sec: casting operators},
    \item[(ii)] the name and description of the equivariance property outlined in Lemma \ref{Lem: equivariance push-forward}.
\end{itemize}

\section{Notation and preliminary results} \label{sec:notation}

\subsection{Miscellaneous notation}
We write $a \lesssim b$ if there exists a positive constant $C >0$ such that $a \leq C b$, and $a \lesssim_\lambda b$ if the constant $C = C(\lambda)$ depends on the parameter $\lambda$. If $a \lesssim b$ and $b\lesssim a$, we write $a \eqsim b$. 

The ball of radius $R$ around $x=0$ in $\mathbb{R}^d$ will always be denoted $B_R$. For vectors $x,y\in \mathbb{R}^d$, we denote their scalar product by $x\cdot y = x^T y$; the transpose of vectors and matrices will be denoted by a superscript $T$.

Given Banach spaces $E_1$ and $E_2$, we denote by $\mathcal{L}(E_1,E_2)$ the Banach space of bounded linear operators from $E_1$ to $E_2$, equipped with the operator norm $\|\cdot\|_{\mathcal{L}(E_1,E_2)}$. We denote by $E^*$ the topological dual of the Banach space $E$.

For smooth scalar functions $f\colon \mathbb{R}^d \to \mathbb{R}$, we denote its partial derivatives by $\partial_j f$ and its gradient by $\nabla f = (\partial_j f)_{j=1}^d$. Moreover, if $F\colon \mathbb{R}^d \to \mathbb{R}^n$ is a vector-valued function, its partial derivatives will be denoted $\partial_j \langle F, e_i\rangle$ with $1\leq j\leq d$, $1\leq i \leq n$. Whenever $n=d$, we will write both $\mathrm{div}\, F = \nabla \cdot F =\sum_{j=1}^d \partial_j \langle F,e_j\rangle $ for the divergence of $F$. We denote by $\mathrm{D}F\colon \mathbb{R}^d \to \mathbb{R}^{n\times d}$, with an upright $\mathrm{D}$, the Fréchet derivative of $F$. We later extend much of the preceding notation to tensor-valued maps of several variables. 

Denote by $C^\infty_c(\mathbb{R}^d;\mathbb{R}^d)$ the set of infinitely differentiable, compactly supported maps $f\colon \mathbb{R}^d \to \mathbb{R}^n$, and equip this set with the test function topology. Let $\mathcal{D}'(\mathbb{R}^d;\mathbb{R}^n)$ denote the space of distributions, here defined as the continuous dual of $C^\infty_c(\mathbb{R}^d;\mathbb{R}^n)$. As is standard, we continue using the same notation of partial/Fréchet derivatives for weak and distributional derivatives. 

For $m\geq 0$, denote by $C^m_{b}(\mathbb{R}^d;\mathbb{R}^n)$ the Banach space of $m$-times continuously differentiable maps $f\colon \mathbb{R}^d \to \mathbb{R}^n$ with bounded derivatives subject to the norm
\begin{equation*}
    \|f\|_{C^m_b} := \sum_{k=0}^m \sup_{x\in \mathbb{R}^d}| \mathrm{D}^{(k)} f(x)|.
\end{equation*}
Let $L^p(\mathbb{R}^d;\mathbb{R}^n)$, for $p\in [1,\infty]$, denote the usual Lebesgue spaces. To differentiate between spatial and temporal integrability, we will always write the temporal integrability first in Bochner--Lebesgue notation: $f\in L^p([0,T]; L^q(\mathbb{R}^d))$ implies $L^p$ in time and $L^q$ in space. We will omit writing the domain or codomain of the functions or vector fields considered whenever they are contextually implicit. If $A\subseteq \mathbb{R}^d$ is a Lebesgue measurable subset, denote by $\|f\|_{L^p(A)}:= \|f \mathbb{1}_A\|_{L^p(\mathbb{R}^d)}$ the restriction of the $L^p$-norm to $A$, where $\mathbb{1}_A$ is the characteristic function of $A$. The Sobolev spaces $W^{k,p}(\mathbb{R}^d;\mathbb{R}^n)$ are defined as usual by
\begin{equation*}
    \|f\|_{W^{k,p}} := \sum_{l=0}^k\|\mathrm{D}^{(l)}f\|_{L^p}.
\end{equation*}

Let $I\subset [0,\infty)$ be a closed interval and consider a map $f\colon I \to E$ for some set $E$; the notation $f_t$ will be used for the evaluation of $f$ at $t\in I$. If the set $E$ is a measurable Banach space $(E,\|\cdot\|_E)$, with Borel $\sigma$-algebra inherited from $\|\cdot\|_E$, and $f\colon I \to E$ is measurable, we denote the set of {\it Borel bounded paths} $\mathcal{B}_b(I;E)$ as the Banach space with norm
\begin{equation*}
    \|f\|_{\mathcal{B}_b(I;E)} := \sup_{t\in I}\|f_t\|_E.
\end{equation*}
Throughout, we will write $L^p(I;E)$ as the Bochner--Lebesgue space of strongly measurable functions $f\colon I \to E$ satisfying 
\begin{equation*}
    \|f\|_{L^p(I; E)} := \left(\int_I \|f_t\|^p_E\,\mathrm{d}t\right)^{1/p} <\infty
\end{equation*}
with the standard modification in the event that $p=\infty$.

We denote the localized version of any Banach space $\mathcal{X}\in \{C^m_b, L^p, W^{k,p}\}$ as above by the subscript $\mathcal{X}_{\mathrm{loc}}$, defined generically by $f \in \mathcal{X}_{\mathrm{loc}}$ if and only if $f\varphi \in \mathcal{X}$ for all $\varphi\in C^\infty_c$. They are all Fréchet spaces with the metric generically given by $f^n \to f$ in $\mathcal{X}_\mathrm{loc}$ if and only if $\|\varphi (f^n-f)\|_\mathcal{X} \to 0$ for all $\varphi \in C^\infty_c$. Moreover, this also extends to $\mathcal{B}_b(I;E_{\mathrm{loc}})$ and $L^p(I;E_{\mathrm{loc}})$, whenever $E$ is Banach, in the obvious way. 

\subsection{Paths of finite $\mathfrak{p}$-variation}

Given a closed interval $I \subset [0,\infty)$, we define the sets \begin{equation*}
    \Delta^{(2)}_I := \{(s,t)\in I^2 : s\leq t\}, \quad \Delta_I^{(3)}  := \{(s,r,t)\in I^3 : s\leq r\leq t\}
\end{equation*}
as the $1$-simplex and $2$-simplex over $I$ respectively, and in the case $I = [0,T]$, with $T$ finite, we simply write $\Delta_T^{(2)} := \Delta_{[0,T]}^{(2)}$ and $\Delta_T^{(3)} := \Delta_{[0,T]}^{(3)}$. Mappings $f\colon \Delta_T^{(2)}\to E$, henceforth called {\it $2$-index maps} in $E$, are maps that evaluate at pairs of points $(s,t)\in \Delta_I^{(2)}$, with the evaluation denoted by $f_{st}$. Similarly, for $3$-index maps $F\colon \Delta^{(3)}_I\to E$, we denote by $F_{srt}$ its evaluation at the triple $(s,r,t)\in \Delta^{(3)}_I$. Given any two maps $x\colon I \to E$ and $f\colon \Delta_I^{(2)}\to E$, we define the {\it first- and second order increment operator} $\delta$ as the (linear) operations yielding $2$- and $3$-index maps respectively, given by
\begin{equation}
    \label{Eqn:delta_first_order_increment}
   \delta x_{st} := x_t-x_s, \quad (s,t)\in \Delta_I^{(2)}, 
\end{equation}
and
\begin{equation}
    \label{Eqn:delta_second_order_increment}
    \delta f_{srt} := f_{st}-f_{sr}-f_{rt}, \quad (s,r,t)\in \Delta^{(3)}_I.
\end{equation}
We also introduce the $(n-1)$-simplex $\Delta^{(n)}_{I}$ defined by
\begin{equation}
\label{Eqn: def of n-1-simplex}
    \Delta^{(n)}_{I} := \{(r_1,r_2,\ldots,r_{n})\in I^n : r_{j-1} \leq r_j \leq r_{j+1} \text{ for } 2\leq j\leq n-1\}
\end{equation}
which will be convenient for rewriting iterated integrals of the form \eqref{eq:URD formally}.

\begin{definition}[$\mathfrak{p}$-variation norm]
    Let $\mathfrak{p}\in(0,+\infty)$ and let $(E,\|\cdot\|_E)$ be a Banach space. Define the set $C^{\mathfrak{p}\mathrm{-var}}_2(I;E)$ of continuous $2$-index maps $f\colon \Delta_I^{(2)}\to E$ of {\it finite $\mathfrak{p}$-variation}, defined as the quantity
    \begin{equation*}
        \llbracket f \rrbracket_{\mathfrak{p},I;E} := \sup_{\pi\in \Pi(I)} \left(\sum_{(t_i,t_{i+1})\in \pi} \|f_{t_it_{i+1}}\|_E^{\mathfrak{p}}\right)^{1/\mathfrak{p}},
    \end{equation*}
    where $\Pi(I)$ is the set of all finite partitions $\pi$ of the interval $I$. Using the first order increment $\delta$ as defined above, we can similarly define the set $C^{\mathfrak{p}\mathrm{-var}}(I;E)$ consisting of continuous maps $x\colon I\to E$ with increment $\delta x \in C^{\mathfrak{p}\mathrm{-var}}_2(I;E)$. In the sequel we will abuse the notation $\llbracket x\rrbracket_{\mathfrak{p}, I;E} := \llbracket \delta x\rrbracket_{\mathfrak{p}, I;E}$, and if the Banach space $E$ is contextually implicit we will drop it from the subscript as well. We will also simply write $\mathcal{C}^{\mathfrak{p}-\mathrm{var}}_2 E$ whenever the interval $I$ is contextually implicit.
\end{definition} 

The use of the fraktur letter $\mathfrak{p}$ to denote the temporal $\mathfrak{p}$-variation avoids notational conflict with $L^p$-spaces. From now on, we will also omit ``$\mathrm{-var}$'' from the superscript $\mathcal{C}^{\mathfrak{p}\mathrm{-var}}$; it will be clear that we always mean temporal variation whenever $\mathfrak{p}$ or multiplies thereof are used.

\begin{remark}    
If $\mathfrak{p}\in [1,\infty)$, then $(C^{\mathfrak{p}}(I;E), \|\cdot\|_{\mathfrak{p}, I;E})$ is a Banach space with the $\mathfrak{p}$-variation as seminorm; for intervals $I = [s,t]$ its norm is thus 
\begin{equation*}
    \|x\|_{\mathfrak{p}, I;E} := \|x_s\|_E + \llbracket x \rrbracket_{\mathfrak{p}, I;E}. 
\end{equation*}
If $\mathfrak{p}\in (0,\infty)$ the seminorm still enjoys lower semicontinuity: if $f^n \to f$ is a uniformly bounded and (pointwise) converging subsequence of $2$-index maps in $C^{\mathfrak{p}}_2(I;E)$, in the sense that $\sup_n \llbracket f^n\rrbracket_{\mathfrak{p},I;E}<\infty$ and $f^n_{st} \to f_{st}$ in $E$ for any pair $(s,t)\in \Delta_I^{(2)}$ and some continuous $2$-index map $f$ over $E$, then we have $f\in C^{\mathfrak{p}}_2(I;E)$ and furthermore
\begin{equation*}
    \llbracket f\rrbracket_{\mathfrak{p}, I; E} \leq \liminf_{n\to \infty} \,\llbracket f^n\rrbracket_{\mathfrak{p}, I; E}.
\end{equation*}
A similar result can be deduced for bounded sequences $x^n\colon I\to E$ in $C^{\mathfrak{p}}(I;E)$ converging pointwise to some continuous map $x$ with values in $E$. 
\end{remark}

A continuous mapping $w\colon \Delta_I^{(2)} \rightarrow \R_+$ will be called a \emph{control} on $I$ provided $w(s,s) = 0$ and superadditivity holds in the sense that 
$$
w(s,r) + w(r,t)  \leq w(s,t) , \quad \forall (s,r,t) \in \Delta_I^{(3)}.
$$

The following characterization of the $\mathfrak{p}$-variation semi-norms will be useful in the sequel. The proof is contained in \cite[Propositions 5.8 and 5.10]{FV2010}.

\begin{lemma}
    A map $g$ belongs to $C^{\mathfrak{p}}_2(I;E)$ if and only if there exists a control $w$ such that 
    $$
    \|g_{st}\|_E \leq w(s,t)^{1/\mathfrak{p}}, \quad \text{for all } (s,t) \in \Delta_I^{(2)}.
    $$
    In this case, the quantity $ w_g(s,t) := \llbracket g\rrbracket_{\mathfrak{p},[s,t];E}^{\mathfrak{p}}$ is an optimal control in the sense that
    \begin{equation} \label{eq:p variation control bound}
     \llbracket g \rrbracket_{\mathfrak{p}, [s,t]; E} \leq w(s,t) , \quad \text{for all } (s,t) \in \Delta_I^{(2)}.
    \end{equation}
\end{lemma}
The following remark details some essential properties of continuous controls.
\begin{remark}[\protect{\cite[Remark 2.6]{GLN}, \cite{FV2010}}]
\label{Rem: control properties}
    Let $w_1$ and $w_2$ be controls. Then the linear combination $a w_1 + bw_2$ for any positive $a,b>0$, as well as the product $w_1w_2$, are also controls. If $\gamma \colon \Delta^{(2)}_T \to \mathbb{R}^+$ is increasing in the sense that $\gamma([s,t]) \leq \gamma ([s',t'])$ for subsets $[s,t]\subseteq [s',t']$, then $\gamma w$ is yet again a control if $w$ is a control. Crucial for our purposes, if $\alpha,\beta>0$, then there exists another control $w_3$ satisfying the exponent-additivity
    \begin{equation}
    \label{Eqn: exponent additivity}
        w_1(s,t)^\alpha\, w_2(s,t)^\beta = w_3(s,t)^{\alpha + \beta},\quad \text{ for all } (s,t)\in \Delta^{(2)}_T.
    \end{equation}
    Lastly, if $w$ is a control, then so is $w^\alpha$ provided $\alpha>1$. 
\end{remark}

We recall the \emph{sewing lemma} from \cite{gubinelli2004}, which we rewrite to suit the a priori estimates appearing in Section \ref{sec:a priori}.

\begin{lemma}[Sewing lemma] 
\label{lemma:sewing}
    Let $E$ be a Banach space, and let $g\colon \Delta_I^{(2)} \rightarrow E$ be a map such that there are two continuous controls $w_{g}$ and $w_{\delta g}$ satisfying 
    $$
    \|g_{st}\|_E \leq w_g(s,t)^{1 + \varepsilon}, \quad \forall \, (s,t) \in \Delta_I^{(2)}
    $$
    and 
    $$
    \|\delta g_{srt}\|_E \leq w_{\delta g}(s,t)^{1 + \varepsilon} , \quad \forall \, (s,r,t) \in \Delta_I^{(3)}
    $$
    for some $\varepsilon > 0$. Then there exists a constant $C = C_{\varepsilon}$ such that 
    $$
    \|g_{st}\|_E \leq C w_{\delta g}(s,t)^{1 + \varepsilon}, \quad \forall \, (s,t) \in \Delta_I^{(2)}.
    $$
\end{lemma}

We shall also need a rough version of Gronwall's lemma. The statement below is taken from \cite{HLN2021} which was based on \cite{DGHT2019}. For completeness, we include a simplified proof (which is an adaptation of the proof in \cite[Lemma 8.10]{friz2021rough} to the $\mathfrak{p}$-variation setting). 

\begin{lemma} \label{lemma:rough gronwall}
Assume $G\colon [0,T] \rightarrow \R_+$ is such that for every $(s,t) \in \Delta_T^{(2)}$ satisfying $w(s,t) \leq L$ we have the bound 
$$
 \delta G_{st}  \leq \| G\|_{\mathcal{B}_b([s,t];\R)} \,w(s,t)^{1/\mathfrak{p}}.
$$
Then, with $\tilde{L} := L \wedge 2^\mathfrak{-p}$, we have
\begin{equation} \label{eq:gronwall bound}
\sup_{t \in [0,T]}  G_t \leq C e^{ w(0,T)/\tilde{L}}  G_0
\end{equation}
for some universal constant $C$.
\end{lemma}
\begin{proof}
Note that if $(s,t) \in \Delta_T^{(2)}$ is such that $w(s,t) \leq \tilde{L}$ and $u \in [s,t]$ we have
$$
G_u \leq G_s + \delta G_{su}\leq G_s + \| G\|_{\mathcal{B}_b([s,t];\R)} w(s,t)^{1/\mathfrak{p}} .
$$
Taking the supremum over $u \in [s,t]$ on the above left hand side and rearranging yields
$$
(1- w(s,t)^{1/\mathfrak{p}}) \| G\|_{\mathcal{B}_b([s,t];\R)}   \leq  G_s
$$
Using $w(s,t)^{1/\mathfrak{p}} \leq \frac12$ together with the inequality $e^{-2x} \leq 1-x$ for $x \in [0,\frac12]$, we find the bound
\begin{equation} \label{eq:rough gronwall partial estimate}
e^{-2w(s,t)^{1/\mathfrak{p}}} \|G\|_{\mathcal{B}_b([s,t];\R)}   \leq  G_s .    
\end{equation}
Construct now the partition $\pi = \{ \tau_i \}_{i=0}^{N}$ by letting 
$$
\tau_0 := 0, \quad \tau_{i+1} := \inf \{ s > \tau_i : w(\tau_i,s)  \geq \tilde{L} \} \wedge T
$$
so that $w(\tau_i , \tau_{i+1}) = \tilde{L}$ for $i =0,1,\dots, N-2$ and $w(\tau_{N-1},\tau_N) \leq \tilde{L}$. 

By the definition of $\tilde{L}$ we use \eqref{eq:rough gronwall partial estimate} to find
$$
\|G\|_{\mathcal{B}_b([\tau_i, \tau_{i+1}];\R)} \leq e^{2 \tilde{L}^{1/\mathfrak{p}}}\,G_{\tau_i} \leq e\,G_{\tau_i}
$$
and by induction we get 
$$
\|G\|_{\mathcal{B}_b([\tau_i, \tau_{i+1}];\R)} \leq  e^{i+1}G_{0} \leq e^{N+1}G_{0}.
$$
Using the superadditivity of the control $w$ we find that $N$ can be bounded using
$$
(N-1)\tilde{L} = \sum_{i=0}^{N-2} w(\tau_i, \tau_{i+1}) \leq w(0,T) ,
$$
implying that
$$
N \leq  \frac{w(0,T)}{\tilde{L}} + 1.
$$
The result now follows. 
\end{proof}

\subsection{Tensor algebraic aspects}

\subsubsection{Combinatorial preliminaries}
    A {\it word} $w$ of length $ p \geq 1$ over the alphabet $\{1,\ldots, n\}$ is a tuple $w= (i_1,\ldots, i_p) \in \{1,\ldots,n\}^p$. We denote its length by $|w| = p$. A word is said to be {\it binary} if $n=2$. The concatenation of two words $u= (i_1,\dots, i_p)$ and $v = (i_{p+1},\ldots,i_{p+q})$ is the word $uv = (i_1,\ldots, i_{p+q})$. As such, the tuple comprising the word $w= (i_1,\ldots,i_p)$ can be written as a concatenation of the length-1 words $i_1,i_2,\ldots, i_p$; in light of this, we will write $w = i_1\ldots i_p$. We will encounter words as strings of dummy indices in sums over the alphabets $\{1,\ldots,d\}$, and also in the upcoming description of {\it doubled tensor fields} when considering so-called {\it binary projections}, which will be introduced shortly.
    
    The action of a permutation $\sigma\in S_p$ can be uniquely described by its action on words of length $p$: 
    \begin{equation*}
        \sigma \cdot i_1\ldots i_p = i_{\sigma(1)}\ldots i_{\sigma(p)}
    \end{equation*}
    Given $p,q \geq 0$, we define a $(p,q)$-shuffle $\sigma\in S_{p,q}$ by the permutation $\sigma\in S_{p+q}$ that interleaves indices from the sets $\{1,\ldots, p\}$ and $\{p+1,\ldots,p+q\}$ while keeping the original internal ordering of both sets; explicitly, this writes
    \begin{equation*}
        \sigma(1) < \sigma(2) <\cdots < \sigma(p) \quad \text{and} \quad \sigma(p+1)<\sigma(p+2)<\cdots < \sigma(p+q).
    \end{equation*}

    \subsubsection{Tensors, tensor fields, tensor products}
    Write $(\mathbb{R}^d)^n := \mathbb{R}^d\times\cdots \times\mathbb{R}^d$ as the $n$-fold Cartesian product of $\mathbb{R}^d$ and $(\mathbb{R}^d)^{\otimes k}$ as the $k$-fold (algebraic) tensor product space of $\mathbb{R}^d$. A generic point $\mathbf{x}$ in $(\mathbb{R}^d)^n$ will be written $\mathbf{x} = (x_1,\dots, x_n)$; not to be confused with the coordinates of the individual vectors $x_j$, which are written $x^\alpha_j$ for $1\leq \alpha\leq d$, $1\leq j \leq n$. For the vector space of tensors $(\mathbb{R}^d)^{\otimes k}$, we will write $e_J = e_{j_1\dots j_k} = e_{j_1}\otimes \cdots \otimes e_{j_n}$ as generic basis elements, where each $e_j$ is a canonical basis vector of $\mathbb{R}^d$. As such, the dual $e_j^*\colon \mathbb{R}^d \to \mathbb{R}$ will be represented by the angled brackets $\langle \,\cdot\,, e_j\rangle $, which extends naturally to the tensor duals $\langle \,\cdot\,,e_{j_1\dots j_k}\rangle \colon (\mathbb{R}^d)^{\otimes k}\to \mathbb{R}$. There will be no confusion with other dual pairings encountered throughout the article (like, for instance, the $L^2$-dual pairing), since we will consistently write $e_{j_1\ldots j_n}$, or something similar, in the second slot when considering the tensor duals. 

    In addition to the tensor products of Euclidean basis elements, we will also define the tensor products of functions, distributions, and operators, all using the same notation given that there is no risk of confusion. The tensor product of scalar functions $f,g\colon \mathbb{R}^d \to \mathbb{R}^d$ will be defined as $f\otimes g \colon \mathbb{R}^{2d}\to \mathbb{R}; (f\otimes g)(x,y) = f(x)g(y)$; see also Remark \ref{Rem: bar-convention}. We can extend the latter tensor product to distributions $f,g\in \mathcal{D}'(\mathbb{R}^d)$ by duality, thus $f\otimes g \in \mathcal{D}'(\mathbb{R}^{2d})$. Moreover, extending the tensor product to operators, if $E_1, E_2$ are abstract Banach spaces, then $v\otimes w \in \mathcal{L}(E^*_2,E_1)$ for $v\in E_1$ and $w\in E_2$ is given by $(v\otimes w)(w^*) = v\langle w,w^*\rangle_{E_2,E_2^*}$ for any $w^*\in E_2^*$; abusing this notation, we interpret $v\otimes w$ as a bilinear operator on $E_1^*\times E_2^*$ by declaring $(v\otimes w)(v^*,w^*) := \langle v,v^* \rangle \langle w,w^* \rangle $, with appropriate dual pairings. If $A_1, A_2$ are linear operators with $A_i\in \mathcal{L}(E_i,F_i)$ for $i=1,2$, then the tensor product $A_1\otimes A_2 \in \mathcal{L}(E_1\otimes E_2; F_1\otimes F_2)$ is defined as the (unique) linear extension of the mapping 
    \begin{equation*}
        (A_1\otimes A_2)(v\otimes w) = (A_1v)\otimes (A_2 w),\quad \text{ for all } (v,w)\in E_1\times E_2.
    \end{equation*}

    Throughout the article, we will generically call a (contravariant) tensor map of the form
    \begin{equation}
    \label{Eqn: tensor field decl}
        F\colon (\mathbb{R}^d)^n \to (\mathbb{R}^d)^{\otimes k},
    \end{equation}
    a {\it tensor field of degree $k$ over $\mathbb{R}^d$}, whenever the number of variables $n$ is contextually implicit. For our purposes we will always consider tensor fields with $k\leq n$. Similarly, if $k=1$ we will call $F$ as in \eqref{Eqn: tensor field decl} a {\it vector field over $\mathbb{R}^d$}. We say that a tensor field of degree $n$ is {\it decomposable} if there are vector fields $f_1,\ldots, f_n$ such that $F(x_1,\ldots, x_n) = f_1(x_1)\otimes \cdots \otimes f_n(x_n)$. Given two tensor fields $F$ and $G$ of equal degree $k$, we define their contraction\footnote{The operation $F: G$ is usually reserved for the Frobenius product of matrices; we prefer its usage here as the contraction of tensor fields to avoid over-using $\langle \,\cdot\,,\,\cdot\,\rangle$ later on.} $F:G$ as the scalar quantity
    \begin{equation*}
        F: G = \sum_{j_1,\ldots,j_k = 1}^d\langle F,e_{j_1\ldots j_{k}}\rangle \langle G, e_{j_1\ldots j_k}\rangle.
    \end{equation*}

    Let $F\colon (\mathbb{R}^d)^n \to (\mathbb{R}^d)^{\otimes k}$ be a tensor field. Denote the gradient derivative in the $j$-th variable by $\nabla^{(j)} F$, explicitly defined as
    \begin{equation*}
        \nabla^{(j)} F(x_1,\ldots, x_n) = \left(\frac{\partial F}{\partial x_j^1},\ldots, \frac{\partial F}{\partial x^d_j} \right) \in (\mathbb{R}^d)^{\otimes k}\otimes \mathbb{R}^d.
    \end{equation*}
    The components of $\nabla^{(j)} F$ are denoted by $\partial^{(j)}_\alpha F$, with $1\leq \alpha \leq d$.
    Extending this, we obtain in a similar way the higher-order derivatives $(\nabla^{(j)})^m F \in (\mathbb{R}^d)^{\otimes k}\otimes (\mathbb{R}^d)^{\otimes m}$. For multi-indices $\mathbf{m} = (m_1,\ldots, m_n)$ we write
    \begin{equation*}
        \nabla^{\mathbf{m}} F := (\nabla^{(1)})^{m_1}\cdots (\nabla^{(n)})^{m_n} F.
    \end{equation*}
    Let $|\cdot |$ denote the Euclidean tensor norm inherited from $\mathbb{R}^d$, and note that
    \begin{equation*}
        |\nabla^\mathbf{m} F(x_1,\ldots, x_n)|\colon (\mathbb{R}^d)^{n} \to \mathbb{R}_{\geq 0}    
    \end{equation*}
    becomes a scalar quantity. With this notation in mind, we can define the {\it Sobolev norm of tensor fields} over $\mathbb{R}^d$ by
    \begin{equation}
        \|F\|_{W^{l,p}((\mathbb{R}^d)^n; (\mathbb{R}^d)^{\otimes k})} := \sum_{|\mathbf{m}| \leq l} \|\nabla^\mathbf{m} F\|_{L^p((\mathbb{R}^{d})^{n})},\quad |\mathbf{m}| := \sum_{j=1}^n m_j,\quad 1\leq p \leq \infty,
    \end{equation}
    where the absolute value in the $L^p$-norm is now replaced by the Euclidean tensor norm. Whenever we wish to explicitly highlight the Sobolev regularity of some tensor field $F$ with degree $k$, we will write $F\in W^{l,p}((\mathbb{R}^d)^n; (\mathbb{R}^d)^{\otimes k})$. Otherwise, we will simply write $\|F\|_{W^{l,p}(\mathbb{R}^d)}$ instead of $\|F\|_{W^{l,p}((\mathbb{R}^d)^n; (\mathbb{R}^d)^{\otimes k})}$ and $F\in W^{l,p}(\mathbb{R}^d)$ for tensor fields $F$ over $\mathbb{R}^d$ whenever both the arity and tensor degree are implicit.
    \begin{remark}
        Our jargon ``tensor field $F(\mathbf{x})$ over $\mathbb{R}^d$'' might at this point seem strange given that its arguments $\mathbf{x}$ are taken from $(\mathbb{R}^d)^n$. Later in this article, we will encounter tensor fields $F\colon (\mathbb{R}^{2d})^n \to (\mathbb{R}^{2d})^{\otimes n}$ that are instead over $\mathbb{R}^{2d}$, with points $\mathbf{z} = (z_1,\ldots,z_n) \in (\mathbb{R}^{2d})^n$. The main purpose of our phrasing is therefore to later highlight this doubling of dimension. We make the convention that tensor fields over $\mathbb{R}^{2d}$ generically have arguments labelled by $z$ instead of $x$. 
    \end{remark}

    Our formal computations require the concept of a pointwise evaluation of tensor fields based on $L^\infty$-spaces. Given two tensor fields $F,G \colon (\mathbb{R}^d)^{n}\to \mathbb{R}^{\otimes k}$ of equal degree with components in $L^\infty(\mathbb{R}^d)$, we declare that
    \begin{equation*}
        F = G  \quad \text{ if and only if }\quad  F(x_1,\ldots, x_n) = G(x_1,\ldots, x_n)
    \end{equation*}
    holds for Lebesgue-almost all $(x_1,\ldots,x_n)\in (\mathbb{R}^d)^n$. We also introduce the {\it diagonal trace operator}, whose action reads
    \begin{equation}
    \label{Eqn: diagonal trace def}
        \mathrm{Tr}_x F = \mathrm{Tr}_{x} F(x_1,\ldots,x_n) = \mathrm{Tr}_{x_1=\cdots=x_n=x}F := F(x,x,\ldots,x)
    \end{equation}
    for almost all $x\in \mathbb{R}^d$. In order for the diagonal trace to be well-defined at every point $x\in \mathbb{R}^d$ we would need additional regularity of the tensor field $F$, however, we only ever need \eqref{Eqn: diagonal trace def} to hold for almost every $x\in\mathbb{R}^d$. Since components of such tensor fields are locally integrable on $(\mathbb{R}^d)^n$, it follows that the diagonal trace \eqref{Eqn: diagonal trace def} can be realized by the Lebesgue differentiation theorem for the joint tuple $(x,x,\ldots, x)$, and almost every such tuple in $(\mathbb{R}^d)^n$. 

    With slight abuse, and in light of our previous notations, we may sometimes write
    \begin{equation*}
        \|\mathrm{Tr}_x F\|_{W^{k,\infty}(\mathbb{R}^d)} = \|F\|_{W^{k,\infty}(\mathbb{R}^d)}.
    \end{equation*}

    \begin{remark}
    \label{Rem: trace remark}
    Later, we will encounter recursive formulae that use ``partial traces'' as a formal trick: we denote by
    \begin{equation*}
        \mathrm{Tr}_{\bar{x} = x} \,F(\,\cdot\,, \bar{x})
    \end{equation*}
    the trace that evaluates the last variable $\bar{x} =x$, leaving the remaining $n-1$ ``free'' variables unchanged. We avoid having to rigorously justify this notation since the final expression of the recursive formulae, when fully developed, will only include the ``full'' diagonal trace as in \eqref{Eqn: diagonal trace def}.

    \noindent
    Instances where the above trick will be used can also feature tensor fields that potentially have differing Sobolev regularities in different arguments. As an example, take $n=2$ and consider the scalar function $f= f(x,y)$ where $f$ is $W^{1,\infty}$ in $y$, but is only essentially bounded in $x$. Then we may choose a representative of $f$ (not relabelled) that is continuous in $y$, and by Lebesgue's differentiation theorem, one can make sense of pointwise evaluation in $x$ for the Lebesgue points of $f(\cdot,y)$, for every $y$ simultaneously, and hence for Lebesgue-almost all $x\in \mathbb{R}^d$. In this scenario, we will again interpret \eqref{Eqn: diagonal trace def} to hold for almost all points in the variables for which the pointwise evaluation may fail.
    \end{remark}

    \subsubsection{Projections of tensor fields}
    
    Given a tensor field $F\colon (\mathbb{R}^d)^n\to (\mathbb{R}^d)^{\otimes k}$ we define $\Pi_j$ as the {\it projection operator} of the form 
    \begin{equation*}
        \Pi_j F(\mathbf{x}) = \sum_{i_{1},\dots, i_{k-1} = 1}^d\langle F(\mathbf{x}), e_{i_{1}\dots i_{k-1} j}\rangle e_{i_{1}\dots i_{k-1}}
    \end{equation*}
    or, equivalently, $\Pi_j = \mathrm{I}^{\otimes (k-1)}\otimes e^*_j$; the final component of $F$ is fixed with index $1\leq j \leq d$. 
    Its cyclical transpose $\Pi^\intercal_j=  e_j^*\otimes \mathrm{I}^{\otimes (k-1)}$ is a projection operator of the form 
    \begin{equation*}
        \Pi^\intercal _j F(\mathbf{x}) = \sum_{i_2,\dots, i_{k}= 1}^d \langle F(\mathbf{x}), e_{j i_2\dots i_{k}}\rangle e_{i_2\dots i_{k}}
    \end{equation*}
    the first component of $F$ is fixed with index $j$. Both of these operators linearly project degree $k$ tensor fields over $\mathbb{R}^d$ to degree $k-1$ tensor fields. We can keep composing these projections, and write $\Pi_{j_l\dots j_1} = \Pi_{j_l}\circ \Pi_{j_{l-1}\dots j_1}$ as the operator which fixes the final $l$-many tensor components with indices $j_l,\ldots, j_1$; similarly for $\Pi^\intercal_{j_1\dots j_l} = \Pi^\intercal_{j_1}\circ \Pi^\intercal_{j_2\dots j_l}$. Finally, if $l=k$ is maximal, then $\Pi_{j_k\cdots j_1} F = \langle F, e_{j_k \dots j_1}\rangle$ and $\Pi^\intercal_{j_1\dots j_k} = \langle F, e_{j_k\dots j_1}\rangle$. 

    \vspace{0.3cm}
    Consider now tensor fields $F\colon (\mathbb{R}^{2d})^{n} \to (\mathbb{R}^{2d})^{\otimes k}$ over $\mathbb{R}^{2d}$, where we now emphasize the doubled dimension $2d$. We understand and define the projections $\Pi_j, \Pi^\intercal_j$ in the same way as before, except now with running indices $1\leq i_l \leq 2d$ in the sums. We are also interested in bisecting the tensor fields into blocks of ``halved'' tensor fields. To illustrate what we mean, if $V$ is a vector field over $\mathbb{R}^{2d}$, we write $V = (\pi_1 V,\pi_2V)^T$, where $\pi_1$ and $\pi_2$ extract the first and second half of the vector respectively. Here, the projections $\pi_1V$ and $\pi_2 V$ are vector fields over $\mathbb{R}^d$, as desired. 
    
    We generalize the binary representation of vector fields to that of tensor fields of degree $k$ by introducing the {\it binary projections} $\pi_{i_1\dots i_n}$, where $i_1\dots i_n$ are binary indices with $i_l \in \{1,2\}$, formally defined by
    \begin{equation}
    \label{Eqn: small pi proj def}
        \pi_{i_1\cdots i_{k}} F := \sum_{\substack{j_1,\dots, j_{k} \\ [j_l] = i_l}} \Pi_{j_1\dots j_{k}} F \,e_{\tau(j_1\dots j_{k})}
    \end{equation}
    where the running indices $j_1,\dots, j_{k}$ in the sum have $1\leq j_l \leq 2d$. Here we have introduced the notation $\tau(J)$ as the transformed multiindex of $J = (j_1,\dots, j_{k})$ where each element is reduced modulo $d$:
    \begin{equation}
        \tau(J) := (\tau(j_1),\dots, \tau(j_{k})),\quad \tau(j_l) := j_l\hspace{-1.5mm}\pmod{d},
    \end{equation}
    additionally, $[j_l]$ in the sum \eqref{Eqn: small pi proj def} denotes the binary representations of indices $j_l$:
    \begin{equation}
        [j_l] := \begin{cases}
    1   & \text{ if } 1\leq j_l \leq d, \\
    2   & \text{ if } d+1 \leq j_l \leq 2d.
    \end{cases}
    \end{equation}
    The exact definition of binary projections as above is included for completeness; we will have no use for the notation $\tau(\cdot)$ in what follows. Informally, the main point is that binary projections take a degree $k$ tensor field over $\mathbb{R}^{2d}$ and output a degree $k$ tensor field over $\mathbb{R}^d$. This relationship becomes crucial when considering tensor fields that are {\it doubled} in the sense that the arguments of the tensor field also possess a binary structure -- this is made precise in the following definition. 
    
    \begin{definition}[Doubled tensor field]
    \label{Def: doubled tensor functions}
        A tensor field $\bar{F}\colon (\mathbb{R}^{2d})^n\to (\mathbb{R}^{2d})^{\otimes n}$ of degree $n$ over $\mathbb{R}^{2d}$ is said to be {\it doubled} if there exists a degree $n$ tensor field $F\colon (\mathbb{R}^d)^n \to (\mathbb{R}^d)^{\otimes n}$ over $\mathbb{R}^d$ such that 
        \begin{equation}
        \label{Eqn: doubled tensor field def}
            \pi_{i_1\dots i_{n}} \bar{F}(z_1,\ldots,z_n) = F(\pi_{i_1}z_1,\ldots, \pi_{i_n}z_n)
        \end{equation}
        for all binary indices $i_1\dots i_n$. In this case, we say that $\bar{F}$ is {\it doubled from} $F$. Compactly, we may equivalently write \eqref{Eqn: doubled tensor field def} as $\pi_{i_1\dots i_n} \bar{F}(\mathbf{z}) = F(\pi_{i_1\dots i_n}(\mathbf{z}))$.
    \end{definition}
    \begin{remark}[Bar convention]
    \label{Rem: bar-convention}
        Throughout the entire article, we make the convention that tensor products, operators, tensor fields, etc.~that are denoted with a bar atop their symbol should be interpreted heuristically as follows: the represented object is {\it doubled} in some sense. For example, barred tensor products $V_1\ftensor V_2$ indicate that if the vector fields $V_1$ and $V_2$ are both defined over $\mathbb{R}^d$, then $V_1\ftensor V_2$ is a vector field over $\mathbb{R}^{2d}$. In the scalar function or distribution case, $(f\ftensor g)(z) = f(x)g(y)$ has doubled its variables in terms of dimension with $z = (x,y)^T \in \mathbb{R}^{2d}$. 
    \end{remark}

\section{Unbounded rough drivers} \label{sec:URD}

\subsection{Definitions and setup}

To give rigorous meaning to \eqref{eq:main eq}, let us denote $\mu_t := \int_0^t b_r \cdot \nabla u_r\, \mathrm{d}r$ so that we may write \eqref{eq:main eq} as
$$
\partial_tu = \partial_t \mu + \dot{\mathrm{X}} \cdot \nabla u.
$$
Using the iteration procedure as in the Introduction, we find 
\begin{equation} \label{eq:expansion with drift}
    \delta u_{st}  = \delta \mu_{st} + \sum_{n=1}^N \mathrm{A}_{st}^{n} u_s + u_{st}^{\natural}
\end{equation}
where the operators $\mathrm{A}_{st}^{n}$ are formally defined by \eqref{eq:URD formally}. Since the solution, $u$, is in general in $L^2(\R^d)$, 
the terms $\mathrm{A}_{st}^{n} u_s$ are spatial distributions with time-regularity $\mathfrak{p}/n$-variation. 

The remainder term $u^{\natural}_{st}$ should be thought of as being implicitly defined by \eqref{eq:expansion with drift}. By the formal expression \eqref{eq:remainder formally}, we expect it to take $N+1$ derivatives of the solution, but its temporal variation is $\mathfrak{p}/(N+1) < 1$. 


To keep track of the number of spatial derivatives and the corresponding time regularity, the unbounded rough drivers framework crucially uses the notion of \emph{scale of spaces} to analyse the expansion \eqref{eq:expansion with drift}.

\begin{definition} 
\label{def:scale of spaces}
A family of Banach spaces $(E_l, \|\cdot\|_{E_l})_{l\geq 0}$ will be referred to as a {\it scale of spaces} provided we have continuous embeddings $E_{l+1} \subset E_l$ for all $l\geq 0$. We write $E_{-l} := E_l^*$ for the topological dual of $E_l$. 
\end{definition}

An important first example of a scale of spaces are the Sobolev spaces $(W^{k,p}(\R^d))_{k \geq 0}$ for any $p \in [1,\infty]$. In the sequel, we will build localizations of the scales $(W^{k,\infty})_{k\geq 0}$ that suit our purposes in the upcoming analysis.

\begin{notation}
    Throughout the rest of this section, whenever there is no room for confusion, we denote by $|\cdot|_{l}$ the norm on $E_l$, similarly for $E_{-l}$. Moreover, we also shorten
    \begin{equation*}
        \|u\|_{[s,t],-l} = \|u\|_{\mathcal{B}_b([s,t];E_{-l})}
    \end{equation*}
    in the interest of brevity. 
\end{notation}

With the concept of a scale of spaces at hand, we now give the definition of an unbounded rough driver.
\begin{definition} \label{def:URD}
    Let $\mathfrak{p} \geq 1$ be given and $N = \lfloor\mathfrak{p}\rfloor$ its integer value. A family of operator-valued maps $\mathbf{A} = \{\mathrm{A}^{n}\}_{n=1}^N$ consisting of components
    \begin{equation} \label{eq:URD takes values in spaces of linear mappings}
        \mathrm{A}^{n}\colon \Delta_T^{(2)} \longrightarrow  \bigcap_{k=0}^{N+1-n} \mathcal{L}(E_{-k},E_{-n-k})
    \end{equation}
    is called an \emph{unbounded rough driver} on the scale $(E_l)_{0 \leq l \leq N+1}$, provided that
    \begin{itemize}
        \item[(i)] there exists a control $w_{\mathbf{A}}$ such that for all $(s,t)\in \Delta^{(2)}_T$
     \begin{equation} \label{eq:URD bound}
        \|\mathrm{A}_{st}^{n}\|_{\mathcal{L}(E_{-k},E_{-n-k})} \leq w_{\mathbf{A}}(s,t)^{n/\mathfrak{p}}, \quad \text{ for } 0\leq k \leq N+1-n,
    \end{equation}
    \item[(ii)] and \emph{Chen's relation} holds; for any $(s,r,t)\in \Delta^{(3)}_T$,
    \begin{equation} \label{eq:URD chen}
        \mathrm{A}_{st}^{n} = \sum_{l=0}^n \mathrm{A}_{rt}^{l} \circ \mathrm{A}_{sr}^{n-l}.
    \end{equation}
    \end{itemize}
    In \eqref{eq:URD chen} we have used the convenient notation $\mathrm{A}^{0}_{st} := \mathrm{Id}$.
\end{definition}

Using the definition of an unbounded rough driver we can define what we mean by a solution of \eqref{eq:expansion with drift}. In the definition below, we keep $\mu\colon [0,T] \rightarrow E_{-1}$ as an abstract bounded variation path.

\begin{definition}
    Let $\mathbf{A}$ be an unbounded rough driver on a scale of spaces $(E_l)_{l \geq 0}$ and $\mu \colon [0,T] \rightarrow E_{-1}$ be continuous and of bounded variation. A bounded Borel path 
    \begin{equation*}
        u \colon [0,T] \rightarrow E_{-0}
    \end{equation*}
    is said to be a {\it solution} of the rough partial differential equation
    \begin{equation} \label{eq:URD diff}
        \mathrm{d}u_t = \mathrm{d} \mu_t  + \mathbf{A}_{\mathrm{d}t}u_t
    \end{equation}
    provided the mapping $u^{\natural}\colon \Delta_T^{(2)}  \rightarrow E_{-(N+1)}$, often called the remainder term, defined by
    $$
    u_{st}^{\natural} := \delta u_{st} - \delta \mu_{st} - \sum_{n=1}^N \mathrm{A}_{st}^{n}u_s
    $$
    is of $\mathfrak{p}/(N+1)$-variation. That is to say, there exists a control $w_{\natural}$ such that 
    $$
    |u_{st}^{\natural}|_{-(N+1)} \leq w_{\natural}(s,t)^{(N+1)/\mathfrak{p}}.
    $$
\end{definition}

Crucial to obtaining a priori estimates is a family of smoothing operators as defined below, which allows us to interpolate spatial and temporal regularity in the expansion equations. 

\begin{definition} \label{def:smoothing operators}
    Given a scale of spaces $(E_l,|\cdot|_l)_{0\leq l \leq N+1}$, a family of operators $(J^{\eta})_{\eta \in (0,1]}$ defined on $E_l$ for all $0\leq l\leq N+1$ is called a {\it smoothing operator} provided
    \begin{align}
        \label{eq:smoothing conditions}
        \begin{split}
            |J^{\eta}\phi |_l &\lesssim \eta^{-m}|\phi|_{l-m}, \quad \,\text{ for } m\leq l \leq N+1,\\
            |(I-J^{\eta})\phi |_l &\lesssim \eta^{m}|\phi|_{l+m} \qquad \text{ for } 0 \leq l \leq N+1-m.
        \end{split}
    \end{align}
    Moreover, suspended proportionality constants can be made independent of $\eta$ and auxiliary variables associated with the scale $(E_l)_l$. 
\end{definition}

For technical reasons, we shall need scales of spaces which are localized in a suitable sense.
Given a function $\rho\colon \R^d \rightarrow \R_+$ we define the scale of spaces $(E_l^{\rho})_{l\geq 0}$ through
\begin{equation}
    \label{Eqn: E space def}
    E_l^{\rho} := E_l^{\rho}(\R^d) := \{ u \in W^{l,\infty}(\R^d) : \, \textrm{for almost all }x ,\, \rho(x) \geq 1 \Rightarrow u(x) = 0 \}.    
\end{equation}
We make the following assumptions on the localization function $\rho$:
\begin{assumption} \label{assumption:psi}
Assume that $\rho \colon \R^d \rightarrow \R_+$ is a smooth function such that 
\begin{itemize}
    \item[(i)] the sublevel set $\{x \in \R^d : \rho(x) \leq 2\}$ is bounded in $\R^d$,
    \item[(ii)] $\rho$ is convex on $\mathbb{R}^d$, and
    \item[(iii)] there exists a constant $C_{\rho} > 0$ such that $|\nabla \rho(x)| \geq C_{\rho}$ uniformly in $x$ such that $\rho(x) \geq \frac12$. 
\end{itemize}
\end{assumption}
We provide the following result concerning the existence of a smoothing on the scale $E^\rho$, and refer to Appendix \ref{App: Construction of smoothing operators} for a proof.
\begin{proposition} 
\label{prop:smoothing}
    Let $\rho$ satisfy Assumption \ref{assumption:psi} and define the scale $(E_l^{\rho})_{l \geq 0}$ as in \eqref{Eqn: E space def}. Then there exists a family of smoothing operators on $(E_l^{\rho})_{l \geq 0}$.
\end{proposition}

We introduce two scales of spaces of great importance to the analysis in Section \ref{sec:renormalization}, and note that both of these spaces coincide with their equivalents defined in \cite[Definition 4.4]{GLN}. Moreover, Proposition \ref{prop:smoothing} implies the existence of a smoothing for both spaces. 
\begin{definition}[The scales $\mathcal{E}$ and $\mathcal{F}$]
\label{Def: the scales E and F}
    Let $0\leq l \leq N+1$. Fixing $R\geq 1$, we define the localization functions $\rho_R\colon \mathbb{R}^d\to \mathbb{R}$ and $\zeta_R\colon \mathbb{R}^d\times \mathbb{R}^d \to \mathbb{R}$ by
    \begin{equation}
        \rho_R(x) := \frac{|x|^2}{R^2},\qquad  \zeta_R(x,y) := \frac{|x_+|^2}{R^2} + |x_-|^2,
    \end{equation}
    where $x_+$ and $x_-$ denote the parallel-/transverse coordinates $x_\pm := (x\pm y)/2$. It is easily checked that both of these functions satisfy Assumption \ref{assumption:psi}.
    Based on \eqref{Eqn: E space def}, we define
    \begin{equation}
    \label{Eqn:cal_F_def}
        \mathcal{F}_{l,R} = \mathcal{F}_{l,R}(\mathbb{R}^d) := E_l^{\rho_R}(\R^d),
    \end{equation}
    and moreover
    \begin{equation}
    \label{Eqn:cal_E_def}
        \mathcal{E}_{l,R} = \mathcal{E}_{l,R}(\mathbb{R}^{2d}) := E^{\zeta_R}_l(\mathbb{R}^{2d}).
    \end{equation}
    We equip both scales with the subspace topologies of $W^{l,\infty}(\mathbb{R}^{d})$ and $W^{l,\infty}(\mathbb{R}^{2d})$ respectively. 
\end{definition}

Note that the spaces $\mathcal{F}_{l,R}$ and $\mathcal{E}_{l,R}$ are crudely related by the following: if $R\geq 1$, one relates their localization functions by the estimate
\begin{equation*}
    \frac12(|x|^2 + |y|^2) = |x_+|^2 + |x_-|^2 \leq R^2\left(\frac{|x_+|^2}{R^2} + |x_-|^2\right) = R^2 \zeta_R(x,y),
\end{equation*}
whence $\mathcal{E}_{l,R}(\mathbb{R}^{2d}) \subset \mathcal{F}_{l,\sqrt{2}R}(\mathbb{R}^{2d})$ for all $0\leq l \leq N+1$. Moreover, the distributional tensor product of distributions modelled from the $\mathcal{F}$ spaces map bilinearly into the $\mathcal{E}$ spaces; this is made precise in the following result.
\begin{proposition}
\label{Prop: distr tensor product emb}
    Let $R\geq 1$. Let $0\leq k,l \leq N+1$. Then the distributional tensor product
    \begin{equation*}
        \mathcal{F}_{-k,R+1}(\R^d)\times \mathcal{F}_{-l,R+1}(\R^d) \longrightarrow \mathcal{E}_{-(l+k),R}(\R^{2d}); \,(f,g) \longmapsto f \otimes g
    \end{equation*}
    is a bounded bilinear map: for any $\Phi \in \mathcal{E}_{k+l,R}$, one has
    \begin{equation*}
        |\langle f\otimes g, \Phi \rangle| \lesssim \|f\|_{\mathcal{F}_{-k,R+1}} \|g\|_{\mathcal{F}_{-l,R+1}} \|\Phi\|_{\mathcal{E}_{k+l,R}}.
    \end{equation*}
    Moreover, if $f\in \mathcal{C}^{\mathfrak{p}/n}_2 \mathcal{F}_{-k,R+1}$ and $g\in \mathcal{C}^{\mathfrak{p}/m}_2 \mathcal{F}_{-l,R+1}$, then 
    \begin{equation*}
        f\otimes g \in \mathcal{C}^{\mathfrak{p}/(n+m)}_2 \mathcal{E}_{-(l+k),R}.
    \end{equation*}
\end{proposition}
\begin{proof}
    The proof of the above boundedness statement can be found in \cite[Appendix C]{GLN}; we note in passing that the increase in radius $R\mapsto R+1$ is due to the support of $\Phi = \Phi(x,y)\in \mathcal{E}_{k+l,R}$ being contained in the Cartesian product $B_{R+1}\times B_{R+1} = \{(x,y): \max\{|x|,|y|\} \leq R+1\}$. 

    Assume now that $f\colon \Delta^{(2)}_T \to \mathcal{F}_{-k,R+1}$ and $g\colon \Delta^{(2)}_T \to \mathcal{F}_{-l,R+1}$. 
    To see the temporal variation of the distributional tensor product $f\otimes g$, consider 
    \begin{equation*}
        \|(f\otimes g)_{st}\|_{\mathcal{E}_{-(l+k),R}} \lesssim \|f_{st}\|_{\mathcal{F}_{-k,R+1}} \|g_{st}\|_{\mathcal{F}_{-l,R+1}} \lesssim w_f(s,t)^{m/\mathfrak{p}}w_g(s,t)^{n/\mathfrak{p}}
    \end{equation*}
    for $(s,t)\in \Delta^{(2)}_T$, where we have used the finite $\mathfrak{p}/n$- and $\mathfrak{p}/m$-variations of $f$ and $g$, respectively. By \eqref{Eqn: exponent additivity} from Remark \ref{Rem: control properties}, combined with the characterization of $\mathfrak{p}$-variation, the claimed regularity class of $f\otimes g$ follows.
\end{proof}

\subsection{A priori estimates} \label{sec:a priori}

In this subsection, we derive a priori estimates of the abstract expansion 
\begin{equation} \label{eq:URD expansion}
     \delta u_{st} = \delta \mu_{st} + \sum_{n=1}^N \mathrm{A}_{st}^{n} u_s + u_{st}^{\natural}
\end{equation}
on a scale of spaces $(E_l)_{l \geq 0}$. We begin by estimating the remainder term $u^{\natural}$.

\begin{proposition} \label{prop:a priori remainder}
    Let $\mathbf{A}$ be an unbounded rough driver as in Definition \ref{def:URD}. Suppose $u\in \mathcal{B}_b([0,T];E_{-0})$ solves \eqref{eq:URD expansion} on a scale of spaces $(E_l)_{0\leq l \leq N+1}$ endowed with smoothening operators $(J^{\eta})_{\eta \in (0,1]}$. Then there exists a constant $C>0$ such that 
    \begin{equation} \label{eq:a priori natural}
    | u_{st}^{\natural}|_{-(N+1)} \leq C \left( \|u\|_{[s,t],-0}\, w_{\mathbf{A}}(s,t)^{(N+1)/\mathfrak{p}}  + w_{\mu}(s,t) w_{\mathbf{A}}(s,t)^{N/\mathfrak{p}}  \right).
    \end{equation}
\end{proposition}

\begin{proof}
Applying $\delta$ to \eqref{eq:URD expansion} we obtain
\begin{align*}
0  = \delta^2 u_{srt} & = \delta^2 \mu_{srt} + \sum_{n=1}^N \delta (\mathrm{A}^{n} u)_{srt} + \delta u^{\natural}_{srt}  \\
 & = \sum_{n=1}^N ( \delta \mathrm{A}^{n}_{srt} u_s - \mathrm{A}_{rt}^{n} \delta u_{sr} ) + \delta u^{\natural}_{srt}.
\end{align*}
Rearranging and plugging in \eqref{eq:URD chen} we find
\begin{align*}
 \delta u^{\natural}_{srt} & = \sum_{n=1}^N  \mathrm{A}_{rt}^{n} \delta u_{sr}  - \sum_{n=1}^N  \sum_{l=1}^n \mathrm{A}^{l}_{rt} \mathrm{A}^{n-l}_{sr} u_s \\
 & = \sum_{n=1}^N  \mathrm{A}_{rt}^{n} \delta u_{sr}  - \sum_{n=1}^{N-1}  \sum_{l=n+1}^N \mathrm{A}^{n}_{rt} \mathrm{A}^{l-n}_{sr} u_s \\
 & = \sum_{n=1}^N  \mathrm{A}_{rt}^{n}  u_{sr}^{\sharp,(N-n)}
\end{align*}
where we have defined the partial remainders
\begin{equation} \label{eq:partial remainders}
    u_{sr}^{\sharp,(k)} := u_{r} - \sum_{l=0}^{k} \mathrm{A}_{sr}^{l} u_s
\end{equation}
for $k\geq 0$. Note that $u^{\sharp,(0)} = \delta u$ and $u^{\sharp,(N)} = u^{\natural} + \delta \mu$, and that using \eqref{eq:URD expansion} we also obtain the representation
\begin{equation} \label{eq:partial remainders 2}
    u_{sr}^{\sharp,(k)} = \delta \mu_{sr} + \sum_{l=k+1}^N \mathrm{A}_{sr}^{l} u_s + u_{sr}^{\natural}.
\end{equation}
For $n$ fixed and $\phi \in E_{N+1}$ we use the smoothing operator $J^{\eta}$ (with $\eta$ to be determined later) with \eqref{eq:partial remainders} and \eqref{eq:partial remainders 2} to decompose
\begin{align*}
\langle \mathrm{A}_{rt}^{n} u_{sr}^{\sharp,(N-n)}, \phi \rangle & = \left\langle u_r - \sum_{l=0}^{N-n} \mathrm{A}_{sr}^{l}u_s, \, (\mathrm{I}-J^{\eta}) \mathrm{A}_{rt}^{n,*} \phi \right\rangle  + \left\langle   \sum_{l=N-n+1}^{N} \mathrm{A}_{sr}^{l}u_s , \,  J^{\eta} \mathrm{A}_{rt}^{n,*} \phi \right\rangle \\
 & +  \left\langle  \delta \mu_{sr}  , \,  J^{\eta} \mathrm{A}_{rt}^{n,*} \phi \right\rangle +  \langle  u_{sr}^{\natural}  , \,  J^{\eta} \mathrm{A}_{rt}^{n,*} \phi \rangle
\end{align*}
We first estimate
\begin{align*}
\left\langle u_r , \, (\mathrm{I}-J^{\eta}) \mathrm{A}_{rt}^{n,*} \phi \right\rangle & \leq \|u\|_{[s,t],-0} \,|(\mathrm{I}-J^{\eta}) \mathrm{A}_{rt}^{n,*} \phi|_0 \leq \|u\|_{[s,t],-0} \eta^{N+1-n} |\mathrm{A}_{rt}^{n,*} \phi|_{N+1-n} \\
& \leq \|u\|_{[s,t],-0} \eta^{N+1-n} w_{\mathbf{A}}(s,t)^{n/\mathfrak{p}} |\phi|_{N+1}
\end{align*}
and for $l \in \{0,\ldots, N-n\}$ we find 
\begin{align*}
    \left\langle \mathrm{A}_{sr}^{l}u_s, \, (\mathrm{I}-J^{\eta}) \mathrm{A}_{rt}^{n,*} \phi \right\rangle & =\left\langle u_s, \, \mathrm{A}_{sr}^{l,*}(\mathrm{I}-J^{\eta}) \mathrm{A}_{rt}^{n,*} \phi \right\rangle \leq \|u\|_{[s,t],-0} w_{\mathbf{A}}(s,t)^{l/\mathfrak{p}} |(I-J^{\eta}) \mathrm{A}_{rt}^{n,*} \phi|_l \\
    & \leq \|u\|_{[s,t],-0} \,w_{\mathbf{A}}(s,t)^{l/\mathfrak{p}} \eta^{N+1+n-l} | \mathrm{A}_{rt}^{n,*} \phi|_{N+1-n} \\
    & \leq \|u\|_{[s,t],-0} \,w_{\mathbf{A}}(s,t)^{(l+n)/\mathfrak{p}} \eta^{N+1-n-l} | \phi|_{N+1}.
\end{align*}
For $l \in \{N-n+1, \dots, N\}$, use instead
\begin{align*}
    \left\langle \mathrm{A}_{sr}^{l}u_s , \,  J^{\eta} \mathrm{A}_{rt}^{n,*} \phi \right\rangle & = \left\langle u_s , \, \mathrm{A}_{sr}^{l,*}  J^{\eta} \mathrm{A}_{rt}^{n,*} \phi \right\rangle  \leq \|u\|_{[s,t],-0} w_{\mathbf{A}}(s,t)^{l/\mathfrak{p}} |J^{\eta} \mathrm{A}_{rt}^{n,*} \phi|_l \\
    & \leq \|u\|_{[s,t],-0} \,w_{\mathbf{A}}(s,t)^{l/\mathfrak{p}} \eta^{(N+1) - l-n}| \mathrm{A}_{rt}^{n,*} \phi|_{N+1-n} \\
    & \leq \|u\|_{[s,t],-0} \,w_{\mathbf{A}}(s,t)^{(l+n)/\mathfrak{p}} \eta^{(N+1) - l-n} | \phi|_{N+1}.
\end{align*}
The drift term is estimated by
\begin{align*}
    \left\langle  \delta \mu_{sr}  , \,  J^{\eta} \mathrm{A}_{rt}^{n,*} \phi \right\rangle 
    &\leq w_{\mu}(s,t) |J^{\eta} \mathrm{A}_{rt}^{n,*} \phi |_1 \\
    &\leq w_{\mu}(s,t) \eta^{N-n}| \mathrm{A}_{rt}^{n,*} \phi |_{N-n+1} \\
    &\leq w_{\mu}(s,t) \eta^{N-n}w_{\mathbf{A}}(s,t)^{n/\mathfrak{p}} | \phi |_{N+1}
\end{align*}
and finally we estimate the remainder
\begin{align*}
\left\langle  u_{sr}^{\natural}  , \,  J^{\eta} \mathrm{A}_{rt}^{n,*} \phi \right\rangle & \leq w_{\natural}(s,t)^{(N+1)/\mathfrak{p}} | J^{\eta} \mathrm{A}_{rt}^{n,*} \phi |_{N+1} \leq w_{\natural}(s,t)^{(N+1)/\mathfrak{p}} \eta^{-n} |\mathrm{A}_{rt}^{n,*} \phi |_{N+1 - n}  \\
& \leq w_{\natural}(s,t)^{(N+1)/\mathfrak{p}}\eta^{-n}w_{\mathbf{A}}(s,t)^{n/\mathfrak{p}} | \phi |_{N+1}.
\end{align*}
This gives, for each $1 \leq n \leq N$,
\begin{align*}
    |\mathrm{A}_{rt}^{(n)} u_{sr}^{\sharp,(N-n)}|_{-(N+1)}  \leq \|u\|_{[s,t],-0} & \,\bigg( \eta^{N+1-n} w_{\mathbf{A}}(s,t)^{n/\mathfrak{p}} + \sum_{l=0}^{N} w_{\mathbf{A}}(s,t)^{n+l/\mathfrak{p}} \eta^{N+1-n-l} \\
    & + w_{\mu}(s,t) w_{\mathbf{A}}(s,t)^{n/\mathfrak{p}} \eta^{N-n} + w_{\natural}(s,t)^{(N+1)/\mathfrak{p}} w_{\mathbf{A}}(s,t)^{n/\mathfrak{p}} \eta^{-n} \bigg).
\end{align*}
Letting $\eta = w_{\mathbf{A}}(s,t)^{1/\mathfrak{p}} \lambda$ for some constant $\lambda$ to be determined, we find the bound 
\begin{align*}
    |\mathrm{A}_{rt}^{(n)} u_{sr}^{\sharp,(N-n)}|_{-(N+1)}  \leq \|u\|_{[s,t],-0} & \,\big(  w_{\mathbf{A}}(s,t)^{(N+1)/\mathfrak{p}} \left(\lambda^{N+1-n} + \sum_{l=0}^{N}  \lambda^{N+1-n-l} \right)\\
    & + w_{\mu}(s,t) w_{\mathbf{A}}(s,t)^{N/\mathfrak{p}} \lambda^{N-n} + w_{\natural}(s,t)^{(N+1)/\mathfrak{p}}  \lambda^{-n} \big).
\end{align*}
Summing over $n$ yields
$$
|\delta u_{srt}^{\natural}|_{-(N+1)} \lesssim \|u\|_{[s,t],-0}\, w_{\mathbf{A}}(s,t)^{(N+1)/\mathfrak{p}} C_1(\lambda) + w_{\mu}(s,t) w_{\mathbf{A}}(s,t)^{N/\mathfrak{p}} C_2(\lambda)  + w_{\natural}(s,t)^{(N+1)/\mathfrak{p}}  C_3(\lambda)
$$
where we have introduced
\begin{equation*}
    C_1(\lambda) := \sum_{n=1}^N \left(\lambda^{N+1-n} + \sum_{l=0}^{N}  \lambda^{N+1-n-l} \right), \quad C_2(\lambda) := \sum_{n=1}^N \lambda^{N-n},\quad C_3(\lambda) := \sum_{n=1}^N \lambda^{-n}.
\end{equation*}
By the sewing lemma, Lemma \ref{lemma:sewing}, we find a universal constant $\kappa = \kappa_\mathfrak{p}$ such that 
$$
| u_{st}^{\natural}|_{-(N+1)} \leq \kappa \left( \|u\|_{[s,t],-0}\, w_{\mathbf{A}}(s,t)^{(N+1)/\mathfrak{p}} C_1(\lambda) + w_{\mu}(s,t) w_{\mathbf{A}}(s,t)^{N/\mathfrak{p}} C_2(\lambda)  + w_{\natural}(s,t)^{(N+1)/\mathfrak{p}}  C_3(\lambda) \right).
$$
Note that the above right hand side is a control. Using the characterization of $\mathfrak{p}$-variation semi-norms \eqref{eq:p variation control bound} and choosing $\lambda$ such that 
$$
\kappa\, C_3(\lambda) \leq \frac12
$$
and $L$ such that $\eta = w_{\mathbf{A}}(s,t)^{1/\mathfrak{p}} \lambda \leq L^{1/\mathfrak{p}} \lambda \leq 1$, we find the estimate
$$
| u_{st}^{\natural}|_{-(N+1)} \leq \kappa \left( \|u\|_{[s,t],-0} w_{\mathbf{A}}(s,t)^{(N+1)/\mathfrak{p}} C_1(\lambda) + w_{\mu}(s,t) w_{\mathbf{A}}(s,t)^{N/\mathfrak{p}} C_2(\lambda) \right).
$$
Note that the above bound is local in time in the sense that it is valid for $(s,t)$ such that $w_{\mathbf{A}}(s,t)  \leq L$. Using \cite[Lemma A.4]{GLN}, we may update the bound to be global, thus concluding the proof. 
\end{proof}

We now turn to finding an a priori estimate of the path $u$ satisfying \eqref{eq:URD expansion}.

\begin{lemma} \label{lemma:a priori u}
    Let $u\in \mathcal{B}_b([0,T];E_{-0})$ solve \eqref{eq:URD expansion}, and let $\mathbf{A}$ be an unbounded rough driver. There exists a constant $C>0$ such that for every $(s,t) \in \Delta^{(2)}_T$ we have 
    $$
    |\delta u_{st}|_{-1} \leq C  (\|u\|_{[s,t],-0} w_{\bA}(s,t)^{1/\mathfrak{p}} + w_{\mu}(s,t)).
    $$
\end{lemma}

\begin{proof}
For $\phi \in E_1$ we use the smoothing operator to write
\begin{align*}
\langle \delta u_{st}, \phi \rangle  & = \langle \delta u_{st}, J^{\eta} \phi \rangle + \langle \delta u_{st}, (I-J^{\eta} )\phi \rangle 
\end{align*}
and estimate the latter term by
\begin{align*}
|\langle \delta u_{st}, (I-J^{\eta} )\phi \rangle | \leq 2 \|u\|_{[s,t],-0} |(I-J^{\eta} )\phi |_0 \leq 2 \|u\|_{[s,t],-0} \eta |\phi |_1.
\end{align*}
For the smoother first term we use the expansion
\begin{align*}
\langle \delta u_{st}, J^{\eta} \phi \rangle & = \langle \delta \mu_{st}, J^{\eta} \phi \rangle + \sum_{n=1}^N \langle u_{s}, \mathrm{A}^{n,*}_{st} J^{\eta} \phi \rangle + \langle  u_{st}^{\natural}, J^{\eta} \phi \rangle \\
& \leq w_{\mu}(s,t) \,|J^{\eta} \phi |_1 + \sum_{n=1}^N |u_{s}|_{-0}\, |\mathrm{A}^{n,*}_{st} J^{\eta} \phi |_0\ + |u_{st}^{\natural}|_{-(N+1)}\,| J^{\eta} \phi |_{N+1} \\
& \leq w_{\mu}(s,t) \,| \phi |_1 +  \|u\|_{[s,t],-0}\sum_{n=1}^N  w_{\bA}(s,t)^{n/\mathfrak{p}}\,|J^{\eta} \phi |_n + w_{\natural}(s,t)^{(N+1)/\mathfrak{p}}\,|J^{\eta} \phi|_{N+1} \\
& \leq \left( w_{\mu}(s,t)  +  \|u\|_{[s,t],-0} \sum_{n=1}^N  w_{\bA}(s,t)^{n/\mathfrak{p}}\eta^{-n+1} + w_{\natural}(s,t)^{(N+1)/\mathfrak{p}} \eta^{-N} \right) |\phi|_1 .
\end{align*}
Next, we use the bound from Proposition \ref{prop:a priori remainder} and choose $\eta = w_{\bA}(s,t)^{1/\mathfrak{p}}$ and $L$ such that $\eta \in (0,1)$ to finally find that there exists a constant $C$ closing the desired estimate
$$
|\delta u_{st} |_{-1} \leq C (\|u\|_{[s,t],-0} w_{\bA}(s,t)^{1/\mathfrak{p}} + w_{\mu}(s,t))
$$
whenever $(s,t) \in \Delta^{(2)}_T$ such that $w_{\bA}(s,t) \leq L$. Again, using \cite[Lemma A.4]{GLN}, we can update this estimate to be global, thus completing the proof.
\end{proof}

\begin{lemma}
\label{Lem: partial remainder estimate}
    Let $u\in \mathcal{B}_b([0,T];E_{-0})$ solve \eqref{eq:URD expansion}, and let $\mathbf{A}$ be an unbounded rough driver. Let $0\leq k \leq N$. Then the partial remainders $u^{\sharp,(k)}$ defined as in \eqref{eq:partial remainders} are contained in the regularity class
    \begin{equation}
        u^{\sharp,(k)} \in \mathcal{C}^{\mathfrak{p}/(k+1)}_2 E_{-(k+1)}.
    \end{equation} 
\end{lemma}

\begin{proof}
    As before, split 
    \begin{equation}
    \label{Eqn: u-sharp-split}
        \langle u^{\sharp,(k)}_{st},\phi\rangle = \langle u^{\sharp,(k)}_{st}, J^\eta\phi\rangle + \langle u^{\sharp,(k)}_{st}, (I-J^\eta)\phi\rangle
    \end{equation}
    and for the first term, use \eqref{eq:partial remainders 2} to estimate
    \begin{align*}
        \langle u^{\sharp,(k)}_{st}, J^\eta\phi\rangle &= \langle \delta \mu_{st}, J^\eta \phi \rangle + \sum_{l=k+1}^N \langle  u_s, \mathrm{A}^{l,*}_{st}J^\eta \phi\rangle + \langle u^\natural_{st},J^\eta \phi\rangle \\
        & \leq w_{\mu}(s,t) |J^\eta \phi|_1 + \sum_{l=k+1}^N |u_s|_{-0} |\mathrm{A}^{l,*}_{st}J^\eta \phi|_0 + |u^\natural_{st}|_{-(N+1)} |J^\eta\phi|_{N+1} \\
        &\leq w_\mu(s,t)|\phi|_{1} + \|u\|_{[s,t],-0}\sum_{l=k+1}^N w_{\mathbf{A}}(s,t)^{l/\mathfrak{p}} |J^\eta \phi|_l + w_{\natural}(s,t)^{(N+1)/\mathfrak{p}}|J^\eta \phi|_{N+1}\\
        &\leq \left(w_\mu(s,t) + \|u\|_{[s,t],-0}\sum_{l=k+1}^N w_{\mathbf{A}}(s,t)^{l/\mathfrak{p}}\eta^{-(l-k-1)} + w_{\natural}(s,t)^{(N+1)/\mathfrak{p}}\eta^{-(N-k)} \right)|\phi|_{k+1}.
    \end{align*}
    Next, we estimate the second term of \eqref{Eqn: u-sharp-split} using \eqref{eq:partial remainders}, thus
    \begin{align*}
        \langle u^{\sharp,(k)}_{st},(I-J^\eta)\phi\rangle &= \langle \delta u_{st}, (I-J^\eta)\phi \rangle - \sum_{l= 1}^k \langle u_s, \mathrm{A}^{l,*}_{st}(I-J^\eta) \phi\rangle \\
        & \leq |\delta u_{st}|_{-1} |(I-J^\eta) \phi|_1 + \sum_{l=1}^k |u_s|_{-0}\,|\mathrm{A}^{l,*}_{st} (I-J^\eta) \phi|_0 \\
        &\lesssim \bigg(\|u\|_{[s,t],-0}w_{\mathbf{A}}(s,t)^{1/\mathfrak{p}}+w_{\mu}(s,t))\eta^k  \\
        &\qquad+ \|u\|_{[s,t],-0} \sum_{l=1}^k w_{\mathbf{A}}(s,t)^{l/\mathfrak{p}}\eta^{k+1-l}\bigg)|\phi|_{k+1},
    \end{align*}
    where we have used Lemma \ref{lemma:a priori u} in the final transition. Finally, let $\eta = w_{\mathbf{A}}(s,t)^{1/\mathfrak{p}}$ in the above estimates, and choose $L$ such that $\eta \in (0,1)$. Then, we obtain the combined estimate
    \begin{equation*}
        |u^{\sharp,(k)}_{st}|_{-(k+1)} \lesssim \|u\|_{[s,t],-0}w_\mathbf{A}(s,t)^{(k+1)/\mathfrak{p}} + (1+ w_{\mathbf{A}}(s,t)^{k/\mathfrak{p}})w_{\mu}(s,t)
    \end{equation*}
    for all $(s,t)\in \Delta^{(2)}_T$ such that $w_{\mathbf{A}}(s,t) \leq L$. Again, updating to a global bound using \cite[Lemma A.4]{GLN}, we conclude the proof.
\end{proof}

\section{Constructing unbounded rough drivers} 
\label{sec:constructing URD}
Until now, the unbounded rough drivers $\mathbf{A}$ have not been made explicit in their form and structure. In this section, we construct and relay properties of these drivers in two parts. First, we formalize a tensor-algebraic {\it casting operator} $\mathcal{A}^{(n)}$, corresponding to the transport equation, which casts tensor fields $F\colon (\mathbb{R}^d)^n \to (\mathbb{R}^d)^{\otimes n}$ into differential operators. Second, we formalize the concept of an {\it admissible tensor field rough path} $\mathbf{X}$, which lifts the noisy vector field $\mathrm{X}$ into a rough path framework and whose components play the role of the tensor fields that are cast into differential operators that define the unbounded rough drivers.

\subsection{Casting operators} 
\label{sec: casting operators}

\indent\indent We want to realize an unbounded rough driver of a transport equation, and we expect the formal differential operator contained within this driver to be based on a procedure which relates back to our (heuristic) iterated integral form \eqref{eq:URD formally}. On this path, take vector fields $f_1, \dots, f_n$ and define the action of the \emph{casting operator} $\mathcal{A}^{(n)}_x$ on decomposable tensor fields $F = f_1\otimes \cdots \otimes f_n$ as the differential operator
\begin{equation}
\label{Eqn: casting on split}
    \clA^{(n)}_x(f_1 \otimes \dots \otimes f_n)(\phi) = f_n \cdot \nabla \big( \dots \big( f_1 \cdot \nabla \phi \big) \dots \big)    
\end{equation}
for suitable functions $\phi = \phi(x)$. Then, in the smooth setting, \eqref{eq:URD formally} and \eqref{eq:signature of X smooth case} should give us
$$
\mathrm{A}_{st}^{n}  = \int_{\Delta_{st}^{(n)}} \clA^{(n)}_x\left( \dot{\mathrm{X}}_{r_1} \otimes \dots \otimes \dot{\mathrm{X}}_{r_n} \right) \mathrm{d}r_1 \dots \mathrm{d}r_n = \clA^{(n)}_x(\mathrm{X}^{(n)}_{st}),
$$
provided we can interchange the order of integration and (spatial) differentiation. Here, $\mathrm{X}^{(n)}_{st}$ is a component of the yet to be introduced tensor field rough path. Since the tensor field $\mathrm{X}^{(n)}_{st}$ is not necessarily a decomposable tensor, we will extend the operator $\clA^{(n)}_x$ to a larger class of tensor fields. 

Before we move on to extending the casting operator to more general tensor fields, we remark that since \eqref{eq:main eq} will be formulated in the weak sense, we shall think of $\mathcal{A}^{(n)}_x$ as being defined through its $L^2$-dual, which reads
\begin{equation}
\label{Eqn: dual casting on split}
    \clA^{(n),*}_x(f_1 \otimes \dots \otimes f_n)(\psi) = (-1)^n \,\Div ( f_1 ( \dots \Div( f_n \psi) \dots ) )    
\end{equation}
for test functions $\psi = \psi(x)$ chosen from an appropriate scale of spaces. Indeed, the required regularity of the vector fields $f_i$, $1\leq i\leq n$, will be decided from the dual casting operator $\mathcal{A}^{(n),*}_x$; the same will be true for the extension of \eqref{Eqn: dual casting on split} to more general tensor fields. When $n=1$, it is well known that if $f_1, \Div\, f_1 \in W^{N,\infty}(\mathbb{R}^d)$ then the dual casting on $f_1$,
$$
\clA^{(1),*}_x(f_1) \colon W^{N,\infty}(\mathbb{R}^d) \rightarrow W^{N-1,\infty}(\mathbb{R}^d),
$$
is a bounded linear operator. Since the operator $\clA^{(1),*}_x(f_1)$ does not increase the support of $\psi$ we may restrict to the localized spaces $\clA^{(1),*}_x(f_1)\colon E^{\rho}_{N+1} \rightarrow E^{\rho}_{N}$. An easy induction argument shows that if $f_l, \Div\, f_l \in W^{N,\infty}(\mathbb{R}^d)$ for $l=1, \dots, n$, the dual casting applied to the tensor product 
$$
\clA^{(n),*}_x(f_1 \otimes \dots \otimes f_n) \colon E^{\rho}_{N+1} \rightarrow E^{\rho}_{N+1-n}
$$
is also a bounded linear operator. We will soon see the generalization of this remark.

\subsubsection{The casting operator $\mathcal{A}^{(n)}$}
We find it instructive to start with the extension of \eqref{Eqn: casting on split} to a larger class of tensor fields.
Recall the diagonal trace operator $\mathrm{Tr}$ and the projection operator $\Pi$ introduced in Section \ref{sec:notation}. The construction is done by a doubling-of-variables-trick, perhaps most clearly demonstrated in the decomposable case: let $F(x_1,\ldots, x_n) = f_1 \otimes \dots \otimes f_n$ and consider
\begin{align*}
    \clA^{(n)}_x(f_1 \otimes \dots \otimes f_n)(\phi)(x) & = \sum_{j=1}^d f_n^j(x) \frac{\partial}{\partial x_j} \big(\clA^{(n-1)}_x(f_1 \otimes \dots \otimes f_{n-1})(\phi)(x) \big) \\
    & =  \mathrm{Tr}_{\bar{x}=x}\sum_{j=1}^d  \frac{\partial}{\partial x_j} \big(\clA^{(n-1)}_x(f_1 \otimes \dots \otimes f_{n-1})(\phi)(x) f_n^j(\bar{x}) \big)  \\
    & =  \mathrm{Tr}_{\bar{x}=x} \left[\mathrm{div}_x(\mathcal{A}^{(n-1)}_x\otimes \mathrm{Id}_{\bar{x}})(f_1 \otimes \dots \otimes f_n)(\phi)(\,\cdot\, ,\bar{x}))\right]
\end{align*}
where, for each $\bar{x}$ fixed, we regard
$$
x \mapsto (\mathcal{A}^{(n-1)}_x\otimes \mathrm{Id}_{\bar{x}})(f_1 \otimes \cdots \otimes f_n)(\phi)(x ,\bar{x}) = \left( (\mathcal{A}^{(n-1)}_x\otimes \mathrm{Id}_{\bar{x}})(f_1 \otimes \cdots \otimes f_n^j)(\phi)(x ,\bar{x})\right)_{j=1}^d
$$
as a vector field, thereby justifying the application of the divergence in $x$. After computing the divergence, the trace $\mathrm{Tr}_{\bar{x}=x}$ evaluates $\bar{x}$ at $x$. Here, the subscript $x$ appearing in the symbol $\mathcal{A}^{(n)}_x$ indicates both that the cast differential operator performs derivatives of $x$, and that the point of evaluation for the diagonal trace $\mathrm{Tr}$ is $x$.

Based on the above derivation, we aim to describe the casting operator applied to general tensor fields in a similar way using a doubling-of-variables trick, culminating in the following definition.
\begin{definition}
\label{Def: casting operator def recursive}
Let $\phi$ be a function of sufficient regularity. We define the $n$-th level {\it casting operator} $\mathcal{A}^{(n)}_x$ on tensor fields $F\colon (\mathbb{R}^d)^n \to (\mathbb{R}^d)^{\otimes n}$ recursively with respect to the tensor degree $n>1$ by 
\begin{equation}
\label{Eqn:recursive_mathcalA_def}
    \mathcal{A}^{(n)}_x(F)\,\phi = \mathrm{Tr}_{\bar{x}=x} \,\mathrm{div}_x\mathcal{A}^{(n-1)}_x(\Pi F(\,\cdot\, ,\bar{x})\phi(x))
\end{equation}
where $\mathrm{Tr}$ is the diagonal trace \eqref{Eqn: diagonal trace def}, and the expression $\mathcal{A}^{(n-1)}_x(\Pi F(\,\cdot\, ,\bar{x}))$ will be interpreted as a shorthand notation for the vector field whose components are given by
\begin{equation}
\label{Eqn:recursive_mathcalA_def_dummy}
    \mathcal{A}^{(n-1)}_x(\Pi_j F(\,\cdot\,,\bar{x}))(\phi(x)), \quad 1\leq j \leq d.
\end{equation}
Note the introduction of the dummy variable $j$ in this latter expression. The suspended variables present in the projected tensor field $\Pi_j F(\,\cdot\,,\bar{x})$ are implicitly unfurled in the recursion; see Example \ref{Ex: writing out mathcal A} below for an elaboration of this point. Recursion stops for the case $n=1$, where we simply define the casting on vector fields by
\begin{equation}
    \label{Eqn: mathcal A level 1 def}
    \mathcal{A}^{(1)}_x(f)\,\phi = f(x)\cdot \nabla\phi.
\end{equation}
We may also conveniently define $\mathcal{A}^{(0)}_x(f) = \mathrm{Id}$ for any scalar function $f$.
\end{definition}

\begin{remark}
\label{Rem: casting operator remark}
    We establish some conventions regarding the casting operator introduced in Definition \ref{Def: casting operator def recursive}. Firstly, casting operators are only ever applied to tensor fields of the form $F\colon (\mathbb{R}^d)^n \to (\mathbb{R}^d)^{\otimes n}$; the degree is equal to the number of arguments. In the recursive definition, this convention assumes that the lower-order casting $\mathcal{A}^{(n-1)}_x$ ignores the doubled variable $\bar{x}$, but this is consistent since we stringently interpret $\bar{x}$ as independent of $x$. Thus, we effectively think of the projected tensor field $\Pi_j F$ as having one less free variable. Secondly, the subscript $x$ of $\mathcal{A}_x^{(n)}$ and $\mathcal{A}_x^{(n-1)}$ denotes the point of evaluation in the recursive definition; however, these points do not have to agree. We will later meet expressions like
    \begin{equation*}
        \mathrm{Tr}_{\bar{x} = \ubar{x} = x} \mathcal{A}^{(n-1)}_{\ubar{x}}(\Pi_jF(\,\cdot\,,\bar{x}))(\phi(\ubar{x}))
    \end{equation*}
    where the point of evaluation on the casting operator is also manipulated. Finally, it is implicitly understood in \eqref{Eqn:recursive_mathcalA_def} how many times $F$ has been projected into $\Pi_{j_1\dots j_k}F$ in the expression \eqref{Eqn:recursive_mathcalA_def_dummy}; see the case where $n=3$ in Example \ref{Ex: writing out mathcal A}.
\end{remark}

At this point, it is instructive to see an example demonstrating how one would fully write out the differential operator that is cast by non-decomposable tensor fields.
\begin{example}    
\label{Ex: writing out mathcal A}
We begin by considering the case $n=2$ with $F\colon (\mathbb{R}^d)^2 \to (\mathbb{R}^d)^{\otimes 2}$. Let $\phi$ be a sufficiently regular function. Then, using the recursive formula \eqref{Eqn:recursive_mathcalA_def}, we write
\begin{align*}
    \mathcal{A}^{(2)}_x(F)\,\phi &= \mathrm{Tr}_{\bar{x} = x}\,\mathrm{div}_{x}(\mathcal{A}^{(1)}_{x}(\Pi F(\,\cdot\,,\bar{x}))\phi(x)) \\
    &= \mathrm{Tr}_{\bar{x} = x} \sum_{j=1}^d \frac{\partial}{\partial x_j}(\mathcal{A}^{(1)}_x(\Pi_j F(\,\cdot\,,\bar{x}))\phi(x)) \\
    &= \mathrm{Tr}_{\bar{x} = x}\,\sum_{i,j=1}^d \partial_j(\Pi_{ij} F(x,\bar{x})\, \partial_i\phi(x)) \\
    &= \mathrm{Tr}_{\bar{x} = x} \nabla_x\cdot (F(x,\bar{x})\cdot \nabla_x \phi)
\end{align*}
where we have used the recursive definition once, combined with the explicit definition of $\mathcal{A}^{(1)}_x$. Note how we introduce the dummy indices $j$ and $i$ in the intermediate steps. Using the product rule, this final expression can also be written
\begin{equation*}
    \mathrm{Tr}_{\bar{x} = x} \nabla_x\cdot (F(x,\bar{x})\cdot \nabla_x \phi) = F(x,x):\nabla^2_x\phi + \nabla_x^{(1)} \cdot F(x,x)\cdot \nabla_x\phi.
\end{equation*}

Consider now the case $n=3$, letting $G\colon (\mathbb{R}^d)^3\to (\mathbb{R}^d)^{\otimes 3}$, where we are mostly interested in seeing what the projections $\Pi G$ look like as we develop the recursion in \eqref{Eqn:recursive_mathcalA_def}. Write $\partial_j = \partial/\partial x_j$. Then we have
\begin{align*}
    \mathcal{A}^{(3)}_x (G)\,\phi &= \mathrm{Tr}_{\bar{x} = x} \sum_{k=1}^d\partial_k (\mathcal{A}^{(2)}_{x}(\Pi_k G(\,\cdot\,,\bar{x}))\phi(x)) \\
    &=\mathrm{Tr}_{\bar{x}  = x} \sum_{j,k=1}^d \partial_k (\mathrm{Tr}_{\ubar{x} = x}\partial_j\mathcal{A}^{(1)}_x(\Pi_{jk}G(\,\cdot\,,\ubar{x},\bar{x}))\phi(x)) \\
    &= \mathrm{Tr}_{\bar{x}=x}\sum_{i,j,k=1}^d \partial_k \big(\mathrm{Tr}_{\ubar{x} = x}\partial_j(\Pi_{ijk}G(x,\ubar{x},\bar{x}))\partial_i\phi(x)\big).
\end{align*}
It is important that diagonal traces are resolved in the correct order: the outermost differential $\partial_k$ must produce three terms from the product rule, but it can only do so if $\mathrm{Tr}_{\ubar{x}=x}$ is evaluated before differentiating. 
\end{example}

\subsubsection{The dual casting operator $\mathcal{A}^{(n),*}$}

We turn to extending the casting operator $\clA^{(n),*}$ in \eqref{Eqn: dual casting on split} from decomposable tensor fields to a suitable domain of tensor fields $F\colon (\R^d)^n \rightarrow (\R^d)^{\otimes n}$. From this point onward, we shall omit the alternating sign $(-1)^n$ found in \eqref{Eqn: dual casting on split} in the coming definitions of the dual casting operators. 

Let $F\colon (\mathbb{R}^d)^n\to (\mathbb{R}^d)^{\otimes n}$ be a tensor field, and let $\phi = \phi(x)$ be a function of sufficient regularity. Then the recursive formula for the dual casting operator is
\begin{equation}
\label{Eqn: recursive formula of dual}
    \clA^{(n),*}_x(F) \,\phi  = \sum_{j=1}^d \partial_j \clA^{(n-1),*}_x(\Pi ^\intercal_jF(x, \cdot\,))(\phi)(x)
\end{equation}
where the recursion continues until $n-1 =0$ with $\mathcal{A}^{(0),*}_x (\cdot) \equiv \mathrm{Id}$, and at each step the operator $\clA^{(n-1),*}_x$ is applied to the degree $n-1$ tensor field
$$
\Pi^\intercal _jF(x, \cdot\,) \colon (\R^d)^{n-1} \rightarrow (\R^d)^{\otimes (n-1)}
$$
where we recall the transposed projection operator $\Pi^\intercal$ introduced in Section \ref{sec:notation}. Informally, none of the derivatives that define $\mathcal{A}^{(n-1),*}_x(\cdot)$ should act on the first argument of $\Pi^\intercal _jF(x, \cdot\,)$; it should be considered {\it frozen} in the sense that
\begin{equation*}
    \clA^{(n-1),*}_x(\Pi ^\intercal_jF(x, \cdot\,))(\phi)(x) = \mathrm{Tr}_{\bar{x}=x}\, \mathcal{A}^{(n-1),*}_x(\Pi^\intercal_j F(\bar{x},\cdot\,))(\phi)(x),
\end{equation*}
which we interpret as applying the lower-order dual casting $\mathcal{A}^{(n-1),*}_x$ to the degree $n-1$ tensor field $\Pi^\intercal_j F(\bar{x},\cdot\,)$, which we think of as having $n-1$ free variables in view of the second point of Remark \ref{Rem: casting operator remark}. Henceforth, we will omit writing this diagonal trace, but will still keep it in mind in light of the recursive formula. 

Thus far, we have not been overly concerned with the regularity of the tensor fields $F$ being cast by the casting operators $\mathcal{A}^{(n)}_x$, but for the dual casting operators we wish to derive sufficient conditions on the tensor fields for the well-definiteness of $\mathcal{A}^{(n),*}_x(F)$. Abstractly, we package this well-definiteness into the following property.
\begin{definition}[Castability]
\label{Def: castability}
    Let $F\colon (\mathbb{R}^d)^n \to (\mathbb{R}^d)^{\otimes n}$ be a tensor field and let $(E_l)_{0\leq l \leq N+1}$ be a scale of spaces. Then $F$ is said to be {\it castable} with respect to $\mathcal{A}^{(n)}$ on the scale $(E_l)_l$ if, for any $k=0,\ldots, N-n+1$, 
    \begin{equation*}
        \mathcal{A}^{(n),*}(F) \colon E_{n+k} \to E_k
    \end{equation*}
    is a bounded linear operator. 
\end{definition}
Moreover, we will turn our attention to the regularity of functions $\phi$ on which $\mathcal{A}^{(n),*}_x(F)$ will be applied, specifically, specifying appropriate scales $(E_l)_{0\leq l \leq N+1}$. Before we do so, we introduce some notation that helps us explicitly write an identity equation for the dual casting operators.

\begin{definition}[Tail derivative, tail divergences]{\ \\}
\label{Def: tail derivative, tail divergences}
    For $i=1,\dots,n$, define the {\it tail derivative}, starting from the $i$-th variable, by
    \begin{equation}
        \label{Eqn: tail_derivative}
        D_i^{\mathfrak{t}} := \nabla^{(i)}+\nabla^{(i+1)}+\cdots+\nabla^{(n)},
    \end{equation}
    whose components are denoted by $D^\alpha_i = \partial^{(i)}_\alpha + \partial^{(i+1)}_\alpha + \cdots + \partial^{(n)}_\alpha$ for $1\leq \alpha \leq d$.
    
    Let $F\colon(\R^d)^n\to(\R^d)^{\otimes n}$ be a tensor field. We define the corresponding {\it tail divergence} in the $i$-th tensor slot by
    \begin{equation}
        \mathrm{div}^{\mathfrak{t}}_{i}\,F :=
        \sum_{\alpha=1}^d
        \sum_{\hat{J} = j_1\dots \hat{j}_i\dots j_n}D_i^{\mathfrak{t},\alpha}\langle F(\mathrm{x}), e_{j_1\dots \alpha\dots j_n }\rangle e_{\hat{J}}.
    \end{equation}
    Thus, $\Div_i^{\mathfrak{t}} F$ is a tensor field of degree $n-1$, obtained by contracting the tail derivative $D_i$ with the $i$-th tensor slot. Let $\mathrm{Asc}(k;n)$, for integers $k\leq n$, denote the collection of {\it ascending subsets} of $\{k,\ldots,n\}$; $L=\{\ell_1,\ldots \ell_m\}\in \mathrm{Asc}(k;n)$ if and only if $\ell_1 < \cdots < \ell_m$. We denote by $|L|$ the number of elements in $L$. Vacuously, we also consider the empty set ascending and include it in $\mathrm{Asc}(k;n)$ with $|\emptyset|=0$. For any ascending subset $L$, define the {\it $L$-iterated tail divergence} by
    \begin{equation}
    \label{Eqn: def of iterated tail div}
        \mathrm{div}_L^{\mathfrak{t}} F:= \mathrm{div}_{\ell_1}^{\mathfrak{t}}\cdots \mathrm{div}_{\ell_m}^{\mathfrak{t}} F.
    \end{equation}
    The composition is always taken in ascending order in elements of $L$, as shown. Note that $\mathrm{div}^{\mathfrak{t}}_L F$ is a tensor field of degree $n-|L|$. 
\end{definition}

It is instructive to see how we unpack the recursive formula \eqref{Eqn: recursive formula of dual}, illustrating previous remarks, and to also show how the iterated tail divergences enter the picture.
\begin{example}
\label{Ex: recursive and tail example}
    We illustrate the recursive construction \eqref{Eqn: recursive formula of dual} and the use of iterated tail divergences as in Definition \ref{Def: tail derivative, tail divergences} for the case where $n=2$. Expanding the recursive definition, we have
    \begin{align*}
        \mathcal{A}^{(2),*}_x(F) \,\phi &= \mathrm{div}(\mathcal{A}^{(1),*}(\Pi^\intercal_j F(x,\cdot\,))(\phi)) \\
        &= \mathrm{Tr}_{\bar{x}= \ubar{x} = x} \sum_{j=1}^d \mathcal{A}^{(1),*}_{\ubar{x}}(\bar{\partial}_j \Pi^\intercal_j F(\bar{x},\cdot\,)\phi(\ubar{x})) +  \ubar{\partial}_j\mathcal{A}^{(1),*}_{\ubar{x}}(\Pi^\intercal_j F(\bar{x},\cdot))(\phi(\ubar{x})) \\
        &= \mathrm{Tr}_{\bar{x}= x} \sum_{j=1}^d \mathcal{A}^{(1),*}_x(\Pi^\intercal_j F(\bar{x},\cdot)\phi(x)) + \partial_j\mathcal{A}^{(1),*}_{x}(\Pi^\intercal_j F(\bar{x},\cdot))(\phi(x))\\
        &= \mathrm{Tr}_{\bar{x} = x} \sum_{i,j} \partial_i[\bar{\partial}_j \Pi^\intercal_{ij} F(\bar{x},x)\phi(x)] + \partial_j\partial_i[ \Pi^\intercal_{ij} F(\bar{x},x)\phi(x)] \\
        &= \mathrm{Tr}_{\bar{x}=x} \sum_{i,j} D^j_1 D^i_2 \Pi^\intercal_{ij} F(\bar{x},x)\phi(x) + D^j_1 \Pi^\intercal_{ij}F(\bar{x},x)\partial_i\phi(x) \\
        &\quad + D^i_2\Pi^\intercal_{ij}F(\bar{x},x) \partial_j \phi(x) + \Pi^\intercal_{ij} F(\bar{x},x)\partial_{j}\partial_i \phi(x) \\
        &= \mathrm{div}^{\mathfrak{t}}_{\{1,2\}}F(x,x)\phi(x) + \mathrm{div}^{\mathfrak{t}}_1 F(x,x) : \nabla \phi(x) \\
        &\quad + \mathrm{div}^{\mathfrak{t}}_2 F(x,x) : \nabla \phi(x) + F(x,x) : \nabla^2 \phi(x)
    \end{align*}
    where we have used point-wise evaluation to resolve the diagonal trace in the final step.
\end{example}

Motivated by the preceding example, we establish the following generalization, which yields an explicit representation of the dual casting operator. 
\begin{lemma}
\label{Lem: dual casting explicit}
    Let $F\colon (\mathbb{R}^d)^n \to (\mathbb{R}^d)^{\otimes n}$ be a tensor field, and let $\phi$ be a function; both of sufficient regularity. The dual casting operator has the following representation
    \begin{equation}
    \label{Eqn: iterated tail div rep}
        \mathcal{A}^{(n),*}_x(F)\,\phi = \sum_{L\in \mathrm{Asc}(1;n)} \left( \mathrm{Tr}_x \Div_L^{\mathfrak{t}} F \right) : \nabla^{n-|L|}\phi
    \end{equation}
    in terms of iterated tail divergences as introduced in Definition \ref{Def: tail derivative, tail divergences}. Here, $\mathrm{Tr}_x$ is the diagonal trace in all variables of the tensor field, as in \eqref{Eqn: diagonal trace def}.
\end{lemma}
\begin{proof}
    We prove the statement using an induction argument on the degree $n$, employing the already known recursive representation \eqref{Eqn: recursive formula of dual} of the dual casting operator to close the inductive step. 

    First, for tensor fields $F$ of degree $n=1$, we straightforwardly compute 
    \begin{equation*}
        \clA^{(1),*}_x(F)\,\phi = \operatorname{div}(F\phi) =(\operatorname{div}F)\phi + F\cdot \nabla \phi.
    \end{equation*}
    The tail divergence of the last tensor slot will always produce the standard divergence with respect to that tensor slot, hence $\mathrm{div}^{\mathfrak{t}}_1 = \mathrm{div}$, verifying the formula for the $n=1$ case. 

    Inductively, assume that \eqref{Eqn: iterated tail div rep} has been proven for tensor fields of degree $n-1$. Note that, in line with Example \ref{Ex: recursive and tail example}, the recursive representation \eqref{Eqn: recursive formula of dual} can be written 
    \begin{equation}
    \label{Eqn: dual bar-ubar recursive}
        \clA^{(n),*}_x(F)\,\phi(x) = \mathrm{Tr}_{\bar{x}=\ubar{x}=x}\sum_{j=1}^d
        (\bar{\partial}_j + \ubar{\partial}_j)
        \left[ \clA^{(n-1),*}_{\ubar{x}}(\Pi^\intercal_jF(\bar{x},\cdot))\phi(\ubar{x}) \right],
    \end{equation}
    where $\bar{\partial}_j = \partial/\partial \bar{x}_j$, similarly for $\ubar{\partial}_j$. 
    By shifting the ascending subsets $L \in \mathrm{Asc}(1;n-1)$ to $L \in \mathrm{Asc}(2;n)$, we have, by the induction hypothesis applied to the $n-1$ tensor field $\Pi^\intercal_jF(\bar{x},\cdot)$, that 
    \begin{equation*}
        \clA^{(n-1),*}_x(\Pi^\intercal_jF(\bar{x},\cdot))\phi(\ubar{x}) = \sum_{L \in \mathrm{Asc}(2;n)} \mathrm{Tr}_{\ubar{x}, (2;n)} \operatorname{div}_L^{\mathfrak{t}} \Pi^\intercal_j F(\bar{x},\cdot) : \nabla^{n-1-|L|}_{\ubar{x}}\phi(\ubar{x}),
    \end{equation*}
    where $\mathrm{Tr}_{\ubar{x},(2;n)}$ denotes the diagonal trace evaluated to $\ubar{x}$ in the last $n-1$ variables. Substituting the above expansion of $\mathcal{A}^{(n-1),*}_x$ into the recursive identity \eqref{Eqn: dual bar-ubar recursive}, we obtain
    \begin{equation*}
        \clA^{(n),*}_x(F)\,\phi = \mathrm{Tr}_{\bar{x} = \ubar{x} =x}\sum_{j=1}^d \sum_{L\in \mathrm{Asc}(2;n)}(\bar{\partial}_j+\ubar{\partial}_{j})\left[\mathrm{Tr}_{\ubar{x}, (2;n)}\operatorname{div}_{L}^{\mathfrak{t}} \Pi^\intercal_jF(\bar{x},\cdot):\nabla^{n-1-|L|}_{\ubar{x}}\phi(\ubar{x})\right]
    \end{equation*}
    Consider for the moment a fixed ascending subset $L$. By the product rule, the derivatives $\ubar{\partial}_j$ are distributed onto either $\mathrm{Tr}_{\ubar{x}, (2;n)}\operatorname{div}_{L}^{\mathfrak{t}} \Pi^\intercal_jF(\bar{x},\cdot)$ or $\nabla^{n-1-|L|}_{\ubar{x}}\phi(\ubar{x})$. In the first case, we may collect terms and observe that
    \begin{equation*}
        (\bar{\partial}_j + \ubar{\partial}_j)[\mathrm{Tr}_{\ubar{x}, (2;n)}\operatorname{div}_{L}^{\mathfrak{t}} \Pi^\intercal_jF(\bar{x},\cdot)] = D^1_j (\mathrm{Tr}_{\ubar{x}, (2;n)}\operatorname{div}_{L}^{\mathfrak{t}} \Pi^\intercal_jF(\bar{x},\cdot)),
    \end{equation*}
    hence after summing over the index $j$ and taking the trace, these terms produce the $I$-iterated divergence
    \begin{equation*}
        (\mathrm{Tr}_x\mathrm{div}^{\mathfrak{t}}_{I} F) : \nabla^{n-|I|}\phi
    \end{equation*}
    with $I = \{1\}\cup L$, $|I| = |L| + 1$. In the second case, the derivatives are applied to $\nabla^{n-1-|L|}_{\ubar{x}}\phi(\ubar{x})$, producing $\nabla^{n-|L|}_{\ubar{x}}\phi(\ubar{x})$ after summing over the index $j$. After taking the diagonal trace, the resulting term produces the contraction
    \begin{equation*}
        (\mathrm{Tr}_x\mathrm{div}^{\mathfrak{t}}_{L} F): \nabla^{n-|L|}\phi
    \end{equation*}
    for every remaining ascending subset $L\in \mathrm{Asc}(2;n)$. This concludes the proof. 
\end{proof}

In light of Lemma \ref{Lem: dual casting explicit}, we are now in a position to easily deduce the sufficient regularity of the tensor fields $F$ that are cast by the dual casting operator.
\begin{definition}
\label{Def: tail divergence norm}
    Let $F\colon (\mathbb{R}^d)^n \to (\mathbb{R}^d)^{\otimes n}$ be a tensor field, and let $k \geq 0$. We define the {\it tail divergence norm} $\|\cdot\|_{\mathfrak{t},k}$ over $\mathbb{R}^d$ by the quantity 
    \begin{equation}
        \|F\|_{\mathfrak{t},k} :=  \sum_{L\in \mathrm{Asc}(1;n)}\|\mathrm{div}^{\mathfrak{t}}_L F\|_{W^{k,\infty}(\mathbb{R}^d)}.
    \end{equation}
    We view $n$ as being implicit from the degree of the input tensor field. 
\end{definition}

The following result is crucial for our analysis. It allows us to easily ascertain a sufficient condition for castability of tensor fields in terms of the tail divergence norm of the tensor fields considered, and also specifies the scales of test functions of interest.  
\begin{lemma}
\label{Lem: boundedness 1}
    Assume $\rho$ satisfies Assumption \ref{assumption:psi} and consider a scale of spaces $(E_l^{\rho})_{0\leq l \leq N+1}$ as in Section \ref{sec:URD}. Let $F\colon (\mathbb{R}^d)^n \to (\mathbb{R}^d)^{\otimes n}$ be a tensor field with finite tail divergence norm $\|F\|_{\mathfrak{t},N-n+1}$ for some $N \geq n-1$. Then $F$ is castable with respect to $\mathcal{A}^{(n)}$ on the scale of spaces $(E_l^\rho)_{0\leq l \leq N+1}$; the operator
    \begin{equation*}
        \mathcal{A}^{(n),*}(F) \colon E^{\rho}_{n+k} \to E^\rho_k 
    \end{equation*}
    is bounded for any $k = 0,1,\ldots, N-n+1$, and the estimate
    \begin{equation}
        \|\mathcal{A}^{(n),*}(F)\,\phi\|_{E^\rho_k} \leq C \|F\|_{\mathfrak{t},k} \|\phi\|_{E^\rho_{n+k}}
    \end{equation}
    holds for some constant $C$ depending only on $k$ and $n$. 
\end{lemma}
\begin{proof}
    Recall that the localized spaces $E^\rho_k$ have norms that are inherited from $W^{k,\infty}(\mathbb{R}^d)$. We treat all cases $0 \leq k \leq N-n+1$ simultaneously by studying how a formal differential operator $\partial^m$ distributes when applied to our representation from Lemma \ref{Lem: dual casting explicit}; explicitly
    \begin{equation*}
        \partial^{m}(\mathcal{A}^{(n),*}_x(F)\phi) = \sum_{L\in \mathrm{Asc}(1;n)} \partial^m\left[\left( \mathrm{Tr}_x\, \Div_L^{\mathfrak{t}} F \right) : \nabla^{n-|L|}\phi\right].
    \end{equation*}
    Consider a fixed summand indexed by $L\in \mathrm{Asc}(1;n)$. By a formal product rule, $\partial^m$ distributes over both $\mathrm{Tr}_x\mathrm{div}^{\mathfrak{t}}_L F$ and $\nabla^{n-|L|}\phi$, and we generically obtain terms of the form
    \begin{equation*}
        \mathrm{Tr}_x(\partial^{(1), m_1}\cdots \partial^{(n),m_n}\mathrm{div}^{\mathfrak{t}}_L F): \partial^{l}\nabla^{n-|L|}\phi
    \end{equation*}
    where $\partial^{(j),m_j}$ denotes the derivative in the $j$-th variable performed $m_j$ times, and, moreover, $m_1+\cdots + m_n + l = m$. It is clear that if $\phi\in E^{\rho}_{n+k}$, then $\nabla^{n-|L|}\phi \in E^\rho_{k+|L|}\subset E^\rho_k$ implies that $0\leq l\leq k$, but also since $\|F\|_{\mathfrak{t},k}$ is finite by assumption, then $m_1+\cdots+m_n \leq k$. Since the orders balance, our maximally allowed total order $m$ must be equal to $k$; the claimed Sobolev regularity follows after taking $L^\infty$-norms with appropriate support restrictions. 
\end{proof}

\begin{corollary}
    \label{Cor: casting operator}
    Assume $\rho$ satisfies Assumption \ref{assumption:psi} and let $(E_l^{\rho})_{0\leq l \leq N+1}$ be a scale of spaces. 
    For any $n$ there exists a unique extension of the casting operator $\clA^{(n)}$ to tensor fields $F\colon (\mathbb{R}^d)^n\to (\mathbb{R}^d)^{\otimes n}$ with $\|F\|_{\mathfrak{t},N-n+1}<\infty$ such that 
    \begin{equation} 
    \label{eq:casting F gives bounded operator}
        \clA^{(n)}(F) \colon E_{-k}^{\rho} \rightarrow E_{-k-n}^{\rho}, \quad 0\leq k\leq N-n+1.
    \end{equation}
    Moreover, we have the estimate
     \begin{equation} \label{eq:casting F bound}
       \| \clA^{(n)}(F) \|_{\mathcal{L}( E_{-k}^{\rho} ; E_{-k-n}^{\rho} )} \leq \|F \|_{\mathfrak{t},k} \leq \|F\|_{\mathfrak{t}, N-n+1}
    \end{equation}
    and, for tensor fields $F$ and $G$ under suitable regularity conditions,
    \begin{equation} 
    \label{eq:casting F tensor G}
       \clA^{(n+m)}(F \otimes G) = \clA^{(m)}(G) \circ \clA^{(n)}(F).
    \end{equation}
\end{corollary}

\subsubsection{Casting tensor fields over $\mathbb{R}^{2d}$}
    \indent\indent In this subsection, we will study properties and transformations of tensor fields over $\mathbb{R}^{2d}$, including the differential operators that are cast from these. We will also study doubled tensor fields, as introduced in Section \ref{sec:notation}, which are tensor fields over $\mathbb{R}^{2d}$ that enjoy additional structure in terms of their binary projections and will be a key instrument in our renormalization analysis carried out in Section \ref{sec:renormalization}.

    \noindent
    As noted in Section \ref{sec:notation}, generic arguments and points will now be denoted with $z$ instead of $x$. In light of this, we remark that the $\mathbb{R}^{2d}$-equivalents of the recursive formulae \eqref{Eqn:recursive_mathcalA_def} and \eqref{Eqn: recursive formula of dual} are exactly the same after performing the obvious substitutions. 
    
    The {\it symmetrization-antisymmetrization operator} $\mathcal{S}$ is a linear isomorphism on $\mathbb{R}^{2d}$, whose action in the canonical basis $\{e_j\}_{j=1}^{2d}$ reads
    \begin{equation}
    \label{Eqn: S transform basis vectors}
        \mathcal{S}e_j  = \begin{cases}
            \dfrac{e_j+e_{j+d}}{2}, &[j] =1, \\
            \dfrac{e_{j-d}-e_{j}}{2}, & [j] = 2.
        \end{cases}
    \end{equation}
    Equivalently, in block matrix form, this transformation is represented by
    \begin{equation}
    \label{Eqn: mathcal S block matrix}
        \mathcal{S} = \frac{1}{2}\begin{pmatrix}
            \mathrm{I}_d & \mathrm{I}_d \\
            \mathrm{I}_d & - \mathrm{I}_d
        \end{pmatrix}
        =
        \frac{1}{2}\begin{pmatrix}
            1 & 1 \\
            1 &- 1
        \end{pmatrix}
        \otimes \mathrm{I}_d
        .
    \end{equation}
    Extending the above to tensors, we write 
    \begin{equation*}
        \mathcal{S}^{\otimes k}e_J = S e_{j_1}\otimes Se_{j_2}\otimes\cdots\otimes Se_{j_k}.
    \end{equation*}
    In light of \eqref{Eqn: S transform basis vectors} we can also make sense of the inverse $\mathcal{S}^{-1}$ when applying $\mathcal{S}$ to vectors and tensors. It also follows that we can define the action of $\mathcal{S}$ on tensor fields: let $F\colon (\mathbb{R}^{2d})^{n} \to (\mathbb{R}^{2d})^{\otimes k}$ be a degree $k$ tensor field over $\mathbb{R}^{2d}$, then its {\it symmetrized-antisymmetrized version} $\mathcal{S}^{\otimes k}F$ is defined by
    \begin{equation}
        \mathcal{S}^{\otimes k}F := \sum_{|J|= k} \langle F, e_J\rangle \mathcal{S}^{\otimes k}e_J  = \sum_{|I|=k} \langle F,\mathcal{S}^{\otimes k}e_I\rangle e_I.
    \end{equation}
    Of course, $\mathcal{S}^{\otimes k}F$ is a new tensor field of degree $k$ over $\mathbb{R}^{2d}$, whose components expressed in coordinates are
    \begin{equation}
    \label{Eqn: sym-asym coeff functions}
        \langle F, \mathcal{S}e_{i_1\dots i_{k}}\rangle = \sum_{j_1,\cdots, j_{k}}\mathcal{S}_{i_1j_1}\cdots \mathcal{S}_{i_{k}j_{k}}\langle F,e_{j_1\dots j_{k}}\rangle
    \end{equation}
    where we have represented $\mathcal{S}_{ij}$ as elements of the block matrix form \eqref{Eqn: mathcal S block matrix}.  

    \begin{remark}
        In \cite{BaiGub2017} the authors define the concept of {\it symmetric-antisymmetric lifts} of scalar functions $h \colon \mathbb{R}^{d}\to \mathbb{R}$ to functions $h^\pm \colon \mathbb{R}^{2d}\to \mathbb{R}$ as 
        \begin{equation*}
            h^\pm(z) = h^{\pm}(x,y) := h(x) \pm h(y)
        \end{equation*}
        where the superscripts denote the symmetrization and antisymmetrization respectively. In our present setting, this is captured by \eqref{Eqn: sym-asym coeff functions} when the tensor fields considered have additional structure. Indeed, if $\bar{F}$ is the doubled tensor field of $F$, then the scalar coefficient functions $\langle \bar{F},\mathcal{S}e_I\rangle$ capture exactly the symmetric-antisymmetric lifts of $F$.
    \end{remark} 

Let $z = (x,y)^T \in \mathbb{R}^{2d}$, and note that for vector fields $f \colon \mathbb{R}^{2d} \to \mathbb{R}^{2d}$, the symmetrization-antisymmetrization operator $\mathcal{S}$ can be used to rewrite the casting operator $\mathcal{A}^{(1)}_z(f)$ in terms of derivatives $\nabla_+$ and $\nabla_-$:
\begin{equation*}
    \mathcal{A}^{(1)}_z(f)\,\Phi = f\cdot \nabla_z \Phi = \mathcal{S}f \cdot (\mathcal{S}^{-1}\nabla_z) \,\Phi(z) = \mathcal{S}f \cdot \nabla_{\pm}(\Phi(z)),
\end{equation*}
where we have naturally defined
\begin{equation}
    \nabla_{\pm} := (\nabla_+,\nabla_-)^T := \mathcal{S}^{-1}(\nabla_x, \nabla_y)^T.
\end{equation}
This observation generalizes, culminating in the following representation result.
\begin{lemma}
\label{Lem: sym-asym correspondance}
    Let $F\colon (\mathbb{R}^{2d})^n \to (\mathbb{R}^{2d})^{\otimes n}$ be a tensor field of degree $n$ over $\mathbb{R}^{2d}$. We have 
    \begin{equation}
    \label{Eqn: sym-asym recursive rule}
        \mathcal{A}^{(n)}_z(F)\,\Phi = \mathrm{Tr}_{\bar{z}=z} \,\mathrm{div}_{\pm} \mathcal{S}\mathcal{A}^{(n-1)}_z (\Pi F(\,\cdot\,, \bar{z}))(\Phi(z))
    \end{equation}
    where $\mathrm{div}_{\pm} V= \nabla_{\pm}\cdot V$ for vector fields $V$, and also 
    \begin{equation}
    \label{Eqn: sym-asym tensorized}
        \mathcal{A}^{(n)}_z(F)\,\Phi = \mathrm{Tr}_{\bar{z}=z} \,\mathrm{div}_{\pm} \mathcal{A}^{(n-1)}_z (\Pi \mathcal{S}^{\otimes n} F(\,\cdot\,,\bar{z}))(\Phi(z))
    \end{equation}
    where $\mathcal{S}^{\otimes n}F$ is the symmetrized-antisymmetrized tensor field of $F$.
\end{lemma}
\begin{proof}
    The recursive formula \eqref{Eqn: sym-asym recursive rule} is easily established and illustrates the main take-away: since $\pi_1(\mathcal{S}^{-1} \nabla_z) = \nabla_+$ and $ \pi_2(\mathcal{S}^{-1}\nabla_z) = \nabla_-$, we have
    \begin{align*}
        \mathcal{A}_z^{(n)}(F)\,\Phi &= \mathrm{Tr}_{\bar{z}=z} \,\mathrm{div}_z\, \mathcal{S}^{-1} \mathcal{S}\mathcal{A}^{(n-1)}_z (\Pi F(\,\cdot\,,\bar{z}))(\Phi(z)) \\
        &= \mathrm{Tr}_{\bar{z} = z}\,\mathrm{div}_{\pm}\, \mathcal{S}\mathcal{A}_z^{(n-1)}(\Pi F(\,\cdot\,,\bar{z}))(\Phi(z))
    \end{align*}
    using the symmetry of the inverse matrix $\mathcal{S}^{-1}$. Proving \eqref{Eqn: sym-asym tensorized} involves a simple induction argument on $n$, where the inductive step uses the recursive formula \eqref{Eqn: sym-asym recursive rule} and the definition of the symmetrized-antisymmetrized tensor field $\mathcal{S}^{\otimes n}F$. 
\end{proof}

The moral of Lemma \ref{Lem: sym-asym correspondance} is that we may rewrite the casting operators $\mathcal{A}^{(n)}_z$ acting on tensor fields $F$ as higher-order divergence operators involving $\nabla_{\pm}$ and the symmetrized-antisymmetrized tensor fields $\mathcal{S}^{\otimes n}F$; we have not changed the coordinates $z$ (or indeed $\bar{z}$). However, a coordinate change in \eqref{Eqn: sym-asym tensorized} results in the push-forward of $\mathcal{A}^{(n)}_z(F)$ by $\mathcal{S}$. Motivated by this observation, we prove a generalization, called the {\it equivariance property}, obeyed by both the casting operator and its dual.
\begin{lemma}[Equivariance]
\label{Lem: equivariance push-forward}
    Let $F\colon(\mathbb{R}^{2d})^n \to (\mathbb{R}^{2d})^{\otimes n}$ be a tensor field over $\mathbb{R}^{2d}$, and let $\mathcal{M}\colon \mathbb{R}^{2d} \to \mathbb{R}^{2d}$ be a constant linear isomorphism. Then the push-forward of the dual casting operator under $\mathcal{M}$ is equivariant in the sense that
    \begin{equation}
    \label{Eqn: equivariance dual casting}
        \mathcal{M}_* (\mathcal{A}^{(n),*}_z(F)) = \mathcal{A}^{(n),*}_w(\mathcal{M}_* F)
    \end{equation}
    where $w = \mathcal{M}z$ and $\mathcal{M}_* F = \mathcal{M}^{\otimes n} F\circ (\mathcal{M}^{-1})^{\times n}$ is the push-forward of the tensor field $F$ by $\mathcal{M}$. Moreover, the same equivariance result is true for the casting operator itself:
    \begin{equation}
    \label{Eqn: equivariance casting}
        \mathcal{M}_* (\mathcal{A}^{(n)}_z(F)) = \mathcal{A}^{(n)}_w(\mathcal{M}_* F).
    \end{equation}
\end{lemma}
\begin{proof}
    The naturality of the push-forward under adjunction implies that it is sufficient to prove \eqref{Eqn: equivariance dual casting}; the equivariance relation \eqref{Eqn: equivariance casting} follows by passing to the dual. 

    Recall the explicit representation \eqref{Eqn: iterated tail div rep} of the dual casting operator. Let $w = \mathcal{M}z$ and consider some test function $\psi = \psi(w)$.
    It is well-known that the push-forward of a differential operator $\mathcal{D}$ under $\mathcal{M}$ writes
    \begin{equation}
        (\mathcal{M}_*(\mathcal{D})\psi)(w) = [\mathcal{D}(\psi\circ \mathcal{M})](\mathcal{M}^{-1}w).
    \end{equation}
    Hence, with $\mathcal{D} = \mathcal{A}^{(n),*}(F)$, the identity to be shown is 
    \begin{equation}
    \label{Eqn: w-z correspondance equivariance}
        \mathcal{A}^{(n),*}_w (\mathcal{M}_*F)\,\psi(w) = \mathcal{A}^{(n),*}_z(F)(\psi\circ \mathcal{M})(z).
    \end{equation}
    On this path, we start from the left and explicitly write
    \begin{equation}
    \label{Eqn: iter tail div w coords}
        \mathcal{A}^{(n),*}_w (\mathcal{M}_*F)\,\psi(w) = \sum_{L\in \mathrm{Asc}(1;n)} (\mathrm{Tr}_w\, \mathrm{div}^{\mathfrak{t}}_{L,w}\, \mathcal{M}_*F): \nabla^{n-|L|}_w \psi(w),
    \end{equation}
    where since $\mathcal{M}_* F = \mathcal{M}^{\otimes n} F\circ (\mathcal{M}^{-1})^{\times n}$ we have by the contractions of iterated tail divergences and the chain rule that
    \begin{equation*}
        \mathrm{Tr}_w\, \mathrm{div}^{\mathfrak{t}}_{L,w}\, \mathcal{M}_*F = \mathcal{M}^{\otimes( n-|L|)}[\mathrm{Tr}_z \,\mathrm{div}^{\mathfrak{t}}_{L,z} F](\mathcal{M}^{-1}w).
    \end{equation*}
    Moreover, the chain rule also yields
    \begin{equation*}
        (\mathcal{M}^{-T})^{\otimes (n-|L|)}\nabla^{n-|L|}_z \psi(\mathcal{M}z) = \nabla_w^{n-|L|}\psi (w),  
    \end{equation*}
    and by contracting tensors, each summand of \eqref{Eqn: iter tail div w coords} reads 
    \begin{equation*}
        (\mathrm{Tr}_w\, \mathrm{div}^{\mathfrak{t}}_{L,w}\, \mathcal{M}_*F): \nabla^{n-|L|}_w \psi(w) = (\mathrm{Tr}_z\, \mathrm{div}^{\mathfrak{t}}_{L,z}\,F): \nabla^{n-|L|}_z \psi(\mathcal{M}z).
    \end{equation*}
    Finally, taking the sum over all ascending subsets $L\in \mathrm{Asc}(1;n)$ proves \eqref{Eqn: w-z correspondance equivariance}, thus concluding the proof.     
\end{proof}

The novelty of the representation \eqref{Eqn: sym-asym tensorized} is that the derivatives $\nabla_+$ and $\nabla_-$ commute nicely with the dual $T^*_\varepsilon$ of the blow-up transformation that we introduce in Section \ref{sec:renormalization} -- see Lemma \ref{Lem: commutation relations}. This dual transformation is the coordinate scaling $(x_+,x_-) \mapsto (x_+,x_-/\varepsilon)$; informally, the aforementioned commutator relation is reflected in that the chain rule used in the proof of Lemma \ref{Lem: equivariance push-forward} now produces a singular scaling $\varepsilon^{-1}$, diverging in the limit $\varepsilon \to 0$. We wish to retain derivatives in $\nabla_+$ and $\nabla_-$ as they are, and thus find it convenient to introduce notation that instead places this singular scaling on the symmetrization-antisymmetrization operator $\mathcal{S}$.
\begin{definition}[$\varepsilon$-transformation]
\label{Def: espilon-transf}
    Let $F\colon (\mathbb{R}^{2d})^n \to (\mathbb{R}^{2d})^{\otimes n}$ be a tensor field over $\mathbb{R}^{2d}$. Let $\varepsilon\in (0,1)$. We define the {\it $\varepsilon$-transformed tensor field} $F_\varepsilon\colon (\mathbb{R}^{2d})^{n}\to (\mathbb{R}^{2d})^{\otimes n}$ by its coefficients
    \begin{equation}
    \label{Eqn: epsilon-transf-def}
        \langle F_\varepsilon(z_1,\ldots,z_n), e_{i_1\dots i_{n}}\rangle = \sum_{j_1\dots j_{n}} \mathcal{S}^\varepsilon_{i_1j_1}\cdots \mathcal{S}^\varepsilon_{i_{n}j_{n}} \langle F(N^\varepsilon z_1,\ldots, N^\varepsilon z_n), e_{j_1\dots j_{n}}\rangle
    \end{equation}
    where the block matrix $\mathcal{S}^\varepsilon$ is the singularly scaled symmetrization-antisymmetrization operator given by 
    \begin{equation}
    \label{Eqn: mathcal S eps def}
        \mathcal{S}^\varepsilon := \frac{1}{2}\begin{pmatrix}
            \mathrm{I}_d & \mathrm{I}_d \\
            \varepsilon^{-1}\,\mathrm{I}_d & -\varepsilon^{-1}\,\mathrm{I}_d
        \end{pmatrix}
         = \frac{1}{2}\begin{pmatrix}
             1 & 1 \\
             \varepsilon^{-1} & -\varepsilon^{-1}
         \end{pmatrix} \otimes \mathrm{I}_d
         =:  s^{\varepsilon}\otimes \mathrm{I}_d
    \end{equation}
    and the coordinate change $N^\varepsilon$ is given by
    \begin{equation}
        N^\varepsilon := (\mathcal{S}^{\varepsilon})^{-1}\mathcal{S}= \frac{1}{2}\begin{pmatrix}
            (1 + \varepsilon)\,\mathrm{I}_d & (1 - \varepsilon)\,\mathrm{I}_d \\
            (1 - \varepsilon)\,\mathrm{I}_d & (1 + \varepsilon)\,\mathrm{I}_d
        \end{pmatrix}. 
    \end{equation}
    Equivalently, using binary projections, \eqref{Eqn: epsilon-transf-def} reads
    \begin{equation}
    \label{Eqn: F-vareps-little-s}
        \pi_{i_1\dots i_n} F_\varepsilon(z_1,\ldots,z_n) = \sum_{j_1\dots j_n} s^\varepsilon_{i_1 j_1}\cdots s^\varepsilon_{i_nj_n}\pi_{j_1\dots j_n} F(N^\varepsilon z_1,\ldots, N^\varepsilon z_n).
    \end{equation}
\end{definition}

Equipped with the notation above, we prove the following representation result. This result later turns up as a representation identity when conjugating differential operators by the dual of the blow-up transformation, $T^*_\varepsilon$, which is yet to be formally introduced, but will become central in Section \ref{sec:renormalization}. 
\begin{proposition}
\label{Prop: pushforward F-eps representation}
    Let $F\colon (\mathbb{R}^{2d})^n \to (\mathbb{R}^{2d})^{\otimes n}$ be a tensor field. Let $\varepsilon\in (0,1)$. Then the push-forward of the dual casting operator under $\mathcal{S}^{-1}\mathcal{S}^{\varepsilon}$ is represented by
    \begin{equation}
    \label{Eqn: blow-up representation dual casting}
          (\mathcal{S}^{-1}\mathcal{S}^{\varepsilon})_* (\mathcal{A}^{(n),*}_{z}(F)) = \mathcal{A}^{(n),*}_{\pm}(F_\varepsilon) 
    \end{equation}
    where $F_\varepsilon$ is the $\varepsilon$-transformed $F$ and the notation $\pm$ in the subscript signals to compute the diagonal trace along $(x_+,x_-)$ and use the derivatives $\nabla_{\pm}$; its explicit representation reads
    \begin{equation}
    \label{Eqn: F-eps transf explicit}
        \mathcal{A}^{(n),*}_{\pm} (F_\varepsilon)\,\Phi = \sum_{L\in \mathrm{Asc}(1;n)} \mathrm{Tr}_{x_{\pm}} \mathrm{div}^{\mathfrak{t},\pm}_L F_\varepsilon : \nabla^{n-|L|}_{\pm}\Phi.
    \end{equation}
    Moreover, the same push-forward of the usual casting operator admits the representation 
    \begin{equation}
         (\mathcal{S}^{-1}\mathcal{S}^\varepsilon)_*(\mathcal{A}^{(n)}_z(F)) = \mathcal{A}^{(n)}_{\pm}(F_\varepsilon).
    \end{equation}
\end{proposition}
\begin{proof}
    As in the proof of Lemma \ref{Lem: equivariance push-forward}, it will be sufficient to only prove \eqref{Eqn: blow-up representation dual casting}. To start, notice that the $\varepsilon$-transform $F_\varepsilon$ is nearly a push-forward itself; it is actually equal to $\mathcal{S}^{\otimes n} (\mathcal{S}^{-1}\mathcal{S}^\varepsilon)_* F$. By the equivariance property described in Lemma \ref{Lem: equivariance push-forward}, we write the push-forward as
    \begin{equation*}
        (\mathcal{S}^{-1}\mathcal{S}^\varepsilon)_*( \mathcal{A}^{(n),*}_{z}(F)) = \mathcal{A}^{(n),*}_{w} ((\mathcal{S}^{-1}\mathcal{S}^{\varepsilon})_* F).
    \end{equation*}
    Just as in Lemma \ref{Lem: sym-asym correspondance}, we are free to symmetrize-antisymmetrize the tensor field $(\mathcal{S}^{-1}\mathcal{S}^{\varepsilon})_* F$ at the cost of exchanging $\nabla_w \mapsto \nabla_{\pm} = \mathcal{S}^{-1}\nabla_w$, hence we have
    \begin{equation*}
        \mathcal{A}^{(n),*}_{w} ((\mathcal{S}^{-1}\mathcal{S}^{\varepsilon})_* F) = \mathcal{A}^{(n),*}_{\pm} (\mathcal{S}^{\otimes n}(\mathcal{S}^{-1}\mathcal{S}^{\varepsilon})_* F) = \mathcal{A}^{(n),*}_{\pm} (F_\varepsilon).
    \end{equation*}
    After a harmless relabelling of $w$ to $z$ and identifying $\mathcal{S}z = (x_+,x_-)^T$, one also obtains the representation \eqref{Eqn: F-eps transf explicit} directly. 
\end{proof}

The following result will be essentially equivalent to the renormalizability of drivers, which will play an important role in Section 6. At this stage, it proves the uniform-in-$\varepsilon$ castability of $\varepsilon$-transformed tensor fields that are doubled from tensor fields $F\in W^{N+1,\infty}(\mathbb{R}^d)$. 
\begin{proposition}
\label{Prop: renorm result tensor fields}
    Let $\bar{F}\colon (\mathbb{R}^{2d})^n \to (\mathbb{R}^{2d})^{\otimes n}$ be a tensor field over $\mathbb{R}^{2d}$ that is doubled from $F\in W^{N+1,\infty}( (\mathbb{R}^d)^n;(\mathbb{R}^d)^{\otimes n})$. Let $\varepsilon\in (0,1)$, and denote by $F_\varepsilon$ the $\varepsilon$-transform of $\bar{F}$. Assume $\zeta$ satisfies Assumption \ref{assumption:psi} and consider a scale of spaces $(E^\zeta_l)_{0\leq l \leq N+1}$ such that, for any $l$,
    \begin{equation}
    \label{Eqn: support condition x-minus}
        \Phi \in E_l^\zeta \text{ implies } \mathrm{supp}\,\Phi \subset \{(x_+,x_-)\in \mathbb{R}^{2d}:|x_-| \leq 1\}.
    \end{equation}
    Then the dual casting operator applied to $F_\varepsilon$ is a bounded operator
    \begin{equation*}
        \mathcal{A}^{(n),*}_\pm (F_\varepsilon)\colon E^\zeta_{n+k} \to E^{\zeta}_k
    \end{equation*}
    on the scales $E^\zeta_k = E^\zeta_k(\mathbb{R}^{2d})$ for any $k=0,1,\ldots, N-n+1$, and the estimate
    \begin{equation}
        \|\mathcal{A}^{(n),*}_\pm (F_\varepsilon)\,\Phi\|_{E^\zeta_k} \leq C  \|F\|_{W^{N+1,\infty}(\mathbb{R}^d)} \|\Phi\|_{E^\zeta_{n+k}}
    \end{equation}
    holds for some constant $C$ that is only dependent on $n$ and $k$. 
\end{proposition}
\begin{proof}
    We split the proof into two parts. First, we give an explicit representation of the iterated tail divergence $\mathrm{div}^{\mathfrak{t},\pm}_L F_\varepsilon$ appearing in the representation \eqref{Eqn: blow-up representation dual casting} in terms of the tensor field $F$; a subsequent argument in the spirit of DiPerna--Lions commutator estimates allows us to cancel the singular $\varepsilon^{-1}$-scaling appearing from $F_\varepsilon$. Second, we take the Sobolev-norms of the resulting explicit representation inserted into \eqref{Eqn: F-eps transf explicit}, and effectively prove that a $E^\zeta_k$-restricted tail divergence norm satisfies the estimate
    \begin{equation*}
        \|F_\varepsilon\|_{\mathfrak{t},k,\zeta} \lesssim \|F\|_{W^{n+k,\infty}(\mathbb{R}^d)}
    \end{equation*}
    with the suspended constant independent of $\varepsilon$.
    
    {\textit{Part 1.}} Since $\bar{F}$ is doubled from $F$, the representation \eqref{Eqn: F-vareps-little-s} reads
    \begin{align*}
        \pi_{i_1\dots i_n} F_\varepsilon(z_1,\ldots,z_n)
        = \sum_{j_1\dots j_n} s^\varepsilon_{i_1j_1}\dots s^\varepsilon_{i_nj_n} F(\pi_{j_1}(N^\varepsilon z_1),\ldots, \pi_{j_n}(N^\varepsilon z_n)).
    \end{align*}
    At this stage, it would be convenient to have a formula that explicitly writes out the tail divergence tensor field $\mathrm{div}^{\mathfrak{t},\pm}_L G$ in terms of the binary projections $\pi_{i_1\dots i_n} G$, where $G\colon (\mathbb{R}^{2d})^n\to (\mathbb{R}^{2d})^{\otimes n}$ is any degree $n$ tensor field over $\mathbb{R}^{2d}$ and $L\in \mathrm{Asc}(1;n)$. To this end, note that the tail derivatives $D^{\mathfrak{t}}_{\pm, \ell}$ and their components $D_{\pm, \ell}^{\mathfrak{t},\alpha}$ are given by
    \begin{equation*}
        D_{\pm, \ell}^{\mathfrak{t}} = \sum_{j \geq \ell}^n(\nabla_{\pm})^{(j)}, \quad D^{\mathfrak{t},\alpha}_{\pm,\ell} =\begin{cases}
            \sum_{j\geq \ell} \partial^{+,(j)}_\alpha,& 1\leq \alpha \leq d, \\
            \sum_{j\geq \ell} \partial^{-,(j)}_{\alpha-d},& d+1 \leq \alpha \leq 2d,
        \end{cases}  
    \end{equation*}
    where $\partial^{+,(j)}_\alpha = \langle \nabla_+^{(j)}, e_\alpha\rangle $, with $\partial^{-,(j)}_\alpha$ defined similarly. This splitting into $\partial^{+,(j)}_\alpha$ and $\partial^{-,(j)}_\alpha$ allows us to further split $\mathrm{div}^{\mathfrak{t},\pm}_\ell$ into two tail divergences based on the derivatives $\nabla_+$ and $\nabla_-$. In components: 
    \begin{align*}
        \langle\mathrm{div}_\ell^{\mathrm{t},\pm} G, e_{j_1\dots \hat{j}_\ell\dots j_n}\rangle 
        &= \sum_{\alpha = 1}^d  D^{\mathfrak{t},\alpha}_{\pm,\ell} \langle G, e_{j_1\ldots \alpha\ldots j_n}\rangle  + \sum_{\alpha = d+1}^{2d}D^{\mathfrak{t},\alpha}_{\pm,\ell} \langle G, e_{j_1\ldots \alpha\ldots j_n}\rangle\\
        &= \sum_{\alpha = 1}^d \sum_{j\geq \ell} \partial^{+,(j)}_\alpha\langle G, e_{j_1\ldots \alpha\ldots j_n}\rangle + \sum_{\alpha = d+1}^{2d}\sum_{j\geq \ell} \partial^{-,(j)}_\alpha\langle G, e_{j_1\ldots \alpha\ldots j_n}\rangle.
    \end{align*}
    We identify the latter two terms as defining tail divergences $\mathrm{div}^{\mathfrak{t},+}_{\ell}$ and $\mathrm{div}^{\mathfrak{t},-}_{\ell}$, respectively. Switching to the binary representation of $G$, we have the splitting
    \begin{align*}    
        \mathrm{div}_{\ell}^{\mathfrak{t},\pm} G&= \sum_{i_1\dots\hat{i}_\ell\dots i_n} \mathrm{div}^{\mathfrak{t},+}_\ell \pi_{i_1\dots 1 \dots i_n} G + \sum_{i_1\dots\hat{i}_\ell\dots i_n}\mathrm{div}^{\mathfrak{t},-}_\ell\pi_{i_1\dots 2 \dots i_n}G \\
        &= \sum_{i_\ell = 1}^2\sum_{i_1\dots\hat{i}_\ell\dots i_n} \mathrm{div}^{\mathfrak{t},\mathrm{sgn}(i_\ell)}_\ell \pi_{i_1\dots \sigma \dots i_n} G = \sum_{i_1\dots i_n}\mathrm{div}^{\mathfrak{t},\mathrm{sgn}(i_\ell)}_\ell \pi_{i_1\dots i_n} G
    \end{align*}
    where in the last line we have written $\mathrm{sgn}(1) = +,\, \mathrm{sgn}(2) = -$. From this, it is clear that for ascending subsets $L = \{\ell_1\ldots, \ell_k\}\in \mathrm{Asc}(1;n)$
    \begin{align}
        \mathrm{div}_L^{\mathrm{t},\pm} G &= \mathrm{div}_{\ell_1}^{\mathrm{t},\pm}\cdots\mathrm{div}_{\ell_k}^{\mathrm{t},\pm} G = \sum_{\substack{i_1\dots i_n \\}}  \mathrm{div}^{\mathfrak{t},\mathrm{sgn}(i_{\ell_1})}_{\ell_1}\cdots \mathrm{div}^{\mathfrak{t},\mathrm{sgn}(i_{\ell_k})}_{\ell_k}\pi_{i_1\dots i_n} G.
    \end{align}
    This is our desired expression for $\mathrm{div}^{t,\pm}_\ell G$ in terms of binary projections. The novelty of this identity comes from the observation that $\mathrm{div}^{\mathfrak{t},-}_{i_\ell}$ acts on an anti-symmetrized component $\pi_{j_1\dots 2\dots j_n}F_\varepsilon$, which we will now show cancels the very same factor $\varepsilon^{-1}$ appearing due to the index $2$ in the binary projection of said component. A generic summand of the aforementioned expression for  $\mathrm{div}^{\mathfrak{t},\pm}_LF_\varepsilon$ will, schematically, have components of the form
    \begin{align*}
         \sum_{\alpha_{\ell_1},\ldots, \alpha_{\ell_k}}\partial^{\mathrm{sgn}(i_{\ell_1}),(\gamma_1)}_{\alpha_{\ell_1}}\cdots \partial^{\mathrm{sgn}(i_{\ell_k}),(\gamma_k)}_{\alpha_{\ell_k}}\langle \pi_{i_1\dots i_n}F_\varepsilon, e_{j_1\dots \alpha_{\ell_1}\ldots \alpha_{\ell_k}\dots j_n } \rangle
    \end{align*}
    where the variable indices $\gamma_1,\ldots,\gamma_k$ satisfy $\gamma_1\leq \cdots \leq \gamma_k$ with $\gamma_r \geq r$ for each $r=1,\ldots,k$, and the binary indices $i_1\ldots i_n$ have $(i_\ell)_{\ell\in L}$ fixed. Using the definition of $F_\varepsilon$ and the identity $(\mathcal{S}^{\varepsilon})^{-1}(x_+,x_-)^T= N^\varepsilon (z)$, we have by repeated use of the chain rule that
    \begin{align*}
        &\sum_{\alpha_{\ell_1},\ldots, \alpha_{\ell_k}}\partial^{\mathrm{sgn}(i_{\ell_1}),(\gamma_1)}_{\alpha_{\ell_1}}\cdots \partial^{\mathrm{sgn}(i_{\ell_k}),(\gamma_k)}_{\alpha_{\ell_k}}\langle \pi_{i_1\dots i_n}F_\varepsilon, e_{j_1\dots \alpha_{\ell_1}\ldots \alpha_{\ell_k}\dots j_n }\rangle \\
        &= \sum_{\alpha_{\ell_1},\ldots, \alpha_{\ell_k}}\sum_{h_1\dots h_n} s^\varepsilon_{i_1h_1}\cdots s^\varepsilon_{i_nh_n}(s^\varepsilon)^{-1}_{h_{\ell_1}i_{\ell_1}}\cdots (s^\varepsilon)^{-1}_{h_{\ell_k}i_{\ell_k}} \partial_{\alpha_{\ell_1}}^{(\gamma_1)}\cdots \partial_{\alpha_{\ell_k}}^{(\gamma_k)}\langle F(\pi_{h_1\dots h_n}(N^\varepsilon\mathbf{z})), e_{j_1\dots \alpha_{\ell_1}\ldots \alpha_{\ell_k}\dots j_n} \rangle.
    \end{align*}
    Observe that the final expression is simply a weighted sum of components of the tail divergences $\mathrm{div}^{\mathfrak{t}}_{L} F(\pi_{h_1\ldots h_n}(\mathbf{z}))$. Summing over the binary indices $i_1\dots i_n$, we have just proven the component representation of the identity 
    \begin{equation}
    \label{Eqn: div_pm to div identity}
        \mathrm{div}^{\mathfrak{t},\pm}_L F_\varepsilon = \sum_{i_1\dots i_n} \sum_{h_1\cdots h_n} \prod_{k\notin L}s^\varepsilon_{i_kh_k} \prod_{\ell\in L} c_{i_\ell h_{\ell}} \mathrm{div}^{\mathfrak{t}}_L F(\pi_{h_1\dots h_n}(N^\varepsilon \mathbf{z}))
    \end{equation}
    where $c_{i_{\ell}h_{\ell}}$ denote elements of the matrix
    \begin{equation*}
        c = \frac{1}{2}\begin{pmatrix}
            1 & 1 \\
            1 & 1
        \end{pmatrix}.
    \end{equation*}
    
    Continuing, let the binary indices $i_1\dots i_n$ in a given summand of \eqref{Eqn: div_pm to div identity} be fixed, and let $m$ be the number of indices equal to $2$. By definition, the term $\pi_{i_1\dots i_n}F_\varepsilon$ has a singular prefactor $\varepsilon^{-m}$, and we have just shown that the tail divergences help cancel a factor $\varepsilon^{-|L|}$ by the chain rule; the remaining singular scaling $\varepsilon^{-m+|L|}$ has to be controlled using the regularity of $F$, which we now demonstrate. Let $L'\subseteq L^c$ consist of all $\ell'$ for which $i_{\ell'} = 2$, and note that $c_{i_{k}h_{k}} = s^\varepsilon_{i_kh_k}$ whenever $k\in (L')^c$. Using Taylor's formula, we rid ourselves of the remaining  singularly scaled difference quotients by
    \begin{align}
    \label{Eqn: part 1 explicit representation}
    \begin{split}
        \mathrm{div}^{\mathfrak{t},\pm}_L F_\varepsilon &=\sum_{h_1\dots h_n} \prod_{\ell'\in L'} s^\varepsilon_{i_{\ell'}h_{\ell'}}\prod_{k \in (L')^c} c_{i_k h_k} \mathrm{div}^{\mathfrak{t}}_L F(\pi_{h_1\dots h_n}(N^\varepsilon\mathbf{z})) \\
        &=\sum_{h_1\dots h_n}\prod_{k\in (L')^c}\frac{1}{2^{l}}\int_{[-1,1]^{l}}c_{i_kh_k}\nabla^{(L')}\mathrm{div}^\mathfrak{t}_L F(C^\varepsilon_{h_1\dots h_n}(\mathbf{z};\theta)) :  (x_-)^{L'}\,\mathrm{d}\theta_1\cdots\mathrm{d}\theta_{l}
    \end{split}
    \end{align}
    where we have written $|L'| = l$, $\nabla^{(L')} = \nabla^{(\ell'_1)}\cdots \nabla^{(\ell'_l)}$, $(x_-)^{L'} = (x_{\ell'_1,-})\otimes \cdots \otimes (x_{\ell'_l,-})$, and the coordinate maps $C^\varepsilon_{h_1\dots h_n}(\mathbf{z};\theta)\colon (\mathbb{R}^{2d})^n \to (\mathbb{R}^d)^n$ have their $i$-th slot determined by
    \begin{equation*}
        [C^\varepsilon_{h_1\dots h_n}(\mathbf{z};\theta)]_i = \begin{cases}
            \pi_{h_i}(N^\varepsilon z_i), & \text{if } i\not\in L', \\
            x_{i,+} + \theta_i \varepsilon x_{i,-}, &\text{if } i\in L'.
        \end{cases}
    \end{equation*}
    The final expression in \eqref{Eqn: part 1 explicit representation} is our desired explicit representation.
    
    {\textit{Part 2.}} Let $\Phi\in E^\zeta_{n+k}(\mathbb{R}^{2d})$. Using the representation \eqref{Eqn: F-eps transf explicit}, we deduce that
    \begin{equation}
        \label{Eqn: part 2 first estimate}
        \|\mathcal{A}^{(n),*}_z(F_\varepsilon)\,\Phi\|_{E^\zeta_k(\mathbb{R}^{2d})} \leq \sum_{L\in \mathrm{Asc}(1;n)} \|\mathrm{Tr}_{x_\pm}\mathrm{div}^{\mathfrak{t},\pm}_L F_\varepsilon : \nabla^{n-|L|}_{\pm} \Phi\|_{E^\zeta_k(\mathbb{R}^{2d})}.
    \end{equation}
    In the above contraction, take the tensor product $(x_-)^{L'}$ appearing in the representation \eqref{Eqn: part 1 explicit representation} and move it to $\nabla^{n-|L|}_\pm \Phi$; we may harmlessly bound each factor $|x_{\ell'_{j},-}| \leq 1$ due to the support condition \eqref{Eqn: support condition x-minus} provided by $\Phi$. Continuing, the remaining tensor fields in the contraction \eqref{Eqn: part 2 first estimate} allow us to estimate
    \begin{align*}
        \|\mathcal{A}^{(n),*}_z(F_\varepsilon)\,\Phi\|_{E^\zeta_k(\mathbb{R}^{2d})} &\leq \sum_{L\in \mathrm{Asc}(1;n)} \|\mathrm{Tr}_{x_\pm}\mathrm{div}^{\mathfrak{t},\pm}_L F_\varepsilon : \nabla^{n-|L|}_{\pm} \Phi\|_{E^\zeta_k(\mathbb{R}^{2d})} \\
        &\lesssim \sum_{L\in \mathrm{Asc}(1;n)} \|\nabla^{(L')}\mathrm{div}^{\mathfrak{t}}_L F\|_{W^{k,\infty}(\mathbb{R}^d)} \|\Phi\|_{E^\zeta_{n+k}(\mathbb{R}^{2d})} \\
        & \lesssim \sum_{L\in \mathrm{Asc}(1;n)} \|\mathrm{div}^{\mathfrak{t}}_L F\|_{W^{n+k-|L|,\infty}(\mathbb{R}^d)} \|\Phi\|_{E^\zeta_{n+k}(\mathbb{R}^{2d})}\\
        &\lesssim \|F\|_{W^{n+k,\infty}(\mathbb{R}^d)}\,\|\Phi\|_{E^\zeta_{n+k}(\mathbb{R}^{2d})}
    \end{align*}
    where we have used $|L'| \leq n-|L|$. The assumed regularity of $F$ furthermore implies that we can bound $\|F\|_{W^{n+k,\infty}(\mathbb{R}^d)} \leq \|F\|_{W^{N+1,\infty}(\mathbb{R}^d)}$ for any $0\leq k\leq N-n+1$. 
\end{proof}

\subsection{Tensor field rough paths} \label{sec:sobolev RP}
In this subsection, we define the concept of an admissible tensor field rough path which will realize an appropriate notion of a rough path lift of vector fields that are of admissible regularity for our setting. 

The following formalism takes inspiration from \cite{BDFT2021}, but we avoid working in the dual Hopf algebra; see Remark \ref{Rem: algebraic dual Hopf alg} below.
Denote by $\mathrm{T}^{n,k}(\mathbb{R}^d)$ the space of degree $k$ tensor fields $F\colon (\mathbb{R}^d)^n\to (\mathbb{R}^d)^{\otimes k}$ over $\mathbb{R}^d$. Naturally, we define the tensor product of two tensor fields $ F = F(x_1,\ldots,x_n)\in \mathrm{T}^{n,k}(\mathbb{R}^d)$ and $G = G(y_1,\ldots,y_m)\in \mathrm{T}^{m,l}(\mathbb{R}^d)$ by the formal concatenation
\begin{equation*}
    \langle(F\otimes G)(x_{1},\ldots, x_n, y_1,\ldots,y_m), e_{i_1\dots i_{k+l}}\rangle = \langle F (\mathbf{x}),e_{i_1\dots i_k}\rangle \langle G(\mathbf{y}), e_{i_{k+1}\dots i_{k+l}}\rangle
\end{equation*}
which is again a tensor field over $\mathbb{R}^d$; $F\otimes G \in \mathrm{T}^{n+m,k+l}(\mathbb{R}^d)$. On the path to defining tensor field rough paths, we shall only consider tensor fields whose arity $n$ and tensor degree $k$ coincide. This gives rise to the {\it tensor field algebra} $\mathcal{T}(\mathbb{R}^d)$ of tensor fields over $\mathbb{R}^d$, defined as the direct sum
\begin{equation}
\label{Eqn: direct sum tensor field algebra}
    \mathcal{T}(\mathbb{R}^d) := \bigoplus_{n = 0}^\infty \mathrm{T}^{n,n}(\mathbb{R}^d)
\end{equation}
where $\mathrm{T}^{0,0} := \mathbb{R}\mathbf{1}$ for some generating element $\mathbf{1}$. Elements $\mathbf{F} \in \mathcal{T}(\mathbb{R}^d)$ are written as a formal series $\mathbf{F} = (\mathbf{1}, F^{(1)},F^{(2)},\ldots)$, with homogeneous components $F^{(n)}\in \mathrm{T}^{n,n}$. Since the tensor product $F\otimes G$ of tensor fields $F$ and $G$ plays the role of a concatenation product, we may well consider its corresponding {\it unshuffle coproduct} defined as the extension from
\begin{equation}
\label{Eqn: unshuffle coproduct def}
    \Delta F^{(n)} = \sum_{k+l = n} \sum_{\sigma\in S_{k,l}} P_\sigma F^{(n)}(x_{\sigma(i_1)},\ldots,x_{\sigma(i_n)})
\end{equation}
for homogeneous elements $F^{(n)}\in \mathrm{T}^{n,n}$, where the tensorial shuffle $P_\sigma$ only acts on the underlying basis; its action on pure tensors reads
\begin{equation*}
    P_\sigma (e_{i_1}\otimes e_{i_2} \otimes \cdots  \otimes e_{i_{n+m}}) = e_{\sigma(i_1)}\otimes e_{\sigma(i_2)}\otimes\cdots \otimes e_{\sigma(i_{n+m})}.
\end{equation*}
One can show the compatibility condition
\begin{equation}
    \Delta(\mathbf{F}\otimes \mathbf{G}) = (\Delta \mathbf{F})\otimes (\Delta \mathbf{G})
\end{equation}
which turns $(\mathcal{T},\otimes, \Delta)$ into a graded connected bialgebra; moreover, it is well-known (see e.g.~\cite{Manchon2006}) that this implies the existence of an antipode map $S\colon \mathcal{T}\to \mathcal{T}$, formally defined on homogeneous components by
\begin{equation}
    \label{Eqn: antipode S def}
    S F^{(n)}(x_1,\dots, x_n) = (-1)^n P_{\mathrm{rev}}\,F^{(n)}(x_n,\ldots, x_1), 
\end{equation}
where, as before, $P_{\mathrm{rev}}$ reverses the tensor basis order, turning $(\mathcal{T},\otimes,\Delta,S)$ into a Hopf algebra. We will use the antipode map in a later discussion about the conservativity of unbounded rough drivers.

The {\it step-$N$ truncated tensor field algebra $\mathcal{T}^{(N)}(\mathbb{R}^d)$} is then defined by truncating the joint grading $n\leq N$ in arity and tensor degree:
\begin{equation}
\label{Eqn: truncated tensor field algebra}
    \mathcal{T}^{(N)}(\mathbb{R}^d) := \bigoplus_{\substack{n}=0}^{N} \mathrm{T}^{n,n}(\mathbb{R}^d).
\end{equation}

\begin{remark}
\label{Rem: algebraic dual Hopf alg}
    In \cite{BDFT2021}, the authors formalize rough paths as being elements $\mathbf{W} \in \mathcal{H}^*$, where $\mathcal{H}^*$ is the dual of a shuffle-deconcatenation Hopf algebra $\mathcal{H} = (H, \shuffle, \Delta_{\mathrm{dec}},S)$. We illustrate briefly why this dual construction is ill-suited for our purposes. Define the {\it shuffle product} $\shuffle \colon \mathcal{T} \otimes \mathcal{T} \to \mathcal{T}$ as the extension of $\shuffle \colon \mathrm{T}^{n,n}\otimes \mathrm{T}^{m,m} \to \mathrm{T}^{n+m,n+m}$ defined on homogeneous elements $F^{(n)} \in \mathrm{T}^{n,n}$ and $G^{(m)}\in \mathrm{T}^{m,m}$ by
    \begin{equation}
    \label{Eqn: shuffle product}
        F^{(n)} \shuffle G^{(m)} :=  \sum_{\sigma\in S_{n,m}} P_\sigma (F^{(n)}\otimes G^{(m)})(x_{\sigma(1)},\ldots,x_{\sigma(n+m)}).
    \end{equation}
    Moreover, one would like to introduce the {\it deconcatenation coproduct} $\Delta_{\mathrm{dec}}\colon \mathcal{T}\to \mathcal{T}\otimes \mathcal{T}$ corresponding to the above shuffle product, extended from
    \begin{equation}
    \label{Eqn: deconcatenation coproduct}
        \Delta_{\mathrm{dec}} F^{(n)} = \sum_{i = 0}^n F^{(n)}_{[1:i]}\otimes F^{(n)}_{[i+1:n]}
    \end{equation}
    for $F^{(n)} = F^{(n)}(x_1,\ldots,x_n) \in\mathrm{T}^{n,n}$, where both $F^{(n)}_{[1:i]} = F^{(n)}_{[1:i]}(x_1,\ldots,x_i) \in \mathrm{T}^{i,i}$ and $F^{(n)}_{[i:1+n]} = F^{(n)}_{[i:1+n]}(x_{i+1},\ldots, x_n) \in \mathrm{T}^{n-i,n-i}$ are tensor fields whose components and arguments are cut from $F^{(n)}$ as indicated. Herein lies the first problem: this suggested coproduct cannot, in general, reconstruct non-decomposable tensor fields $F$. Hence, we do not have a well-defined coproduct unless we introduce a suitable completion (if it even exists).
    
    The second problem arises when trying to define the dual Hopf algebra. Truncate the above (completed) tensor field algebra jointly in arity and tensor degree to produce $\mathcal{T}^{(N)}(\mathbb{R}^d)$. If we wanted to turn the infinite dimensional algebraic dual $\mathcal{T}^{(N),*}$ into a graded bialgebra $(\mathcal{T}^{(N),*}(\mathbb{R}^d),*, m^*)$, where $*$ is the convolution product and $m^*$ is the dual of the shuffle product
    \begin{equation}
        m^*(\varphi)(\mathbf{F}\otimes \mathbf{G}) := \varphi(\mathbf{F}\shuffle \mathbf{G}), \quad \varphi \in \mathcal{T}^{(N),*}
    \end{equation}
    then $m^*\colon \mathcal{T}^{(N),*} \to (\mathcal{T}^{(N)}\otimes \mathcal{T}^{(N)})^*\supsetneq \mathcal{T}^{(N),*}\otimes \mathcal{T}^{(N),*}$, hence the closure of the coproduct necessarily fails. Indeed, this problem relates to the infinite-dimensional nature of the tensor field Hopf algebra; the bilinear form $m^*(\varphi)$ need not have finite rank. 
\end{remark}

We consider the rough vector field $\mathrm{X}$ as a path with values in a space of vector fields, $V \subseteq  \mathrm{T}^{1,1}$, so that
$$
\mathrm{X}\colon [0,T] \rightarrow V.
$$
In this setting, a {\it rough path lift} $\bX$ of the vector field $\mathrm{X}$ can be regarded as a mapping 
$$
\bX \colon \Delta_T^{(2)} \rightarrow V^{(N)},
$$
where $V^{(N)}$ is a subalgebra of $(\mathcal{T}^{(N)}(\mathbb{R}^d),\otimes, \Delta,S)$ such that elements of degree $1$ are contained in $V$. The homogeneous components of $\mathbf{X}$ are denoted by $\mathrm{X}^{(n)} \in\mathrm{T}^{n,n}(\mathbb{R}^d)$; in the sequel, we write $\mathbf{X} = (\mathrm{X}^{(n)})_{n=0}^N$ with $\mathrm{X}^{(0)} \equiv \mathbf{1}$, and we also identify $\mathrm{X}^{(1)} \equiv \mathrm{X}$ for the lift. When the underlying path $\mathrm{X}$ is smooth, the $n$-th homogeneous component is the object canonically defined through iterated integrals
\begin{equation} 
\label{eq:signature of X smooth case}
   \mathrm{X}_{st}^{(n)}(x_1, \dots, x_n)  = \int_{\Delta_{st}^{(n)}} \dot{\mathrm{X}}_{r_1}(x_1) \otimes \dots  \otimes \dot{\mathrm{X}}_{r_n}(x_n) \,\mathrm{d}r_1 \dots \mathrm{d}r_n.    
   \end{equation}
From \eqref{eq:signature of X smooth case}, one can plainly see that Chen's relation for such a rough path lift has to shuffle the argument variables of the tensor fields; subdividing the iterated integral domain into lower-order simplices rearranges the tensor factors, labelled $(x_j,r_j)$, according to the simultaneous shuffle described by \eqref{Eqn: shuffle product}. Also, from \eqref{eq:signature of X smooth case}, we see that it is natural to impose a spatial regularity on $\mathrm{X}^{(n)}$ that at least partially mirrors the spatial regularity of each tensor factor; in our case, we also impose a sufficient condition for castability. These considerations motivate the following definition of an admissible tensor field rough path lift of $\mathrm{X}$.

\begin{definition}[Tensor field rough paths]
\label{def:sobolev rough path}
    Let $\mathfrak{p} \geq 1$ and let $N  = \lfloor\mathfrak{p}\rfloor$ denote its integer value. We call a mapping
    $$
    \bX \colon \Delta_T^{(2)} \rightarrow \bigoplus_{n=0}^{N} W^{N,\infty}((\R^d)^n;(\R^d)^{\otimes n}) \subseteq \mathcal{T}^{(N)}(\mathbb{R}^d)
    $$
    a $\mathfrak{p}$-variation {\it admissible tensor field rough path}, henceforth simply {\it tensor field rough path}, over $\mathbb{R}^d$ provided
    \begin{itemize}
        \item[(i)] the tail divergence norm $\|\mathrm{X}^{(n)}_{st}\|_{\mathfrak{t},N-n+1}$ is finite; moreover, there exists a continuous control $w_{\bX}$ such that
            \begin{equation} 
            \label{eq:sobolev RP regularity}
            \|\mathrm{X}^{(n)}_{st}\|_{W^{N,\infty}(\mathbb{R}^d)} + \|\mathrm{X}^{(n)}_{st}\|_{\mathfrak{t},N-n+1} \leq w_{\bX}(s,t)^{n/\mathfrak{p}}
            \end{equation}
            for any $(s,t) \in \Delta_T^{(2)}$ and $n=1,\dots, N$, and
        \item[(ii)] {\it Chen's relation} is satisfied; for any $n=1,\ldots, N$, at the $n$-th level there holds
            \begin{equation} \label{eq:sobolev RP chen}
                \mathrm{X}^{(n)}_{st} = \sum_{l=0}^n \mathrm{X}_{sr}^{(n-l)} \otimes \mathrm{X}_{rt}^{(l)}
            \end{equation}    
            for any $(s,r,t) \in \Delta_T^{(3)}$.
    \end{itemize}
    If conditions (i)--(ii) as above are satisfied for $\mathbf{X}$, we write $\mathbf{X} \in \mathcal{C}^\mathfrak{p}\mathcal{W}^N(\mathbb{R}^d)$. To emphasize the spatial regularity, we write $\mathbf{X}_{st} \in \mathcal{W}^{N}(\mathbb{R}^d)$ for any pair $(s,t)\in \Delta^{(2)}_T$, where we have defined the regularity class
    \begin{equation}
        \mathcal{W}^{N}((\mathbb{R}^d)^n;(\mathbb{R}^d)^{\otimes n}) := \{ F\colon (\mathbb{R}^d)^n \to (\mathbb{R}^d)^{\otimes n} : \|F\|_{W^{N,\infty}} + \|F\|_{\mathfrak{t},N-n+1}<\infty\},
    \end{equation}
    where, again, we simply write $\mathcal{W}^{N}(\mathbb{R}^d) = \mathcal{W}^{N}((\mathbb{R}^d)^n;(\mathbb{R}^d)^{\otimes n})$ whenever the arity and degree of tensor fields are contextually implicit.
\end{definition}

The admissibility condition \eqref{eq:sobolev RP regularity} provides a sufficient condition for the castability of such tensor field rough paths $\mathbf{X}$: $\mathcal{A}^{(n),*}_x(\mathrm{X}^{(n)}_{st})$ is bounded as a linear operator on an appropriate localized scale of spaces by Lemma \ref{Lem: boundedness 1}. Inspecting Definition \ref{Def: tail divergence norm}, one is easily convinced that if each homogeneous component has increased regularity
\begin{equation*}
    \mathbf{X}\colon \Delta^{(2)}_{T} \to \bigoplus_{n=0}^{N} W^{N+1,\infty}((\R^d)^n;(\R^d)^{\otimes n}) \subseteq \mathcal{T}^{(N)}(\mathbb{R}^d),
\end{equation*}
then $\mathbf{X}$ is castable due to the coarse estimate $\|\mathrm{X}^{(n)}_{st}\|_{W^{N,\infty}(\mathbb{R}^d)}+\|\mathrm{X}^{(n)}_{st}\|_{\mathfrak{t},N-n+1} \leq \|\mathrm{X}^{(n)}_{st}\|_{W^{N+1,\infty}(\mathbb{R}^d)}$.

\begin{remark}
    Note that we are at this point not specifying whether the rough path $\bX$ is geometric or branched. As observed in the expansion procedure laid out in the introduction, the information contained in the truncated tensor field algebra $\mathcal{T}^{(N)}(\mathbb{R}^d)$ is enough to describe the expansion of the solution $u$ in $\bX$, which comes from the fact that our equation is linear. This is in contrast to the Lagrangian point of view, where the richer Connes--Kreimer Hopf-algebraic structure is needed to describe the expansion; see \cite{GUBINELLI2010693, HairerKelly}. 
    Note that the correspondence between Eulerian coordinates $u$ and Lagrangian coordinates $\phi$ given by $u_t(x) = u_0(\phi_t^{-1}(x))$ can only be expected to hold when the driving noise is geometric.
\end{remark}

We shall need the notion of \emph{geometric} rough paths, essentially conditions on $\bX$ which ensures that $\bX$ (and the solution $u$) are amenable to the usual chain rule. 
As in the finite-dimensional case we shall distinguish between \emph{weak} and \emph{strong} geometricity.

\begin{definition} \label{def:weak geometric}
    A tensor field rough path $\bX\in \mathcal{C}^\mathfrak{p}\mathcal{W}^N(\mathbb{R}^d)$ is said to be \emph{weak geometric} provided the group-like property
    \begin{equation}
        \Delta \mathbf{X}_{st} = \mathbf{X}_{st}\otimes \mathbf{X}_{st}
    \end{equation}
    holds in the step-$N$ truncated tensor field algebra for all $(s,t)\in \Delta^{(2)}_T$. Equivalently, for all $n+m \leq N$, one has the component identity
    \begin{equation} \label{eq:weak geometric}
        \mathrm{X}_{st}^{(n)} \otimes \mathrm{X}_{st}^{(m)}(x_1, \dots, x_{n+m}) = \sum_{\sigma \in S_{n,m}} P_{\sigma} \mathrm{X}_{st}^{(n+m)} (x_{\sigma(1)}, \dots, x_{\sigma(n+m)}).
    \end{equation}
    The space of weak geometric rough paths is denoted $\mathcal{C}^\mathfrak{p}_{\mathrm{wg}}\mathcal{W}^N(\mathbb{R}^d)$.
\end{definition}
As a simple example, when $n=m=1$ equation \eqref{eq:weak geometric} reads
$$
\mathrm{X}_{st}^{(1)}(x_1) \otimes  \mathrm{X}_{st}^{(1)}(x_2)  = \mathrm{X}_{st}^{(2)}(x_1,x_2)  + \mathrm{X}_{st}^{(2)}(x_2,x_1)^T,
$$
where the final term is transposed in the matrix sense.

\begin{definition} \label{def:strong geometric}
    A tensor field rough path $\bX$ is said to be \emph{strong geometric} provided there exists a family $(\mathrm{X}^{\varepsilon})_{\varepsilon > 0}$ of continuous paths of bounded variation 
    \begin{equation}
        \label{Eqn: paths_bnd_var}
        \mathrm{X}^{\varepsilon}\colon[0,T] \rightarrow \clW^{N}(\R^d)    
    \end{equation}
    such that when we define the canonical lift
    $$
    \mathrm{X}^{\varepsilon,(n)}_{st} (x_1, \dots, x_n) : =  \int_{\Delta_{st}^{(n)}} \dot{\mathrm{X}}^{\varepsilon}_{r_1}(x_1) \otimes \dots \otimes \dot{\mathrm{X}}^{\varepsilon}_{r_n}(x_n) \,\mathrm{d}r_1 \dots \mathrm{d}r_n
    $$
    we have
    $$
    \la \mathrm{X}^{\varepsilon,(n)} - \mathrm{X}^{(n)} \ra_{\frac{\mathfrak{p}}{n}-\textrm{var}; \clW^{N}(\mathbb{R}^d)} \rightarrow 0
    $$
    as $\varepsilon \rightarrow 0$. The space of strong geometric rough paths is denoted $\mathcal{C}^\mathfrak{p}_{\mathrm{g}}\mathcal{W}^N(\mathbb{R}^d)$.
\end{definition}
\begin{remark}
    We warn that the paths $\mathbf{X}^\varepsilon$ as in \eqref{Eqn: paths_bnd_var} should not be confused with the $\varepsilon$-transformed paths as encountered in Section \ref{sec:renormalization}.
\end{remark}

It is straightforward to check that a strong geometric path $\bX$ is also weak geometric. It is well-known that in the finite-dimensional setting the two notions agree up to tuning the temporal regularity $\mathfrak{p}$ (see \cite{FV2010}). When $\mathfrak{p} \in [2,3)$ a similar result holds also in the infinite-dimensional setting, see \cite{GRONG2022151}, albeit requiring much higher spatial regularity than in the present paper. To the best of the authors' knowledge, no such results exist in the literature for $\mathfrak{p} \in [3,\infty)$. Thus, both notions will be needed for the present paper; we shall need a weak geometric $\bX$ to prove that $u^2$ satisfies a similar equation as $u$ (that is, the chain-rule is valid), thus leading to uniqueness. For our existence result in Section \ref{sec:existence}, we shall instead use an approximation procedure which requires a strong geometric $\bX$. 

We can also conditionally specialize our notion of rough path to that of finite-dimensional rough paths. Assume $\bX$ is a tensor field rough path. Fix $x \in \R^d$ and consider the tensor-valued map
\begin{equation*}
    \mathbf{X}(x)\colon \Delta^{(2)}_T \to \bigoplus_{n=0}^N (\R^d)^{\otimes n}
\end{equation*}
defined via the diagonal trace at $x$ of $\mathbf{X} = (\mathrm{X}^{(n)})_{n=1}^N$ with 
$$
\bX_{st}(x)|_{(\R^d)^{\otimes n}} := \mathrm{X}_{st}^{(n)}(x, \dots, x) = \mathrm{Tr}_{x} \,\mathrm{X}_{st}^{(n)}.
$$
It is straightforward to check that the mapping $(s,t) \mapsto \bX_{st}(x)$ defines a rough path over $\mathbb{R}^d$ that agrees with the finite-dimensional, $\mathfrak{p}$-variation notion of rough paths as encountered in e.g. \cite{FV2010}.

\begin{example}
    Consider a probability space $(\Omega, \clF,\bbP)$ which supports a Gaussian noise vector field $\mathrm{X} \colon [0,T] \times \R^d \times \Omega \rightarrow \R^d$ with covariance given by
    $$
    \bbE[\mathrm{X}_t(x)^j\, \mathrm{X}_s(y)^i] = R(t,s) \,Q(x,y)^{ji}
    $$
    for $R\colon [0,T]^2 \rightarrow \R$ and $Q \colon \R^d \times \R^d \rightarrow (\R^d)^{\otimes 2}$.
    We will make the following assumptions on $R$ and $Q$. 

    \noindent
    \textbf{Time correlation}: We make the assumption that the time correlation is such that an i.i.d. Gaussian family of stochastic processes with this time correlation admits a rough path lift as described in \cite{FV2010}. More specifically, define the 4-parameter function 
    $$
    \tilde{R} \left(
    \begin{array}{cc}
    s & t \\
    \bar{s} & \bar{t} \\
    \end{array}
    \right) := R(\bar{t},t) - R(\bar{s},t) - R(\bar{t},s) + R(\bar{s},s)
    $$
    and assume the $2$-dimensional $\varrho$-variation satisfies $[\tilde{R}]_{\varrho; [s,t]} \lesssim |t-s|^{1/\varrho}$ for some $\varrho \in [1,2)$. Here, we define the 2-dimensional $\varrho$-variation by 
    $$
    [\tilde{R}]_{\varrho; [s,t]}  = \sup_{\pi, \bar{\pi} \in \Pi([s,t])}  \left( \sum_{\substack{(t_i,t_{i+1}) \in \pi \\ 
    (\bar{t}_i, \bar{t}_{i+1}) \in \bar{\pi}} }  
    \left|\, \tilde{R}
    \begin{pmatrix}
        t_j & t_{j+1} \\
        \bar{t}_j & \bar{t}_{j+1} \\    
    \end{pmatrix} \right|^{\varrho} \right)^{1/\varrho}.
    $$

    \noindent
    \textbf{Spatial correlation}: For notational simplicity we will write $H^{\gamma} := W^{\gamma,2}(\R^d;\R^d)$ and we denote by $\langle \cdot, \cdot \rangle_{\gamma}$ denotes canonical inner product. Assume the induced operator $\clQ$ defined by
    $$
    \clQ f(x) := \int_{\R^d} Q(x,y) \,f(y)\, \mathrm{d}y
    $$ 
    is positive, symmetric, trace class and diagonalizable on $H^{\gamma}$. Thus, there exists an orthonormal basis $\{f_l\}_{l \geq 1}$ of $H^{\gamma}$ such that 
    $$
    \clQ f_l = q_l f_l, \qquad \textrm{Tr}(Q) := \sum_{l=1}^{\infty}  q_l < \infty.
    $$
    We next define 
    $$
    B_t^l := \frac{1}{\sqrt{q_l}} \langle \mathrm{X}_t(\cdot), f_l \rangle_{\gamma}.
    $$
    It follows that $\{B^l\}_{l \geq 1}$ is a sequence of i.i.d. Gaussian processes such that
    $$
    \bbE[B^l_t B^k_s]  = \delta_{l=k} R(t,s)
    $$
    and that we have the cylindrical decomposition 
    $$
    \mathrm{X}_t(x) = \sum_{l=1}^{\infty} \sqrt{q_l} B_t^l f_l(x).
    $$
    The above sum converges in $L^2(\Omega; H^{\gamma})$ since we have
    $$
    \bbE[ \|\mathrm{X}_t \|_{\gamma}^2] = R(t,t) \sum_l q_l < \infty. 
    $$
    Next, define the iterated integrals, for $n=2,3$
    $$
    \mathrm{X}_{st}^{(n)}(x_1,...,x_n) = \sum_{|w| = n} \sqrt{q_{w}}\, \bB_{st}^{w} \,f_{w} (x_1, ... , x_n) 
    $$
    where we adopt the word notation $q_{w} =q_{l_1} ... q_{l_n}$ for words $w=l_1 ... l_n$ in $\bbN$. Above, $\bB_{st}^{w}$ denotes the (geometric) iterated integrals of $\{B^l\}_{l \geq 1}$. Note that the above sum converges in $L^2(\Omega; (H^{\gamma})^{\otimes n})$ since we have
    $$
    \bbE\big[ \| \mathrm{X}_{st}^{(n)} \|^2_{(H^{\gamma})^{\otimes n}} \big] = \sum_{|w| = n} q_{w} \,\bbE[ |\bB_{st}^{w}|^2 ] \lesssim |t-s|^{2n/\varrho} \big(\textrm{Tr}(Q) \big)^n
    $$
    where we have used that $\bbE[ |\bB_{st}^{w}|^2 ]  \lesssim |t-s|^{2n/\varrho}$.

    Using equivalence of moments in the Gaussian chaos decomposition and the Kolmogorov continuity theorem for rough paths, we get that for each $\varepsilon > 0$ there exists a random variable $C$ such that
    $$
    \| \mathrm{X}_{st}^{(n)} \|_{(H^{\gamma})^{\otimes n}} \leq C|t-s|^{\frac{1}{2 \varrho} - \varepsilon}.
    $$

    For $\varrho \in [1, \frac32)$ and $\varepsilon$ small enough, we have $\mathfrak{p} \in [2,3)$. Using the Sobolev embedding theorem with $\gamma > 3 + \frac{d}{2}$ we get that $\bX$ satisfies \eqref{eq:sobolev RP regularity} since $\|\mathrm{X}^{(n)}\|_{W^{N+1,\infty}(\mathbb{R}^d)}<\infty$. Likewise for $\varrho \in [\frac32, 2)$ and $\varepsilon$ small enough, we have $\mathfrak{p} \in [3,4)$, and Sobolev embedding with $\gamma > 4 + \frac{d}{2}$ yields that $\bX$ satisfies \eqref{eq:sobolev RP regularity}.

    Using that the signature $\{\bB^w\}$ satisfies Chen's relation,
    $$
    \bB_{st}^w = \sum_{uv = w} \bB_{sr}^u \bB_{rt}^v
    $$
    it is straightforward to check that \eqref{eq:sobolev RP chen} also holds.
\end{example}

\begin{remark}
    To avoid technical details, we did not really complete the proof that the above example gives the correct pathwise time-regularity. To fill this gap, one can either employ the Kolmogorov continuity theorem applied to $G^3(H^{\gamma})$ -- the step-3 nilpotent Lie group over $H^{\gamma}$ -- or extend the rough path Kolmogorov continuity theorem \cite[Theorem 3.1]{FH2020} to level 3 rough paths. 
\end{remark}

\subsubsection{Doubled tensor field rough paths} 
\label{sec:doubling rough path}

It is well known that given two $\R^d$-valued paths $\mathrm{Y}^1$ and $\mathrm{Y}^2$ with lifts $\bY^1$ and $\bY^2$ respectively, there is no canonical way of defining a joint rough path lift of $(\mathrm{Y}^1,\mathrm{Y}^2)$ unless some regularity in time is assumed.
However, given a tensor field valued rough path $\bX$ and two spatial points $x,y \in \R^d$, there is a canonical way of defining the rough path lift of the path $t \mapsto (\mathrm{X}_{t}(x),\mathrm{X}_{t}(y))$. In fact, we shall define a new tensor field rough path $\bar{\bX}$ which is the rough path lift of the doubled vector field
\begin{equation}
\label{Eqn: joint rough path X bar}
    \bar{\mathrm{X}}_{t} \colon \R^d \times \R^d \rightarrow \R^{2d}, \quad \bar{\mathrm{X}}_{t}(x,y) := \begin{pmatrix}
    \mathrm{X}_t(x) \\
    \mathrm{X}_t(y)
\end{pmatrix}.    
\end{equation}
Towards this end, recall our definitions of binary projections and doubled tensor fields from Section \ref{sec:notation}. We postulate that the following definition of $\bar{\mathbf{X}}$ gives a canonical tensor field rough path lift of \eqref{Eqn: joint rough path X bar}.
\begin{definition}[The DiPerna--Lions Lift]
\label{Def: DiPerna-Lions lift}
    Given a tensor field $\mathfrak{p}$-rough path $\mathbf{X}$ over $\mathbb{R}^d$, we define its {\it DiPerna--Lions lift} $\bar{\mathbf{X}}_{st}\in \mathcal{C}^\mathfrak{p}\mathcal{W}^{N}(\mathbb{R}^{2d})$ as the tensor field rough path over $\mathbb{R}^{2d}$ by its binary projections
    \begin{equation}
    \label{Eqn: X bar DiPerna Lions lift}
        \pi_{i_1\cdots i_n}\bar{\mathrm{X}}^{(n)}_{st}(z_1,\ldots,z_n) = \mathrm{X}^{(n)}_{st}(\pi_{i_1}z_1,\ldots, \pi_{i_n}z_n)
    \end{equation}
    for all binary words $i_1\dots i_n \in \{1,2\}^n$.
\end{definition}

Ignoring the Sobolev-norms themselves, the definition of an admissible tensor field rough path, Definition \ref{def:sobolev rough path}, is independent of the spatial dimension $d$; the Hopf-algebraic definition of a tensor field rough path over $\mathbb{R}^{2d}$ is similar to those over $\mathbb{R}^d$. With that in mind, we prove that the DiPerna--Lions lift is canonical under the assumption that $\mathbf{X}$ has slightly more spatial regularity.
\begin{lemma}
\label{Lem: canonical lift}
Let $\mathbf{X}\in \mathcal{C}^{\mathfrak{p}}W^{N+1,\infty}(\mathbb{R}^d)$ be an admissible tensor field rough path over $\mathbb{R}^d$. Then, the DiPerna--Lions lift \eqref{Eqn: X bar DiPerna Lions lift} canonically defines an admissible tensor field rough path $\bar{\bX}\in \mathcal{C}^{\mathfrak{p}}\mathcal{W}^{N}(\mathbb{R}^{2d})$ with control $w_{\bar{\mathbf{X}}}(s,t) \leq w_{\mathbf{X}}(s,t)$.
\end{lemma}

\begin{proof}
    To see Chen's relation, take any binary word $w = i_1 \dots i_n \in \{1,2\}^n$ and write
    \begin{align*}
        \pi_w \bar{\mathrm{X}}^{(n)}_{st}(z_1, \dots, z_n) & = \mathrm{X}_{st}^{(n)}(\pi_{i_1}z_1, \dots, \pi_{i_n} z_n) \\
        &= \sum_{l=0}^n \mathrm{X}_{sr}^{(l)}(\pi_{i_1}z_1, \dots, \pi_{i_l} z_l) \otimes \mathrm{X}_{rt}^{(n-l)}(\pi_{i_{l+1}}z_{l+1}, \dots, \pi_{i_n} z_n) \\
        & = \sum_{l=0}^n \pi_{u^l} \bar{\mathrm{X}}_{sr}^{(l)}(z_1, \dots,  z_l) \otimes \pi_{v^l} \bar{\mathrm{X}}_{sr}^{(n-l)}(z_{l+1}, \dots,  z_n) \\
        & = \sum_{l=0}^n (\pi_{u^l} \otimes \pi_{v^l}) (\bar{\mathrm{X}}_{sr}^{(l)}(z_1, \dots,  z_l) \otimes  \bar{\mathrm{X}}_{sr}^{(n-l)}(z_{l+1}, \dots,  z_n)) \\
        & = \pi_w \left( \sum_{l=0}^n (\bar{\mathrm{X}}_{sr}^{(l)} \otimes  \bar{\mathrm{X}}_{sr}^{(n-l)})(z_{1}, \dots,  z_n)  \right) 
    \end{align*}
    where, for each $l$, we have deconcatenated $w$ into the words $u^l := i_1 \dots i_l$ and $v^l := i_{l+1} \dots i_n$. The regularity condition \eqref{eq:sobolev RP regularity} follows from the estimate $\|\bar{\mathrm{X}}^{(n)}_{st}\|_{W^{N,\infty}(\mathbb{R}^{2d})} \lesssim \|\mathrm{X}^{(n)}_{st}\|_{W^{N,\infty}(\mathbb{R}^d)}$ combined with $\|\bar{\mathrm{X}}^{(n)}_{st}\|_{\mathfrak{t},N-n+1} \lesssim \|\mathrm{X}^{(n)}\|_{W^{N+1,\infty}(\mathbb{R}^d)}$. We saw in the proof of Proposition \ref{Prop: renorm result tensor fields} how to bound the tail divergence norm of the $\varepsilon$-transformed $\bar{\mathrm{X}}^{(n)}_{st}$; one can check that substituting $\varepsilon = 1$ and following the same line of argument yields the desired bound of the tail divergence norm. 
\end{proof}

\begin{remark}
    Morally, it should not be necessary to increase spatial regularity to have a canonical lift of tensor field rough paths. Indeed, it is possible to upgrade the spatial regularity of $\mathbf{X}$ in the statement of Lemma \ref{Lem: canonical lift} to where $\mathrm{X}^{(n)}_{st}\in W^{N,\infty}(\mathbb{R}^d)$ and $\|\mathrm{X}^{(n)}_{st}\|^*_{\mathfrak{t},N-n+1} <\infty$, where $\|\cdot\|^*_{\mathfrak{t},k}$ is a sharper version of the tail divergence norm $\|\cdot\|_{\mathfrak{t},k}$, given by
    \begin{equation}
        \|\mathrm{X}^{(n)}_{st}\|^*_{\mathfrak{t},k} := \sum_{L\in \mathrm{Asc}(1;n)} \sum_{r\in \mathcal{R}(L)} \| D_{L,r} \mathrm{X}^{(n)}_{st}\|_{W^{k,\infty}(\mathbb{R}^d)}
    \end{equation}
    using an alternative but equivalent identity for $\mathrm{div}^{\mathfrak{t}}_L \mathrm{X}^{(n)}_{st}$ that writes, for tensor fields $F$, 
    \begin{equation}
    \label{Eqn: def of iter-tail-div alt}
        \mathrm{div}^{\mathfrak{t}}_L F = \sum_{r\in \mathcal{R}(L)} D_{L,r} F, \quad D_{L,r} F := \left(\bigotimes_{\ell \in L} \nabla^{(r(\ell))}\right) \intprod_L F
    \end{equation}
    where $\mathcal{R}(L)$ is the set of ascending assignments $r\colon L \to \{1,\ldots,n\}$ such that $r(\ell) \geq \ell$ for all $\ell \in L$. Note that the triangle inequality immediately yields $\|\cdot\|^*_{\mathfrak{t},k} \leq \|\cdot\|_{\mathfrak{t},k}$. This additional sharpness is needed to control terms similar to $D_{L,r}\mathrm{X}^{(n)}_{st}$ that one encounters when studying binary representations of the iterated tail divergences $\mathrm{div}^{\mathfrak{t}}_L \bar{\mathrm{X}}^{(n)}_{st}$.
\end{remark}

\subsubsection{Transforming tensor field rough paths} \label{sec:transforming rough path}

In the proof of the renormalizability property of our rough transport equation, we shall encounter the so-called blow-up transformation of the unbounded rough drivers. These transformations turn out to be linear transformations on the underlying tensor field rough paths and in this section we briefly argue that this is a well-defined operation. 

Let $M$ and $N$ be isomorphisms on $\mathbb{R}^d$ that are constant with respect to $t$ and $x$. We want to define a new rough vector field 
$$
(T_{M,N}\mathrm{X})_t(x) = M \mathrm{X}_t(Nx)
$$
as well as its rough path lift $T_{M,N}\mathbf{X}$.
Towards this end, let us define
$$
(T_{M,N}\bX)^{(n)}_{st}(x_1, \dots, x_n) = M^{\otimes n} \mathrm{X}^{(n)}_{st}(Nx_1, \dots, Nx_n),
$$
which, for words $j_1 \dots j_n$ in $\{1,\dots, d\}$, is expressed in components by
$$
\langle (T_{M,N}\bX)^{(n)}_{st}(x_1, \dots, x_n),e_{j_1 \dots j_n}\rangle = \sum_{i_1, \dots ,i_n=1}^d M_{j_1 i_1} \dots M_{j_n i_n} \langle \mathrm{X}^{(n)}_{st}(Nx_1, \dots, Nx_n), e_{i_1 \dots i_n}\rangle.
$$
Since $M$ and $N$ are constant isomorphisms and therefore do not depend on $t\in [0,T]$ or spatial points $\mathbf{x} = (x_1,\ldots,x_n)\in (\mathbb{R}^d)^{\otimes n}$, it can be checked that both Chen's relation \eqref{eq:sobolev RP chen} and the regularity requirement \eqref{eq:sobolev RP regularity} are valid for the transformed path $T_{M,N}\mathbf{X}$. Moreover, this immediately implies that the push-forward of $\mathbf{X}$ under constant isomorphisms are again tensor field rough paths, as well as the symmetrization-antisymmetrization and $\varepsilon$-transform encountered previously in this section.

We caution that the doubling morphism of tensor field rough paths $\mathbf{X}\mapsto \bar{\mathbf{X}}$ is {\it not} covered by transformations described above, and in general the ``doubledness'' property is not stable under such transformations either: if $\bar{\mathbf{X}}$ is doubled, then $T_{M,N}\bar{\mathbf{X}}$ need not be. An example of this is given in Remark \ref{Rem: non-restrictability of blow-up Gamma}.

\subsection{Casting tensor field rough paths}

We now show how to construct an unbounded rough driver by applying casting operators to admissible tensor field rough paths.

\begin{lemma} \label{lemma:casting rough paths}
    Let $\rho$ satisfy Assumption \ref{assumption:psi} and let $(E_l^{\rho})_{0\leq l \leq N+1}$ be a scale of spaces. Suppose $\bX\in \mathcal{C}^{\mathfrak{p}}\mathcal{W}^N(\mathbb{R}^d)$ is an admissible tensor field rough path and define $\mathbf{A} = (\mathrm{A}^n)_{n=0}^N$ through the casting operator as
    \begin{equation}
        \label{Eqn: URD generated from X}
        \mathrm{A}_{st}^n := \clA^{(n)}(\mathrm{X}_{st}^{(n)}).
    \end{equation}
    Then $\bA$ is an unbounded rough driver on the scale of spaces $(E_l^{\rho})_{0\leq l \leq N+1}$.
\end{lemma}

\begin{proof}
    Using \eqref{eq:casting F gives bounded operator} from Corollary \ref{Cor: casting operator}, we see that \eqref{eq:URD takes values in spaces of linear mappings} follows immediately; moreover, its associated bound \eqref{eq:URD bound} also follows from \eqref{eq:casting F bound} and the admissibility condition \eqref{eq:sobolev RP regularity}. To prove Chen's relation \eqref{eq:URD chen} we write 
    \begin{align*}
        \mathrm{A}_{st}^n = \clA^{(n)}(\mathrm{X}_{st}^{n}) = \sum_{l=0}^n \clA^{(n)}(\mathrm{X}_{sr}^{(n-l)} \otimes  \mathrm{X}_{rt}^{(l)}) =\sum_{l=0}^n \clA^{(l)}(\mathrm{X}_{rt}^{(l)}) \circ  \clA^{(n-l)} ( \mathrm{X}_{sr}^{(n-l)}) =\sum_{l=0}^n \mathrm{A}_{rt}^l \circ \mathrm{A}_{sr}^{n-l},
    \end{align*}
    with $(s,r,t)\in \Delta^{(3)}_T$, where we have used \eqref{eq:sobolev RP chen} and \eqref{eq:casting F tensor G}.
\end{proof}

To close off this section, we include a higher-order generalization of the notion of {\it conservativity} from \cite{BaiGub2017}, culminating in a characterization result in Corollary \ref{Cor: conservativity}.
\begin{definition}[Conservativity]
    Assume that $\mathbf{X}\in \mathcal{C}^{\mathfrak{p}}_{wg}\mathcal{W}^N(\mathbb{R}^d)$. Let $\mathbf{A}$ be cast from $\mathbf{X}$ as in \eqref{Eqn: URD generated from X}, and note that $\mathcal{A}^{(0)}(\mathbf{1}) = \mathrm{Id}$. We say that the unbounded rough driver $\mathbf{A}$ is {\it conservative} if the formal operator identity
    \begin{equation}
    \label{Eqn: conservativity main id}
        \mathbf{A}^*_{st}\mathbf{A}_{st} = \mathrm{Id},
    \end{equation}
    truncated at step-$N$, holds for all $(s,t)\in \Delta^{(2)}_T$. At level $n$, the identity \eqref{Eqn: conservativity main id} writes
    \begin{equation}
        \sum_{k+l=n} \mathrm{A}^{k,*}_{st}\,\mathrm{A}^l_{st} = 0, \quad \text{for }1 \leq n \leq N.
    \end{equation}
\end{definition}

Recall the definitions of the tail derivative, tail divergence, and the representation of $\mathrm{A}^{n,*}_{st} = \mathcal{A}^{(n),*}(\mathrm{X}^{(n)}_{st})$ as in Lemma \ref{Lem: dual casting explicit}. The operator $\mathrm{A}^n_{st} = \mathcal{A}^{(n)}(\mathrm{X}^{(n)}_{st})$ admits a similar representation, which we now briefly outline and subsequently compare to the aforementioned dual representation.

We suggestively define the {\it head derivative} $D^{\mathfrak{h}}_i$ by
\begin{equation}
    D^\mathfrak{h}_i =\sum_{j=1}^{i-1}\nabla^{(j)} =  \sum_{j<i} \nabla^{(j)},\quad i \geq 2,
\end{equation}
and note the strict inequality $j<i$ for the dummy indices, as well as the front-to-back ordering contra the tail derivatives from before. Let $F\colon (\mathbb{R}^d)^n \to (\mathbb{R}^d)^{\otimes n}$ be a tensor field over $\mathbb{R}^d$. The $i$-th {\it head divergence} is given by contracting tensor fields with the head derivative 
\begin{equation}
    \mathrm{div}^{\mathfrak{h}}_i F = \sum_{\alpha = 1}^d \sum_{j_1\dots \hat{j}_i \dots j_n} D^{\mathfrak{h},\alpha}_i \langle F(\mathbf{x}),e_{j_1\dots \alpha \dots j_n}\rangle e_{j_1\dots \hat{j}_i \dots j_n}.
\end{equation} 
For ascending subsets $L \in \mathrm{Asc}(1;n)$, define the {\it iterated head divergence} by
\begin{equation}
    \mathrm{div}^\mathfrak{h}_L := \mathrm{div}^{\mathfrak{h}}_{\ell_1}\cdots \mathrm{div}^{\mathfrak{h}}_{\ell_m}.
\end{equation}
It turns out that we have the following representation formula:
\begin{equation}
\label{Eqn: casting explicit}
    \mathcal{A}^{(n)}_x (F)\, \psi = \sum_{L\in \mathrm{Asc}(2;n)} \mathrm{Tr}_x \,\mathrm{div}^{\mathfrak{h}}_L F : \nabla^{n-|L|}_x\psi
\end{equation}
where we could also include the case $L = \{1\}$ since we may simply define $\mathrm{div}^{\mathfrak{h}}_1 F = 0$. We will not prove this representation formula, as the proof would be almost identical to that of Lemma \ref{Lem: dual casting explicit}. The explicit formula for the dual casting, now including the correct alternating signature, reads
\begin{equation}
\label{Eqn: dual casting explicit with sign}
    \mathcal{A}^{(n),*}_x(F) = (-1)^n \sum_{L \in \mathrm{Asc}(1;n)} \mathrm{Tr}_x \,\mathrm{div}^{\mathfrak{t}}_L F : \nabla^{n-|L|}_x\psi.
\end{equation}
Let $\mathfrak{r}F := P_{\mathrm{rev}} F (x_n,\ldots,x_1)$ denote the reversed tensor field of $F$. Introduce the divergence of the $k$-th tensor slot as a contraction with $\nabla^{(k)}$, which reads
\begin{equation}
    \mathrm{div}^{(k)} F := \sum_{\alpha = 1}^d \sum_{j_1\dots\hat{j}_k\dots j_n} \partial^{(k)}_\alpha \langle F(\mathbf{x}), e_{j_1\dots\alpha\dots j_n}\rangle e_{j_1\dots\hat{j}_k\dots j_n}.
\end{equation}
Observe that the tail and head divergences are related by
\begin{equation*}
    \mathrm{div}^{\mathfrak{t}}_k F = \mathfrak{r}^{-1}\mathrm{div}^{\mathfrak{h}}_{k^\vee} \mathfrak{r}F + \mathrm{div}^{(k)} F,
\end{equation*}
where $k^\vee := n+1-k$. Note that $\mathfrak{r}\circ \mathfrak{r} = \mathrm{Id}$. Iterating this identity, we have 
\begin{equation}
\label{Eqn: tail-head divergence general rel}
    \mathfrak{r}\,\mathrm{div}^\mathfrak{t}_L F = \prod_{\ell \in L} (\mathrm{div}^{\mathfrak{h}}_{\ell^\vee} + \mathrm{div}^{(\ell^\vee)}) (\mathfrak{r}F).
\end{equation}
We say that a tensor field $F\colon (\mathbb{R}^d)^n \to (\mathbb{R}^d)^{\otimes n}$ is {\it divergence-free} if 
\begin{equation}
\label{Eqn: divergence-free tensor field}
    \mathrm{div}^{(k)} F (\mathbf{x}) = 0 \quad\text{ for any } k=1,2,\ldots,n
\end{equation}
holds for (almost) all $\mathbf{x}\in (\mathbb{R}^d)^n$. If $F$ is divergence-free, then \eqref{Eqn: tail-head divergence general rel} reduces to
\begin{equation}
\label{Eqn: div-free head-tail rel}
    \mathrm{div}^{\mathfrak{t}}_L F = \mathfrak{r}\, \mathrm{div}^{\mathfrak{h}}_{L^\vee} \mathfrak{r} F
\end{equation}
where $L^\vee = \{\ell^\vee_m < \cdots < \ell^\vee_1\}$ for ascending subsets $L = \{\ell_1 <\cdots <\ell_m\}$.
\begin{corollary}
\label{Cor: conservativity}
    Let $\mathbf{X}$ be a weak geometric tensor field rough path over $\mathbb{R}^d$, where each component $\mathrm{X}^{(n)}$ is divergence free as a tensor field in the sense of \eqref{Eqn: divergence-free tensor field}. Then its associated unbounded rough driver $\mathbf{A}$ is conservative. 
\end{corollary}
\begin{proof}
    By comparing the two representation formulae \eqref{Eqn: casting explicit} and \eqref{Eqn: dual casting explicit with sign} in light of \eqref{Eqn: div-free head-tail rel}, we readily establish that the divergence-freeness of $\mathrm{X}^{(n)}$ leads to 
    \begin{equation}
        \mathcal{A}^{(n),*}_x (\mathrm{X}^{(n)}_{st}) = (-1)^n \mathcal{A}^{(n)}_x (\mathfrak{r}\,\mathrm{X}^{(n)}_{st}) = \mathcal{A}^{(n)}_x (S \mathrm{X}^{(n)}_{st}) 
    \end{equation}
    where $S$ as in \eqref{Eqn: antipode S def} is the antipode map of the tensor field Hopf algebra. It is well-known that the group-like property of weak geometric paths $\mathbf{X}_{st}$ implies $S \mathbf{X}_{st} = \mathbf{X}_{st}^{-1}$, where the inverse is in the concatenation sense: $\mathbf{X}_{st}\otimes \mathbf{X}^{-1}_{st} = \mathbf{1}$ in the tensor field algebra. By the property $\mathcal{A}^{(k+l)}(F\otimes G) = \mathcal{A}^{(l)}(G)\circ \mathcal{A}^{(k)}(F)$ from Corollary \ref{Cor: casting operator}, we have
    \begin{equation*}
        \sum_{k+l = n} \mathrm{A}^{k,*}_{st}\,\mathrm{A}^l_{st} = \sum_{k+l = n} \mathcal{A}^{(k),*}_x(\mathrm{X}^{(k)}_{st})\circ \mathcal{A}^{(l)}_x(\mathrm{X}^{(l)}_{st}) =  \mathcal{A}^{(n)}_x\left(\sum_{k+l=n}\mathrm{X}^{(l)}_{st}\otimes S \mathrm{X}^{(k)}_{st}\right) = \mathcal{A}^{(n)}_x (\mathbf{1}) = 0
    \end{equation*}
    for every $1\leq n \leq N$. This concludes the proof.
\end{proof}

\section{Existence of rough solutions} 
\label{sec:existence}

The main goal of this section is to prove existence of a solution of \eqref{eq:main eq}. Let us start by defining what is meant by a solution using Lemma \ref{lemma:casting rough paths}.

\begin{definition}\label{def:main eq}
    Assume $b \in L^1([0,T];L^{1}_{\loc}(\R^d;\R^d))$ is such that $\Div \,b \in L^1([0,T];L^{\infty}_{\loc}(\R^d))$. Let $N= \lfloor \mathfrak{p}\rfloor$ and assume $\mathbf{X}\in \mathcal{C}^\mathfrak{p}\mathcal{W}^N(\mathbb{R}^d)$ is a tensor field rough path. We call a path $u\in \mathcal{B}_b( [0,T]; L_{\loc}^{\infty}(\R^d))$ a {\it  solution} of the rough partial differential equation
    \begin{equation}
    \label{Eqn: rough solution def}
        \partial_t u = b\cdot \nabla u + \dot{\bX} \cdot \nabla u
    \end{equation}
    provided $u$ satisfies the abstract expansion \eqref{eq:URD diff} on the scale $(\clF_{l,R})_{0\leq l \leq N+1}$ for every $R \geq 1$,
    where the drift term $\mu$ is defined by
    $$
    \mu\colon [0,T] \rightarrow \clF_{-1,R}, \qquad \mu_t := \int_0^t b_r \cdot \nabla u_r \,\mathrm{d}r, 
    $$
    and the unbounded rough driver is cast from $\mathbf{X}$ as in Lemma \ref{lemma:casting rough paths}, 
    $$
    \mathrm{A}_{st}^n = \clA^{(n)}(\mathrm{X}_{st}^{(n)}), \quad n= 1, \dots, N.
    $$
\end{definition}

In the above definition, we assume $\mu\colon [0,T] \rightarrow \clF_{-1,R}$ for every $R \geq 1$. Using the assumptions on $u$ and $b$ as above, this holds automatically since if $\phi \in \clF_{1,R}$ we have
$$
\langle  \mu_{t} , \phi \rangle  = - \int_0^t \langle u_r, b_r \cdot \nabla \phi + \Div \,b_r\, \phi \rangle \,\mathrm{d}r \leq \|\phi\|_{\clF_{1,R}} \int_0^t \| u_r b_r \|_{L^1(B_R)} + \| u_r  \Div\, b_r \|_{L^1(B_R)} \,\mathrm{d}r.
$$
Moreover, we note that our abstract expansion-based theory from Section \ref{sec:a priori} needed that $\mu\in \mathcal{C}^{1\mathrm{-var}}\mathcal{F}_{-1,R}$, and we see that finite temporal variation also follows from the above. Thus, we are well-equipped to use the aforementioned a priori estimates.

Before showing the existence of a solution, we shall use a priori estimates to show that if we have a \emph{positive} solution $u$ of \eqref{eq:main eq}, we can control the growth of its $L^1(\R^d)$-norm. Then, the strategy for proving the existence of a solution of \eqref{eq:main eq} is through an approximation procedure, given that we assume that the tensor field rough path $\bX$ is \emph{strong} geometric in the sense of Definition \ref{def:strong geometric}.

\begin{theorem} \label{thm:positive solutions are uqnique}
    Assume $b \in L^1([0,T];L^{\infty}(\R^d;\R^d))$ is such that $\Div\, b \in L^1([0,T];L^{\infty}(\R^d))$ and that $\bX$ is a Sobolev rough path as in Definition \ref{def:sobolev rough path}.
    Assume $u\in \mathcal{B}_b([0,T]; L^1(\R^d))$ is a positive solution of 
    $$  
    \partial_t u = b \cdot \nabla u + \dot{\bX} \cdot \nabla u
    $$
    as in Definition \ref{def:main eq}, now with initial condition $u_t |_{t=0} = u_0 \in L^1$.
    Then we have  
    $$
     \sup_{t \in [0,T]} \|u_t \|_{L^1} \leq C \|u_0\|_{L^1}, \quad \textrm{ and }  \quad |  \|u_t \|_{L^1} -  \|u_s \|_{L^1} | \leq  Cw_*(s,t)^{1/\mathfrak{p}}
    $$
    for some constant $C$ and a control $w_*$. 
\end{theorem}

\begin{proof}
    Take any $\phi \in \clF_{1,R}$. Then the drift term is estimated by
    \begin{equation} \label{eq:drift estimate bounded b}
    \langle \delta \mu_{st}, \phi\rangle = -\int_s^t \langle u_r, \Div\, b_r \phi +  b_r \cdot \nabla \phi \rangle \,\mathrm{d}r \leq \|u\|_{\mathcal{B}_b([s,t]; L^1)}\, w_b(s,t) \,\|\phi\|_{1,R}    
    \end{equation}
    where we have defined the control
    $$
    w_b(s,t) := \int_s^t \|\mathrm{div}\, b_r \|_{L^{\infty}} + \|b_r \|_{L^{\infty}} \,\mathrm{d}r.
    $$
    Moreover, we clearly have $\|u\|_{\mathcal{B}_b([s,t]; \clF_{-0,R})} \leq \|u\|_{\mathcal{B}_b([s,t]; L^1)}$ and so from Proposition \ref{prop:a priori remainder} we find 
    $$
    \|u^{\natural}_{st} \|_{\clF_{-(N+1),R}} \leq \|u\|_{\mathcal{B}_b([s,t]; L^1)} w_*(s,t)^{(N+1)/\mathfrak{p}}
    $$
    for some control $w_*$.

    Next, choose a smooth function $\phi$ such that 
    $$
    \phi(x) = 
    \begin{cases}
    1, & \textrm{ if } |x| \leq \frac12, \\
    0, & \textrm{ if } |x| \geq 1, \\
    \end{cases}
    $$
    and let $\phi_R(x) := \phi\left( \frac{x}{R} \right)$. We see that $\phi_R \in \clF_{\infty,R}$ and $\|\phi_R\|_{\clF_{l,R}} \lesssim \|\phi\|_{W^{l,\infty}}$. Testing \eqref{eq:URD expansion} against $\phi_R$ we find
    \begin{align*}
        \langle \delta u_{st} , \phi_R \rangle & = - \int_s^t \langle u_r, \Div b_r \phi_R + b_r \cdot \nabla  \phi_R \rangle \,\mathrm{d}r + \big\langle u_s, \sum_{n=1}^N \mathrm{A}_{st}^{n,*} \phi_R \big\rangle  + \langle u_{st}^{\natural}, \phi_R \rangle \\
        & \leq \|u\|_{\mathcal{B}_b([s,t]; L^1)} w_b(s,t) \|\phi\|_{W^{1,\infty}} + \|u_s\|_{L^1} \sum_{n=1}^N w_{\bA}(s,t)^{n/\mathfrak{p}}  \|\phi\|_{{W^{n,\infty}}} \\
        & + \|u\|_{\mathcal{B}_b([s,t]; L^1)} w_*(s,t)^{(N+1)/\mathfrak{p}} \|\phi\|_{W^{N+1,\infty}} \\
        & \leq \|u\|_{\mathcal{B}_b([s,t]; L^1)} w_*(s,t)^{1/\mathfrak{p}}  \|\phi\|_{W^{N+1,\infty}}
    \end{align*}
    for some (new) control $w_*$. Letting $R \rightarrow \infty$ and using that $u$ is positive, we get that 
    $$
 \langle \delta u_{st} , \phi_R \rangle \rightarrow  \delta (\|u\|_{L^1})_{st} .
    $$
    The result now follows from the rough Gronwall lemma, Lemma \ref{lemma:rough gronwall}, with $G_t := \|u_t\|_{L^1}$. 
\end{proof}

Next we prove the existence of a solution of \eqref{eq:main eq} via smoothing out the coefficients which gives an approximate solution $u^{\varepsilon}$. The above theorem allows us to get uniform bounds in $\varepsilon$ which then enables using a compactness criterion from \cite[Proposition B.4]{GLN}, recalled in Appendix \ref{app:compactness}.
    
\begin{theorem} \label{thm:existence}
    Let $u_0 \in L^2(\R^d)$ and assume $b \in L^1([0,T];L^{\infty}(\R^d;\R^d))$ is such that $\Div b \in L^1([0,T];L^{\infty}(\R^d))$. Assume $\bX\in \mathcal{C}^{\mathfrak{p}}_g\mathcal{W}^N(\mathbb{R}^d)$ is a strong geometric tensor field rough path as in Definition \ref{def:strong geometric}. Then there exists a solution of 
    $$
    \partial_t u = b \cdot \nabla u + \dot{\bX} \cdot \nabla u
    $$
    with initial condition $u|_{t=0} = u_0$. 
\end{theorem}

\begin{proof}
    Let $\{\mathrm{X}^{\varepsilon}\}_{\varepsilon > 0}$ be as described in Definition \ref{def:strong geometric}. Take a mollifier $\psi_{\varepsilon}$ and define $b^{\varepsilon}_t(x) := b_t \ast \psi_{\varepsilon}(x)$ where the convolution is taken in the spatial variable. It is then classical that the equation
    \begin{equation} \label{eq:smooth coefficients}
        \partial_t u^{\varepsilon} = b^{\varepsilon} \cdot \nabla u^{\epsilon} + \dot{\mathrm{X}}^{\varepsilon}\cdot \nabla u^{\varepsilon}, \quad u_0 \in L^2(\R^d)
    \end{equation}
    admits a renormalized solution in the sense that $v^{\varepsilon}_t(x) := (u_t^{\varepsilon}(x))^2$ satisfies \eqref{eq:smooth coefficients} with initial condition given by $v_0^{\varepsilon} = u_0^2 \in L^1(\R^d)$. Arguing as in Section \ref{sec:intro} we find that 
    $$
    \delta v_{st}^{\varepsilon} = \delta \mu_{st}^{2,\varepsilon} + \sum_{n=1}^N \mathrm{A}_{st}^{\varepsilon,n} v_s^{\varepsilon} + v_{st}^{\varepsilon,\natural}
    $$
    on the scale $\{\clF_{k,R}\}_{k \geq 0}$ for any $R > 0$. Above, we have defined 
    $$
    \mathrm{A}_{st}^{\varepsilon,n} := \clA^{(n)}(\mathrm{X}_{st}^{\varepsilon,(n)}), \quad     \textrm{and }
    \delta \mu_{st}^{2,\varepsilon} := \int_s^t b_r^{\varepsilon} \cdot \nabla v_r^{\varepsilon} \,\mathrm{d}r.
    $$
    Clearly, the mollified $v^{\varepsilon}$ are positive, hence Theorem \ref{thm:positive solutions are uqnique} yields 
    $$
    \sup_{t \in [0,T]} \|v_t^{\varepsilon}\|_{L^1(\R^d)} \leq C \|v_0^{\varepsilon}\|_{L^1(\R^d)} 
    $$
    which again implies 
    $$
    \sup_{t \in [0,T]} \|u_t^{\varepsilon}\|_{L^2(\R^d)} \leq C \|u_0\|_{L^2(\R^d)}.
    $$
    Fix $R > 0$ and consider the equation for $u^{\varepsilon}$ on the scale $(\clF_{k,R})_{0\leq k \leq N+1}$, viz.
    $$
    \delta u_{st}^{\varepsilon} = \delta \mu_{st}^{\varepsilon} + \sum_{n=1}^N \mathrm{A}_{st}^{\varepsilon,n} u_s^{\varepsilon} + u_{st}^{\varepsilon,\natural}
    $$
    where $\mathrm{A}_{st}^{\varepsilon,n}$ is as before and $\delta \mu_{st}^{\varepsilon} := \int_s^t b_r^{\varepsilon} \cdot \nabla u_r^{\varepsilon} \,\mathrm{d}r$. For $\phi \in \clF_{1,R}$ we have
    \begin{align*}
    \langle \delta \mu_{st}^{\varepsilon},\phi \rangle & = -\int_s^t \langle u_r^{\varepsilon}, b^{\varepsilon}_r \cdot \nabla \phi + \phi\,\Div\, b^{\varepsilon}_r  \rangle \,\mathrm{d}r \\
    & \leq \int_s^t  \| u_r^{\varepsilon}\|_{L^1(B_R)} ( \| b^{\varepsilon}_r\|_{L^{\infty}}  + \| \Div \,b^{\varepsilon}_r\|_{L^{\infty}} )\|\phi\|_{\clF_{1,R}} \mathrm{d}r \\
    & \lesssim   \sqrt{\mathrm{vol}(B_R)}\, \| u_0\|_{L^2(\R^d)} w_b(s,t) \|\phi\|_{\clF_{1,R}}  
    \end{align*}
    where in the last inequality we have used the property that 
    $$
   \int_s^t \|b^{\varepsilon}_r\|_{L^{\infty}} + \| \Div \,b^{\varepsilon}_r\|_{L^{\infty}} \,\mathrm{d}r  \leq \int_s^t \|b_r\|_{L^{\infty}} + \| \Div \,b_r\|_{L^{\infty}} \mathrm{d}r =: w_b(s,t)
    $$ 
    by standard facts about mollifiers. 

    Using Lemma \ref{lemma:a priori u} we can now invoke the compactness criterion provided by Proposition \ref{prop:weak compactness} to find that there exists a bounded path $u\colon [0,T] \rightarrow L^2(\R^d)$ and a subsequence $\{u^{\varepsilon_k}\}_{k \geq 1}$ converging to $u$ in the sense that for all $\phi \in C^{\infty}_c(\R^d)$ 
    $$
    \sup_{t \in [0,T]} | \langle u_t - u_t^{\varepsilon_k}, \phi \rangle | \rightarrow 0
    $$
    as $k \rightarrow \infty$. 

    By construction, for any $n$, the convergence $\mathrm{A}^{\varepsilon_k,n,*}_{st} \rightarrow \mathrm{A}^{n,*}_{st}$ holds in the strong operator topology. Moreover, we have $b^{\varepsilon_k} \rightarrow b$ and $\Div \,b^{\varepsilon_k} \rightarrow \Div \,b$ strongly in $L^1([0,T];L^2_{\loc}(\R^d))$. Coupled with weak compactness of $u^{\varepsilon_k}$ we see that  
    for every $\phi \in \clF_{N+1,R}$ we have
    $$
    \langle \mu_{st}^{\varepsilon_k}, \phi \rangle  \rightarrow \langle\mu_{st}, \phi \rangle  \qquad \textrm{and} \qquad \langle\mathrm{A}_{st}^{\varepsilon_k,n} u_s^{\varepsilon_k}, \phi \rangle  \rightarrow \langle \mathrm{A}_{st}^{n}u_s, \phi \rangle .
    $$
    Thus, we find that the remainder satisfies
    $$
    \langle u_{st}^{\natural}, \phi\rangle  = \lim_{k \rightarrow \infty} \langle u_{st}^{\varepsilon_k,\natural}, \phi\rangle 
    $$
    for any $\phi \in \clF_{N+1,R}$. 
    Finally, by Proposition \ref{prop:a priori remainder} we get
    $$
    \sup_k |\langle u^{\varepsilon_k, \natural}_{st}, \phi \rangle |\leq w_*(s,t)^{(N+1)/\mathfrak{p}} \|\phi\|_{\clF_{N+1,R}},
    $$
    which shows that $u^{\natural}$ satisfies the required temporal regularity. 
\end{proof}

\begin{remark} \label{rem:growth on b}
    It would be desirable to be able to show the existence of solutions of \eqref{eq:main eq} for drifts $b$ belonging to the full DiPerna--Lions regime, i.e.~allowing for the growth condition
    $$
    \frac{b}{1+|x|} \in L^1([0,T];L^1(\R^d;\R^d) + L^{\infty}(\R^d;\R^d)).
    $$
    However, at the present level of understanding, this result is not within reach unless additional conditions are assumed. Indeed, if we allow growth on $b$, the estimate \eqref{eq:drift estimate bounded b} has to be replaced by an estimate of the form 
    \begin{equation} \label{eq:drift estimate ubounded b}
    \langle \delta \mu_{st}, \phi\rangle \leq \|u\|_{\mathcal{B}_b([s,t]; L^1)} \,w_b(s,t) \|\phi\|_{1,R} (1+R)
    \end{equation}
    for a suitable control $w_b$, which would yield a growth in $R$ in the estimate of the remainder term 
    \begin{equation} 
    \label{eq:growth in remainder when b has linear growth}
        \|u^{\natural}_{st} \|_{\clF_{-(N+1),R}} \leq \|u\|_{\mathcal{B}_b([s,t]; L^1)}\, w_*(s,t)^{(N+1)/\mathfrak{p}}(1+R)  
    \end{equation}
    and it is not clear how to obtain uniform estimates on $\langle \delta u_{st} , \phi_R \rangle$ from this point.

    However, if we assume that the noisy vector field $\mathrm{X}$ along with its corresponding lift $\mathbf{X}$ are \emph{divergence free} in the sense of \eqref{Eqn: divergence-free tensor field}, i.e.~if $\mathbf{A}$ is a conservative driver, the dual $\mathrm{A}^{n,*}_{st}$ is then a transport operator cast from the antipode map applied to $\mathbf{X}$:
    \begin{equation*}
        \mathrm{A}^{n,*}_{st} =  \mathcal{A}^{(n)}_x (S\mathrm{X}^{(n)}_{st}),
    \end{equation*}
    hence every term of $\mathrm{A}^{n,*}_{st}$ contains a derivative that hits $\phi$, and so we have
    \begin{equation}
    \label{eq:decay in URD when divergence free}
        \|\mathrm{A}_{st}^{n,*} \phi_R\| \lesssim R^{-1} \|\phi\|_{W^{n,\infty}}.        
    \end{equation}
    Then, with partial remainders $u^{\sharp,(N-n)}$ as in the proof of Proposition \ref{prop:a priori remainder}, we find
    $$
    \langle \delta u_{srt}^{\natural}, \phi_R \rangle  = \sum_{n=1}^N \langle u_{sr}^{\sharp,(N-n)}, \mathrm{A}_{rt}^{n,*} \phi_R \rangle
    $$
    so that the decay in \eqref{eq:decay in URD when divergence free} can be used to cancel out the growth in \eqref{eq:growth in remainder when b has linear growth}. This strategy was used in \cite{GLN} in the case when $\mathfrak{p} \in [2,3)$.
\end{remark}

\begin{remark}
In light of Theorem \ref{thm:positive solutions are uqnique} we see the following strategy for proving uniqueness (which will be implemented in Section \ref{sec:renormalization}); assume $u$ is a solution of \eqref{eq:main eq}. Under extra regularity assumptions on $b$ and $\bX$ we can show that also $u^2$ solves the same equation. Then, since $u^2 \geq 0$, Theorem \ref{thm:positive solutions are uqnique} gives us $\sup_{t \in [0,T]} \|u_t\|_{L^2} \leq C \|u_0\|_{L^2}$ which implies uniqueness since the equation is linear. 
\end{remark}

\section{Renormalizability of rough transport equations} \label{sec:renormalization}

The main goal of this section is to prove that the rough transport equation \eqref{eq:main eq} is {\it renormalizable}. Towards the end of this section, we use the renormalizability to prove the uniqueness and stability of solutions \eqref{eq:main eq} assuming square integrable initial conditions.

\subsection{Motivation and outline of argument}

We begin by motivating our line of argument for this section. To that end, consider an abstract linear operator $\mathrm{A}$ that is compatible with the multiplication $m\colon  f\otimes g \mapsto f\cdot g$, but not necessarily bounded, and note that the coproduct on operators $\Delta_{\mathrm{op}}$ in this case reads 
\begin{equation}
    \label{Eqn: coproduct definition abstract}
    \mathrm{A}\,(m(f\otimes g)) = m(\Delta_{\mathrm{op}}(\mathrm{A})(f\otimes g)).
\end{equation}
The point of departure for our discussion is that due to the presence of distributions, the multiplication products on either side of \eqref{Eqn: coproduct definition abstract} are generally ill-defined. Here lies the point of the {\it tensorization} $\Gamma$ of our equation: applying $\Gamma = \Delta_{\mathrm{op}}(\mathrm{A})$ to tensor products of distributions still makes sense. Indeed, when $D$ is a differential operator compatible with multiplication, the tensorization $\Gamma = \Delta_{\mathrm{op}}(D)$ reflects the usual Leibniz rule; for example, if $D= \partial^2/\partial x^2$ we have
\begin{equation*}
    \Gamma(f\otimes g) = \partial_x^2 f \otimes g + 2\partial_x f \otimes \partial_x g + f\otimes \partial_x^2 g
\end{equation*}
in the sense of distributions. By studying the properties of the tensorization $\Gamma$, we hope to make sense of the left-hand side of \eqref{Eqn: coproduct definition abstract} by a {\it renormalization} procedure, which we now schematically describe. 

Recall that $u_t$ is a rough path solution of \eqref{Eqn: rough solution def} if it satisfies the expansion
\begin{equation}
\label{Eqn: first expansion for tensorization}
    \delta u_{st} = \delta \mu_{st} +  \sum_{n=1}^N \mathrm{A}^n_{st}u_s + u^\natural_{st},
\end{equation}
where the remainder term $u^\natural_{st}$ has finite $\mathfrak{p}/(N+1)$-variation, and the forcing term $\mu\colon [0,T]\to \mathcal{F}_{-1,R}$ is given explicitly in terms of the drift $b = b_t(x)$ by
\begin{equation*}
    \mu_t := \int_0^t b_r \cdot \nabla u_r \,\mathrm{d}r.
\end{equation*} 
Inspired by the doubling-of-variables procedure à la DiPerna--Lions \cite{diperna1989ordinary}, we consider the increment expansion for the distributional tensor product $u^{\otimes 2}_t = u_t\ftensor u_t$ and write 
\begin{equation*}
    \delta (u \ftensor u)_{st} = (\delta u_{st})\ftensor u_s + u_s \ftensor (\delta u_{st}) + (\delta u_{st}) \ftensor (\delta u_{st}).
\end{equation*}
Using the expansion \eqref{Eqn: first expansion for tensorization} and gathering terms of the same order, we obtain the formal tensorization 
\begin{equation}
\label{Eqn: tensorized equation}
    \delta (u^{\otimes 2})_{st} = \delta M_{st} +  \sum_{n=1}^N \Gamma^n_{st}\,u_s^{\otimes 2} + u^{\otimes 2, \natural}_{st}
\end{equation}
where $\Gamma^n_{st}$ is the $n$-th component of the tensorization $\Gamma_{st} = \Delta_{\mathrm{op}}(\mathbf{A}_{st})$, here given by
\begin{equation}
\label{Eqn: tensorization def}
    \Gamma^n_{st} := \sum_{k+l = n} \mathrm{A}^k_{st}\ftensor \mathrm{A}^l_{st},
\end{equation}
the term $\delta M$ denotes the increment of the tensorized forcing defined by
\begin{equation}
\label{Eqn: tensorized forcing M}
    M_{t} := \int_0^t (b_r\cdot \nabla u_r)\ftensor u_r + u_r\ftensor (b_r\cdot \nabla u_r)\,\mathrm{d}r,
\end{equation}
and the remainder term $u^{\otimes 2, \natural}_{st}$ has absorbed the terms of finite $\mathfrak{p}/(N+1)$-variation; see Proposition \ref{Prop: welldefiniteness of tensorized equation} for details.

The renormalization procedure of the above equation hinges on the so-called {\it blow-up transformation} introduced in \cite{BaiGub2017}, which we now describe.

\begin{definition}[The blow-up transformation]
    Let $\varepsilon \in (0,1)$, and consider $\Phi \in L^\infty(\mathbb{R}^{2d})$. The {\it blow-up transformation} $T_\varepsilon$ acting on $\Phi$ is formally defined by 
    \begin{equation}
    \label{Eqn: T-epsilon def on functions}
        T_\varepsilon \Phi (x,y) = \varepsilon^{-d} \,\Phi\left(x_+ + \frac{x_-}{\varepsilon}, x_+ - \frac{x_-}{\varepsilon}\right)
    \end{equation}
    where, as before, $x_\pm$ denote the parallel and transverse coordinates $x_\pm = (x\pm y)/2$. The blow-up transformation is a linear automorphism on $W^{l,\infty}(\mathbb{R}^{2d})$ for every $l\in \mathbb{N}$.
\end{definition}
Recall the scale of spaces $\mathcal{E}_{l,R}$ from Definition \ref{Def: the scales E and F}, and define the transformed scale $\tilde{\mathcal{E}}_{l,\varepsilon,R} = T_\varepsilon(\mathcal{E}_{l,R})$ where, as before, $R\geq 1$ is kept fixed. We observe that this transformed scale can be characterized by the prototypical scale $E^{\rho}_l$ as in \eqref{Eqn: E space def}, whose localization function is given by
\begin{equation*}
    \zeta_{\varepsilon,R} (x,y) = \zeta_R\left(x_++\frac{x_-}{\varepsilon}, x_+-\frac{x_-}{\varepsilon}\right)
\end{equation*}
where $\zeta_R$ defined the scale $\mathcal{E}_{l,R}$, and by extension we have
\begin{equation*}
    \tilde{\mathcal{E}}_{l,\varepsilon,R} = E^{\zeta_{\varepsilon,R}}_l(\mathbb{R}^{2d}) =  \left\{\Phi \in W^{l,\infty}(\mathbb{R}^{2d}) : \frac{|x_+|^2}{R^2} + \frac{|x_-|^2}{\varepsilon^2} \geq 1 \Rightarrow \Phi(x,y) = 0\right\}.
\end{equation*}
Thus, $\tilde{\mathcal{E}}_{l,\varepsilon,R}$ is a closed subspace of $W^{l,\infty}(\mathbb{R}^{2d})$ and is consequently endowed with the corresponding subspace topology. Moreover, for $\varepsilon \in (0,1)$ one easily checks that
\begin{equation*}
    \zeta_{\varepsilon,R}(x,y) \geq \zeta_R(x,y),
\end{equation*}
hence $\tilde{\mathcal{E}}_{l,\varepsilon,R} \subset \mathcal{E}_{l,R}$, and $T_\varepsilon$ is invariant on $\mathcal{E}_{l,R}$ for all $\varepsilon\in (0,1)$. With the pairing inherited from $L^2(\mathbb{R}^{2d})$ (or the distributional pairing from $\mathscr{D}'(\mathbb{R}^{2d})$), we also consider the formal dual operator $T^*_\varepsilon$, given pointwise by
\begin{equation}
    \label{Eqn: T-epsilon-star def on functions}
    T^*_\varepsilon \Phi (x,y) = \Phi(x_++\varepsilon x_-, x_+-\varepsilon x_-)
\end{equation}
whenever pointwise evaluation of $\Phi$ makes sense. 

Applying the blow-up transformation to test functions $\Phi$ from some appropriate scale of spaces, the tensorized equation \eqref{Eqn: tensorized equation} reads
\begin{equation}
\label{Eqn: tensorized eqn blow-up test}
    \langle \delta u ^{\otimes 2}_{st}, T_\varepsilon \Phi\rangle = \langle \delta M_{st}, T_\varepsilon \Phi \rangle + \sum_{n=1}^N \langle \Gamma^n_{st} u^{\otimes 2}_{s}, T_\varepsilon \Phi\rangle + \langle u^{\otimes 2, \natural}_{st}, T_\varepsilon \Phi \rangle
\end{equation}
or equivalently, using the dual transformation $T^*_\varepsilon$, as a distributional expansion
\begin{equation}
\label{Eqn: tensorized eqn dual blow up}
    T^*_\varepsilon \delta u ^{\otimes 2}_{st} = \delta M^\varepsilon_{st} +\sum_{n=1}^N \Gamma^{\varepsilon,n}_{st} \,T^*_\varepsilon u^{\otimes 2}_{s} + u^{\otimes 2, \natural,\varepsilon}_{st}
\end{equation}
where we have defined the blow-up transformed versions
\begin{equation}
    \Gamma^{\varepsilon,n}_{st} := T^*_\varepsilon \Gamma^n_{st}T^{*,-1}_\varepsilon, \quad  M^\varepsilon_t := T^*_\varepsilon M_t, \quad  u^{\otimes 2, \natural,\varepsilon}_{st} := T^*_\varepsilon u^{\otimes 2,\natural}_{st}.
\end{equation}
The latter formulation \eqref{Eqn: tensorized eqn dual blow up} yields a rough driver equation in terms of the transformed distributional tensor product $T_\varepsilon^* u^{\otimes 2}_{t}$ whenever the remainder term $u^{\otimes 2, \natural,\varepsilon}_{st}$ has finite $\mathfrak{p}/(N+1)$-variation. Testing \eqref{Eqn: tensorized eqn blow-up test} against a suitably constructed $\Psi$, which we now introduce, will allow us to make sense of the singular limit $\varepsilon\to 0$. Whenever we write $\Psi$ in the sequel, we always refer to such a distinguished test function. 

\begin{definition}
\label{Def: Psi function}
    Let $\phi \in \mathcal{F}_{N+1,R}$ be arbitrary and fix $\chi\in C^\infty_c(\mathbb{R}^d)$ with $\mathrm{supp}\,\chi \subset B_{1/2}$ such that 
    \begin{equation*}
        \chi_\varepsilon(x) := \varepsilon^{-d}\chi(x/\varepsilon)
    \end{equation*}
    is a standard mollifier for $0<\varepsilon<1$; in particular, we normalize $\int_{B_{1/2}}\chi \,\mathrm{d}x = 1$. We define a distinguished test function $\Psi\in \mathcal{E}_{N+1,2R}$ by
    \begin{equation}
        \Psi(x,y) := \phi(x_+)\,\chi(2x_-).
    \end{equation}
    Note that for such a function, its blow-up transformation is $T_\varepsilon \Psi = \phi(x_+)\,\chi_\varepsilon (x-y)$.
\end{definition}
Indeed, applying the blow-up transformation on $\Psi$, its $\varepsilon\to 0$ limit produces a Dirac delta distribution that collapses the diagonal $x=y$; one has
\begin{equation*}
    T_\varepsilon \Psi (x,y) \longrightarrow \phi(x)\,\delta_{x=y}
\end{equation*}
with convergence in the distributional sense. Testing with $\Psi$, and provided that we can pass every term in \eqref{Eqn: tensorized eqn blow-up test} to the limit $\varepsilon\to 0$, the convergence above recovers the equation 
\begin{equation}
\label{Eqn: expansion eqn squared transport}
    \langle \delta u^2_{st},\phi \rangle = \langle \delta \tilde{\mu}_{st},\phi\rangle + \sum_{n=1}^N \langle \mathrm{A}^n_{st} u_s^2,\phi \rangle + \langle u^{2,\natural}_{st},\phi\rangle
\end{equation}
which is nothing but the rough driver expansion of 
\begin{equation}
\label{Eqn: squared transport}
    \partial_t u^2 = \partial_t \tilde{\mu} + \dot{\mathrm{X}}\cdot \nabla u^2,
\end{equation}
where the new drift takes the form
\begin{equation}
    \tilde{\mu}_t = \int_0^t b_r\cdot \nabla u_r^2\,\mathrm{d}r,
\end{equation}
proving renormalizability in the sense of squared solutions. As a nod to our previous discussion surrounding \eqref{Eqn: coproduct definition abstract}, the outlined procedure will allow us to rigorously make sense of the seemingly illegal formal computation
\begin{equation*}
    \mathrm{A}^n_{st}u^2_s =  \lim_{\varepsilon\to 0}m(\Gamma^{\varepsilon,n}_{st} T^*_\varepsilon(u_s\ftensor u_s))
\end{equation*}
using renormalizability to prove convergence results.

\subsection{The tensorization}
In this subsection, our goal is to establish an algebraic relationship between the tensorization $\Gamma = \{\Gamma^n\}_{n=1}^N$ as in \eqref{Eqn: tensorized equation} and an unbounded rough driver defined by applying a casting operator to the DiPerna--Lions lift $\bar{\mathbf{X}}$ of a tensor field rough path $\mathbf{X}$.

First, before continuing with representations of the tensorization itself, we prove that the tensorized equation yields a suitable rough analogue of the doubled-variable version of the transport equation. Similar results of the following kind have become standard in the unbounded rough drivers literature -- our presentation of this result in particular takes inspiration from \cite[Proposition 4.21]{GLN}.
\begin{proposition}
\label{Prop: welldefiniteness of tensorized equation}
    Let $u$ be a solution of \eqref{Eqn: first expansion for tensorization} and assume that 
    \begin{equation*}
        u\in \mathcal{B}_b([0,T]; L^\infty_{\mathrm{loc}}), \quad b\in L^1_t L^{\infty}_{\mathrm{loc}},\quad \mathrm{div}\,b \in L^1_t L^{\infty}_{\mathrm{loc}}.
    \end{equation*}
    Let $\mathbf{A} = (\mathrm{A}^n)_{n=1}^N$ be an unbounded rough driver on the scales $(\mathcal{F}_{l,R+1})_{0\leq l \leq N+1}$. Then the distributional tensor product $u^{\otimes 2} = u\ftensor u \in \mathcal{B}_b([0,T];L^\infty_{\mathrm{loc}}(\mathbb{R}^{2d}))$ is an unbounded rough driver solution of 
    \begin{equation}
        \mathrm{d}(u^{\otimes 2})_t = \mathrm{d}M_t + \Gamma _{\mathrm{d}t}(u^{\otimes 2})_t 
    \end{equation}
    on the scale of spaces $(\mathcal{E}_{l,R})_{0\leq l \leq N+1}$ given in Definition \ref{Def: the scales E and F} for every $R \geq 1$, where the unbounded rough driver $\Gamma = \{\Gamma^n\}_{n=1}^N$ is defined through the formal tensorization \eqref{Eqn: tensorization def} and the forcing $M$ is given in \eqref{Eqn: tensorized forcing M}. Moreover, $M$ belongs to $\mathcal{C}^{1\mathrm{-var}}\mathcal{E}_{-1,R}$ for every $R\geq 1$.
\end{proposition}
\begin{proof}
    Fix the parameter $R\geq 1$. The assumptions on $b$ yield that $\mu\in \mathcal{C}^{1\mathrm{-var}}\mathcal{F}_{-1,R+1}$; see the discussion following Definition \ref{def:main eq}.
    Recall that the remainder term of expansions such as \eqref{Eqn: tensorized equation} is defined by
    \begin{equation}
    \label{Eqn: def of remainder tensorized}
        u^{\otimes 2,\natural}_{st} := \delta(u^{\otimes 2})_{st} -\delta M_{st}- \sum_{n=1}^N \Gamma^n_{st} u^{\otimes 2}_s 
    \end{equation}
    where we furthermore expand 
    \begin{equation}
    \label{Eqn: expansion of increment tensorized}
        \delta(u^{\otimes 2})_{st} = (\delta u_{st})\ftensor u_s + u_s\ftensor (\delta u_{st}) + (\delta u_{st})\ftensor (\delta u_{st}).  
    \end{equation}
    Substituting the expansion $\delta u_{st} = \delta \mu_{st} + \sum_n \mathrm{A}^n_{st}u_s + u^\natural_{st}$ in \eqref{Eqn: expansion of increment tensorized}, and by gathering and subsequently cancelling the resulting terms against the tensorization terms $\Gamma^n_{st}u_s$ and the tensorized forcing $\delta M_{st}$, the remaining terms of \eqref{Eqn: def of remainder tensorized} take the form
    \begin{equation*}
        u^{\otimes 2,\natural}_{st} = \mathcal{R}^{\mathrm{old}}_{st}  + \mathcal{R}^{\mathrm{drift}}_{st} + \mathcal{R}^{\otimes 2}_{st}, 
    \end{equation*}
    where $\mathcal{R}^{\mathrm{old}}$ contains the old remainder terms of finite $\mathfrak{p}/(N+1)$-variation
    \begin{equation*}
        \mathcal{R}^{\mathrm{old}}_{st} := u^\natural_{st}\ftensor u_s + u_s \ftensor u^{\natural}_{st},
    \end{equation*}
    and the term $\mathcal{R}^{\mathrm{drift}}$ reflects the remainder from the drift increments
    \begin{equation*}
        \mathcal{R}^{\mathrm{drift}}_{st} := \delta{\mu}_{st}\ftensor u_s + u_s\ftensor \delta \mu_{st} - \delta M_{st} = -\int_s^t \dot{\mu}_{r}\ftensor \delta u_{sr} + \delta u_{sr}\ftensor \dot{\mu}_r\,\mathrm{d}r.
    \end{equation*}
    Note that by Lemma \ref{lemma:a priori u} one has $\delta u\in \mathcal{C}_2^{\mathfrak{p}}\mathcal{F}_{-1,R+1}$, hence
    \begin{equation*}
        -\int_s^t \dot{\mu}_r\ftensor \delta u_{sr}\,\mathrm{d}r \in \mathcal{C}^{1}_2\mathcal{F}_{-1,R+1}\otimes \mathcal{C}^{\mathfrak{p}}_2\mathcal{F}_{-1,R+1},
    \end{equation*}
    likewise for the other term comprising $\mathcal{R}^{\mathrm{drift}}$. Lastly, $\mathcal{R}^{\otimes 2}$ contains the terms in $(\delta u_{st})\ftensor (\delta u_{st})$ that are not cancelled by the quantized terms $\Gamma^n_{st}u_s^{\otimes 2}$ in the tensorization:
    \begin{equation}
    \label{Eqn: mathcal_R_otimes2}
        \mathcal{R}^{\otimes 2}_{st} := (\delta u_{st})^{\otimes 2} - \sum_{n=1}^N\Gamma^n_{st}u^{\otimes 2}_s = (\delta u_{st})^{\otimes 2} - \sum_{k+l \leq N} \mathrm{A}^k_{st}u_s \ftensor \mathrm{A}^l_{st}u_s.
    \end{equation}
    Recall from the proof of Proposition \ref{prop:a priori remainder} the partial remainders 
    \begin{equation*}
        u^{\sharp,(m)}_{st} := \delta u_{st}-\sum_{k=1}^m \mathrm{A}^{k}_{st}u_s = \delta \mu_{st} + \sum_{k=m+1}^N \mathrm{A}^{k}_{st}u_s + u^{\natural}_{st}
    \end{equation*}
    and, by writing the sum in \eqref{Eqn: mathcal_R_otimes2} telescopically, we observe that
    \begin{align*}
        \mathcal{R}^{\otimes 2}_{st} &= \left(\delta u_{st}- \sum_{k=1}^N \mathrm{A}^{k}_{st}u_s\right)\ftensor \delta u_{st} + \sum_{k=1}^N \mathrm{A}^k_{st}u_s\ftensor \left(\delta u_{st}-\sum_{l=1}^{N-k}\mathrm{A}^l_{st}u_s\right) \\
        &= (\delta \mu_{st} + u^\natural_{st})\ftensor \delta u_{st} + \sum_{k=1}^N \mathrm{A}^k_{st}u_s\ftensor u^{\sharp,(N-k)}_{st}.
    \end{align*}
    Every factor of $\mathrm{A}^k_{st}u_s\ftensor u^{\sharp,(N-k)}_{st}$ present in this final sum has the appropriate regularity
    \begin{equation*}
        \mathrm{A}^{k}_{st}u_s \in \mathcal{C}^{\mathfrak{p}/(k+1)}_2 \mathcal{F}_{-(k+1),R+1} \quad \text{ and } \quad  u^{\sharp,(N-k)} \in \mathcal{C}^{\mathfrak{p}/(N-k+1)}_2\mathcal{F}_{-(N-k+1),R+1},
    \end{equation*}
    where the latter follows from Lemma \ref{Lem: partial remainder estimate},
    and it follows that the distributional tensor product of the two is contained in $\mathcal{C}^{\mathfrak{p}/(N+1)}_2\mathcal{E}_{-(N+1),R}$ by Proposition \ref{Prop: distr tensor product emb}. We combine the term $u^\natural_{st}\ftensor \delta u_{st}$ from $\mathcal{R}^{\otimes 2}_{st}$ with the terms of $\mathcal{R}^{\mathrm{old}}_{st}$, along with $u\in \mathcal{B}_b([0,T];\mathcal{F}_{-0,R+1})$, to obtain
    \begin{equation*}
        u^{\natural}_{st}\ftensor u_t + u_s \ftensor u^{\natural}_{st} \in \mathcal{C}^{\mathfrak{p}/(N+1)}_2\mathcal{E}_{-(N+1),R}.
    \end{equation*}
    Finally, Proposition \ref{Prop: distr tensor product emb} implies that both terms comprising $\mathcal{R}^{\mathrm{drift}}_{st}$ and the term $\delta \mu \ftensor \delta u$ are in the class $\mathcal{C}^{\mathfrak{p}/(N+1)}_2\mathcal{E}_{-(N+1),R}$. Similar arguments also yield $M\in \mathcal{C}^{1\mathrm{-var}}\mathcal{E}_{-1,R}$.
\end{proof}

As mentioned previously, doubled tensor fields carry more structure than generic tensor fields over $\mathbb{R}^{2d}$ -- this is also true of the derivative operators that are cast from these. Let $z= (x,y)$. We define a {\it restricted casting operator} $\bar{\mathcal{A}}$ on doubled, decomposable tensor fields $\bar{F}= \bar{F}_1\otimes \cdots \otimes \bar{F}_n$ by
\begin{equation*}
    \bar{\mathcal{A}}^{(n)}_z (\bar{F}_1\otimes \cdots \otimes \bar{F}_n) = \sum_{k+l = n} \sum_{\sigma \in S_{k,l}} \mathcal{A}_x^{(k)}\ftensor \mathcal{A}_y^{(l)} (F^{(1)}_{\sigma(1)}\otimes \cdots \otimes F^{(1)}_{\sigma(k)}\otimes F^{(2)}_{\sigma(k+1)}\otimes\cdots\otimes F^{(2)}_{\sigma(k+l)})
\end{equation*}
where each factor $\bar{F}_i = (F^{(1)}_i,  F^{(2)}_i)^T = F_i ^{(1)}\ftensor F_i^{(2)}$ for $i=1,\ldots,n$ is doubled in the sense that $F^{(1)}_i$ depends only on $x$ and $F^{(2)}_i$ depends only on $y$. The role of the shuffles $\sigma\in S_{k,l}$ is to match the derivative operators in $x$ and $y$ accordingly. The extended definition to doubled, non-decomposable tensor fields $\bar{F}$ over $\mathbb{R}^{2d}$ reads
\begin{equation}
    \label{Eqn: mathcal Abar tensor field def}
    \bar{\mathcal{A}}^{(n)}_z(\bar{F}) = \sum_{k+l = n}\sum_{\sigma\in S_{k,l}} \mathcal{A}^{(k)}_x\ftensor \mathcal{A}^{(l)}_y ((\pi^{\otimes k}_1 \otimes\pi^{\otimes l}_2) P_\sigma\bar{F}).
\end{equation}
Indeed, whenever $\bar{F}$ is doubled as a tensor field over $\mathbb{R}^{2d}$, the restricted casting operator agrees with the usual casting operator in the sense that $\mathcal{A}^{(n)}_z (\bar{F}) = \bar{\mathcal{A}}^{(n)}_z (\bar{F})$.

\begin{example}
    We use an example to illustrate the role of the restricted casting operator. Consider first $n=1$, and recall that $\mathcal{A}^{(1)}_z (\bar{\mathrm{X}}^{(1)}_{st})$ is simply the first-order operator
    \begin{equation}
    \label{Eqn: example X^1 restr}
        \mathcal{A}^{(1)}_z (\bar{\mathrm{X}}^{(1)}_{st})\,\Phi =  \bar{\mathrm{X}}^{(1)}_{st}\cdot \nabla_z \Phi.
    \end{equation}
    The vector field $\bar{\mathrm{X}}^{(1)}_{st}$ is doubled from $\mathrm{X}^{(1)}_{st}$, hence $\bar{\mathrm{X}}^{(1)}_{st}(z) = (\mathrm{X}^{(1)}(x),\mathrm{X}^{(1)}(y))^T$ can be used to split the dot product in \eqref{Eqn: example X^1 restr} as follows, for functions $\Phi(z)= \Phi(x,y)$
    \begin{align*}
        \mathcal{A}^{(1)}_z (\bar{\mathrm{X}}^{(1)}_{st})\,\Phi &= \pi_1 \bar{\mathrm{X}}^{(n)}_{st} \cdot \nabla_x \Phi + \pi_2 \bar{\mathrm{X}}^{(n)}_{st} \cdot \nabla_y \Phi \\
        &= \mathcal{A}^{(1)}_x (\pi_1 \bar{\mathrm{X}}^{(1)}_{st})\,\Phi + \mathcal{A}^{(1)}_y (\pi_2 \bar{\mathrm{X}}^{(1)}_{st})\, \Phi \\
        &= \mathcal{A}^{(1)}_x\ftensor \mathcal{A}^{(0)}_y \big(\pi_1 \otimes \mathrm{Id}(\bar{\mathrm{X}}^{(1)}_{st})\big)\,\Phi + \mathcal{A}^{(0)}_x\ftensor \mathcal{A}^{(1)}_y \big(\mathrm{Id}\otimes \pi_2 (\bar{\mathrm{X}}^{(1)}_{st})\big)\,\Phi \\
        &= \bar{\mathcal{A}}^{(1)}_z (\bar{\mathrm{X}}^{(1)}_{st})\,\Phi.
    \end{align*}
    Note that in the intermediate step, we formally obtain 
    \begin{equation*}
        \mathcal{A}_x^{(1)}(\mathrm{X}^{(1)}_{st}) \,\Phi(x,y) + \mathcal{A}_y^{(1)}(\mathrm{X}^{(1)}_{st})\,\Phi(x,y),    
    \end{equation*}
    reducing the original operator in \eqref{Eqn: example X^1 restr} to the casting applied to the underlying vector field $\mathrm{X}^{(1)}_{st}$. Recall from Example \ref{Ex: writing out mathcal A} that the $n=2$ casting operator writes
    \begin{equation}
        \mathcal{A}^{(2)}_z (\bar{\mathrm{X}}^{(2)}_{st}) = \mathrm{Tr}_{\bar{z}=z} \nabla_z \cdot (\bar{\mathrm{X}}^{(2)}_{st}(z,\bar{z})\cdot \nabla_z \Phi(z)).
    \end{equation}
    Splitting this expression similarly to the case $n=1$, now with $\bar{z} = (\bar{x},\bar{y})$, we obtain
    \begin{align*}
        \mathcal{A}^{(2)}_z (\bar{\mathrm{X}}^{(2)}_{st}) &= \mathrm{Tr}_{\bar{z} = z}\, \big[\nabla_x \cdot \pi_{11} \bar{\mathrm{X}}^{(2)}_{st} (x,\bar{x})\cdot \nabla_x \Phi + \nabla_x\cdot \pi_{21}\bar{\mathrm{X}}^{(2)}_{st}(y,\bar{x})\cdot \nabla_y \Phi \\
        &\qquad\qquad\,\,  +\nabla_y\cdot \pi_{12}\bar{\mathrm{X}}^{(2)}_{st}(x,\bar{y})\cdot \nabla_x \Phi + \nabla_y\cdot \pi_{22}\bar{\mathrm{X}}^{(2)}_{st}(y,\bar{y})\cdot \nabla_y \Phi\big] \\
        &= \mathcal{A}_x^{(2)}\ftensor \mathcal{A}^{(0)}_y(\pi_{11}\bar{\mathrm{X}}^{(2)}_{st}) + \mathcal{A}^{(1)}_x\ftensor \mathcal{A}_y^{(1)} (\pi_{12}\bar{\mathrm{X}}^{(2)}_{st}) \\
        &\quad + \mathcal{A}_y^{(1)}\ftensor\mathcal{A}_x^{(1)} (\pi_{21}\bar{\mathrm{X}}^{(2)}_{st}) + \mathcal{A}^{(0)}_x\ftensor\mathcal{A}_y^{(2)}(\pi_{22}\bar{\mathrm{X}}^{(2)}_{st}) \\
        &= \bar{\mathcal{A}}^{(2)}_z (\bar{\mathrm{X}}^{(2)}_{st})
    \end{align*}
    where, again, the intermediate step could be reduced to a sum of casting operators on $\mathrm{X}^{(2)}_{st}(x,x)$, $\mathrm{X}^{(2)}_{st}(y,x)$, etc. 
\end{example}

The following result shows that the tensorization $\Gamma = \{\Gamma^n\}_{n=1}^N$ is intimately connected to the DiPerna--Lions lift $\bar{\mathbf{X}}$ once we assume that its underlying rough path $\mathbf{X}$ is weak geometric. Schematically, it also proves that $\bar{\mathcal{A}} = \Delta_{\mathrm{op}}\circ {\mathcal{A}} = (\mathcal{A}\ftensor \mathcal{A}) \circ \Delta$, thus relating the coproducts $\Delta_{\mathrm{op}}$ for operators and $\Delta$ for the tensor field Hopf algebra. 
\begin{proposition}
\label{Prop: tensorization-mathcalAbar equivalence}
    Let $\Gamma = \{\Gamma^n\}_{n=1}^N$ be the tensorization introduced in \eqref{Eqn: tensorization def}. Assume that the underlying tensor field rough path is weak geometric; $\mathbf{X}\in \mathcal{C}^{\mathfrak{p}}_{wg}\mathcal{W}^N(\mathbb{R}^d)$. Then, we have the identity
    \begin{equation}
        \Gamma^n_{st} = \bar{\mathcal{A}}^{(n)}_z (\bar{\mathrm{X}}^{(n)}_{st})
    \end{equation}
    for every $1\leq n \leq N$, where $\bar{\mathbf{X}}$ is the DiPerna--Lions lift of $\mathbf{X}$ as in Definition \ref{Def: DiPerna-Lions lift}.
\end{proposition}
\begin{proof}
By the definition of $\Gamma^n_{st}$, we have
\begin{equation*}
    \Gamma^n_{st} = \sum_{k+l = n} \mathrm{A}^k_{st} \ftensor \mathrm{A}^l_{st} = \sum_{k+l = n} \mathcal{A}^{(k)}_x(\mathrm{X}^{(k)}_{st}) \ftensor \mathcal{A}^{(l)}_y(\mathrm{X}^{(l)}_{st}) = \sum_{k+l = n} \mathcal{A}^{(k)}_x\ftensor \mathcal{A}^{(l)}_y (\mathrm{X}^{(k)}_{st} \ftensor \mathrm{X}^{(l)}_{st})
\end{equation*}
where, in accordance with Remark \ref{Rem: bar-convention}, the bars indicate that the variables have been doubled to that of $x$ and $y$. Continuing, by weak geometricity of $\mathbf{X}$, the latter term above writes
\begin{align*}
    \sum_{k+l=n}\mathcal{A}^{(k)}_x\ftensor \mathcal{A}^{(l)}_y (\mathrm{X}^{(k)}_{st}\ftensor \mathrm{X}^{(l)}_{st}) = \sum_{k+l=n}\sum_{\sigma\in S_{k,l}}\mathcal{A}^{(k)}_x\ftensor \mathcal{A}^{(l)}_y (P_\sigma \mathrm{X}^{(n)}_{st}).
\end{align*}
As indicated by the identity \eqref{eq:weak geometric}, the arguments of $P_\sigma \mathrm{X}^{(n)}_{st}$ are a mix of $x_i$'s and $y_j$'s that are subsequently shuffled by $\sigma$. This allows us to compare $P_\sigma \mathrm{X}^{(n)}_{st}$ with the components $\bar{\mathrm{X}}^{(n)}_{st}$. As an example, take $k=2$ and $l=1$ in the case where $n=3$. Then
\begin{align*}
    \mathrm{X}^{(2)}_{st}(x_1,x_2)\ftensor \mathrm{X}^{(1)}_{st}(y_3) &= P_{\sigma_1}\mathrm{X}^{(3)}_{st}(x_1,x_2,y_3) + P_{\sigma_2}\mathrm{X}^{(3)}_{st}(x_1,y_3,x_2) + P_{\sigma_3}\mathrm{X}^{(3)}_{st}(y_3,x_1,x_2) \\
    &= P_{\sigma_1} \pi_{112}\bar{\mathrm{X}}^{(3)}_{st} (z_1,z_2,z_3) + P_{\sigma_2} \pi_{121}\bar{\mathrm{X}}^{(3)}_{st} (z_1,z_3,z_2) + P_{\sigma_3} \pi_{211}\bar{\mathrm{X}}^{(3)}_{st}(z_3,z_1,z_2).
\end{align*}
Generalizing, the DiPerna--Lions lift allows us to write
\begin{equation*}
    \mathrm{X}^{(k)}_{st}\ftensor \mathrm{X}^{(l)}_{st} = \sum_{\sigma\in S_{k,l}}P_\sigma \mathrm{X}^{(n)}_{st} = \sum_{\sigma\in S_{k,l}} P_\sigma \pi_{[\sigma]_{k,l}} \bar{\mathrm{X}}^{(n)}_{st}
\end{equation*}
where $[\sigma]_{k,l}$ is a binary word of length $k+l$, defined for $\sigma\in S_{k,l}$ such that 
\begin{equation}
    [\sigma]_{k,l} = i_1\cdots i_{k+l},\quad  i_j = \begin{cases}
        1, & \sigma^{-1}(j)\leq k, \\
        2, & \sigma^{-1}(j) > k.
    \end{cases}
\end{equation}
Equivalently, one has the following commutation relation
\begin{equation*}
    P_\sigma \pi_{[\sigma]_{k,l}} \bar{\mathrm{X}}^{(n)}_{st} = ( \pi_{1}^{\otimes k}\otimes \pi_2^{\otimes l}) P_{\sigma}\bar{\mathrm{X}}^{(n)}_{st}
\end{equation*}
which, by the definition of the restricted casting \eqref{Eqn: mathcal Abar tensor field def} on doubled tensor fields, finally concludes that 
\begin{equation*}
    \Gamma^n_{st} = \bar{\mathcal{A}}^{(n)}_z (\bar{\mathrm{X}}^{(n)}_{st})
\end{equation*}
after summing over shuffles $\sigma \in S_{k,l}$ and indices $k,l$.
\end{proof}

\begin{remark}
    Inspecting the proof of Proposition \ref{Prop: tensorization-mathcalAbar equivalence}, we surmise that the DiPerna--Lions lift $\bar{\mathbf{X}}$ is weak geometric if and only if $\mathbf{X}$ is weak geometric. Indeed, since $\bar{\mathbf{X}}$ is doubled from $\mathbf{X}$, it suffices to show what happens for arbitrary binary projections of $\bar{\mathrm{X}}^{(n)}$. It is straightforward to see that
    \begin{equation*}
        \pi_{i_1\dots i_k}\otimes \pi_{i_{k+1}\dots i_n} (\bar{\mathrm{X}}^{(k)}_{st}\otimes \bar{\mathrm{X}}^{(l)}_{st}) = \sum_{k+l=n} \pi_{i_1\dots i_n}P_\sigma \bar{\mathrm{X}}_{st}
    \end{equation*}
    by the weak geometricity of $\mathbf{X}$. Conversely, the above identity is just the projected identity of weak geometricity of $\bar{\mathbf{X}}$, whence weak geometricity of $\mathbf{X}$ follows as well. 
\end{remark}

We conclude from Proposition \ref{Prop: tensorization-mathcalAbar equivalence} that the tensorization $\{\Gamma^n\}_{n=1}^N$ is simply the casting of a lifted tensor field rough path $\bar{\mathbf{X}}$. Moreover, if we assume that the lifted $\bar{\mathbf{X}}$ is admissible in the sense that
\begin{equation*}
    \|\bar{\mathrm{X}}^{(n)}_{st}\|_{W^{N,\infty}(\mathbb{R}^{2d})} + \|\bar{\mathrm{X}}^{(n)}_{st}\|_{\mathfrak{t},N-n+1}\leq w_{\bar{\mathbf{X}}}(s,t)^{n/\mathfrak{p}},
\end{equation*}
which was the case when $\mathrm{X}^{(n)}_{st} \in W^{N+1,\infty}(\mathbb{R}^d)$ in Lemma \ref{Lem: canonical lift}, then Lemma \ref{lemma:casting rough paths} implies that the tensorization is an unbounded rough driver on the scales $(\mathcal{E}_{l,R})_{0\leq l \leq N+1}$.

Continuing, we will transform the Sobolev valued rough path $\bar{\mathbf{X}}$ and its corresponding unbounded rough driver, now identified with the tensorization, which in turn results in manipulating the tensorization itself.

\subsection{Renormalizability of unbounded rough drivers}

We turn to studying the blow-up transformed equation of the form \eqref{Eqn: tensorized eqn dual blow up}. On this path, we motivate and repackage some of the tensor field-based results included in Section \ref{sec: casting operators}. In order to properly understand the various transformations and concepts introduced, we use the vector field case $n=1$ as a guiding example.

Consider the vector field $\mathrm{X}^{(1)}_{st}\colon \mathbb{R}^d \to \mathbb{R}^d$, along with the doubled vector field $\bar{\mathrm{X}}^{(1)}_{st}\colon \mathbb{R}^{2d}\to \mathbb{R}^{2d}$ canonically defined by $\bar{\mathrm{X}}^{(1)}_{st}(z) = (\mathrm{X}^{(1)}(x),\mathrm{X}^{(1)}(y))^T$. In Subsection \ref{sec:doubling rough path}, we saw how to symmetrize-antisymmetrize tensor fields, and moreover, we proved the representation 
\begin{equation*}
    \bar{\mathcal{A}}^{(n)}_z (\bar{\mathrm{X}}^{(n)}_{st}) = \mathcal{A}^{(n)}_{\pm}(\mathcal{S}^{\otimes n}\bar{\mathrm{X}}^{(n)}_{st})
\end{equation*}
given by Lemma \ref{Lem: sym-asym correspondance}, which for $n=1$ explicitly becomes
\begin{align}
\label{Eqn: explicit rep sym-asym X^1}
\begin{split}
    \bar{\mathrm{X}}^{(1)}_{st} \cdot \nabla_z \Phi &= \left(\frac{\pi_1\bar{\mathrm{X}}^{(1)}_{st} + \pi_2\bar{\mathrm{X}}^{(1)}_{st}}{2}\right)\cdot \nabla_+ \Phi + \left(\frac{\pi_1\bar{\mathrm{X}}^{(1)}_{st} - \pi_2 \bar{\mathrm{X}}^{(1)}_{st} }{2}\right)\cdot \nabla_- \Phi \\
    &= \left(\frac{\mathrm{X}^{(1)}_{st}(x) + \mathrm{X}^{(1)}_{st}(y)}{2}\right)\cdot \nabla_+ \Phi + \left(\frac{\mathrm{X}^{(1)}_{st}(x) - \mathrm{X}^{(1)}_{st}(y)}{2}\right)\cdot \nabla_- \Phi.
\end{split}
\end{align}
As mentioned earlier, this final expression transforms nicely under the dual blow-up transformation $T^*_\varepsilon$. Indeed, the strategy used in \cite{BaiGub2017} and \cite{DGHT2019} is to apply the blow-up (or its dual) to expressions that are symmetrized-antisymmetrized as above and using the commutation relation \eqref{Eqn: commutator rule derivatives} as described in the following lemma.

\begin{lemma}[Commutation relations]
\label{Lem: commutation relations}
    Let $z=(x,y)^T$ and let $\mathcal{S}$, $\mathcal{S}^{\varepsilon}$ be given as in \eqref{Eqn: mathcal S block matrix} and \eqref{Eqn: mathcal S eps def} respectively with $\mathcal{S}z = x_\pm$. The gradient operators comprising $\nabla_\pm = (\nabla_+, \nabla_-)^T = \mathcal{S}^{-1}\nabla_z$ commute under $T^*_\varepsilon$ according to
    \begin{equation}
    \label{Eqn: commutator rule derivatives}
        T^*_\varepsilon \nabla_+ = \nabla_+T^*_\varepsilon, \quad T^*_\varepsilon \nabla_- = \frac{1}{\varepsilon}\nabla_-T^*_\varepsilon.
    \end{equation}
    Compactly, \eqref{Eqn: commutator rule derivatives} equivalently writes $T_\varepsilon^* \nabla_\pm = \mathcal{S}^\varepsilon \nabla_z T_\varepsilon^*$. Consequently, the dual blow-up transformation $T^*_\varepsilon$ commutes with $\mathrm{div}_z = \mathrm{div}_{\pm}\, \mathcal{S}$ according to 
    \begin{equation}
    \label{Eqn: commutator rule divergence}
        T^*_\varepsilon \mathrm{div}_{\pm} \mathcal{S}V = \mathrm{div}_{\pm}\,\mathcal{S}^\varepsilon T^*_\varepsilon V
    \end{equation}
    where $V\colon \mathbb{R}^{2d} \to \mathbb{R}^{2d}$ is a vector field.
    Finally, in light of $T^*_\varepsilon$ being an isomorphism on $\mathbb{R}^{2d}$ for $\varepsilon\in (0,1)$, we commute and transform the diagonal trace according to the rule
    \begin{equation}
    \label{Eqn: commutator rule diagonal trace}
        T^*_\varepsilon\,\mathrm{Tr}_{\bar{z}=z} = \mathrm{Tr}_{\substack{\bar{x}_+ = x_+ \\ \bar{x}_- = x_-}} \,T^*_\varepsilon.
    \end{equation}
\end{lemma}

The final two commutation relations, \eqref{Eqn: commutator rule divergence} and \eqref{Eqn: commutator rule diagonal trace}, are included in light of our recursive formulae for the casting operators appearing in Section \ref{sec: casting operators}; the application of the dual blow-up transformation can be understood and computed directly from the recursive formulae if desired. Applying $T^*_\varepsilon$ on \eqref{Eqn: explicit rep sym-asym X^1} and using \eqref{Eqn: commutator rule derivatives}, we obtain
\begin{align*}
    T^*_\varepsilon(\bar{\mathrm{X}}^{(1)}_{st}\cdot\nabla_z\Phi) &= \left(\frac{\mathrm{X}^{(1)}_{st}(x_+ + \varepsilon x_-) + \mathrm{X}^{(1)}_{st}(x_+ - \varepsilon x_-)}{2}\right)\cdot \nabla_+\,T^*_\varepsilon\Phi \\
    &\quad + \left(\frac{\mathrm{X}^{(1)}_{st}(x_+ + \varepsilon x_-) - \mathrm{X}^{(1)}_{st}(x_+ - \varepsilon x_-)}{2}\right)\cdot \frac{\nabla_-}{\varepsilon}\,T^*_\varepsilon\Phi
\end{align*}
whereupon moving the $\varepsilon^{-1}$ denominator onto the difference quotient, we identify
\begin{equation*}
    T^*_\varepsilon(\bar{\mathcal{A}}^{(1)}_z(\bar{\mathrm{X}}^{(1)}_{st})\,\Phi ) = T^*_\varepsilon(\bar{\mathrm{X}}^{(1)}_{st}\cdot\nabla_z\Phi) = \mathrm{X}^{\varepsilon,(1)}_{st}\cdot \nabla_\pm T_\varepsilon^*\Phi = \mathcal{A}^{(1)}_\pm (\mathrm{X}^{\varepsilon,(1)}_{st})\,T_\varepsilon^*\Phi.
\end{equation*}
Indeed, this illustrates why we introduced the $\varepsilon$-transform in Section \ref{sec: casting operators}. Furthermore, Proposition \ref{Prop: pushforward F-eps representation} implies that the resulting calculation is equivalent to computing the push-forward under $\mathcal{S}^{-1}\mathcal{S}^\varepsilon$:
\begin{equation*}
    (\mathcal{S}^{-1}\mathcal{S}^\varepsilon)_*\bar{\mathcal{A}}^{(1)}_z(\bar{\mathrm{X}}^{(1)}_{st}) = T^*_\varepsilon (\bar{\mathcal{A}}^{(1)}_z(\bar{\mathrm{X}}^{(1)}_{st})).
\end{equation*}
Summarizing the above discussion and generalizing to arbitrary tensor fields, we arrive at the following representation result.
\begin{proposition}
\label{Prop: blow-up of tensorization}
    The blow-up transformation of the tensorization $\Gamma = \{\Gamma^n\}_{n=1}^N$ is equivalently described as the  unbounded rough driver $(\mathcal{A}_{\pm}^{(n)}(\mathrm{X}^{\varepsilon,(n)}))_{n=1}^N$; more concretely, we have the identity
    \begin{equation}
        \Gamma^{\varepsilon,n}_{st} = T^*_\varepsilon \Gamma^n_{st}T^{*,-1}_\varepsilon  = \mathcal{A}_{\pm}^{(n)}(\mathrm{X}^{\varepsilon, (n)}_{st})
    \end{equation}
    for every level $1\leq n \leq N$, where $\mathbf{X}^{\varepsilon}$ is the $\varepsilon$-transformed tensor field rough path of the DiPerna--Lions lift $\bar{\mathbf{X}}$.
\end{proposition}
\begin{proof}
    From Proposition \ref{Prop: tensorization-mathcalAbar equivalence}, we know that applying a (restricted) casting operator to the DiPerna--Lions lift results in the tensorization. Moreover, taking the push-forward of $\bar{\mathcal{A}}^{(n)}(\bar{\mathrm{X}}^{(n)}_{st})$ under $\mathcal{S}^{-1}\mathcal{S}^\varepsilon$ results in $\mathcal{A}^{(n)}_\pm (\mathrm{X}^{\varepsilon,(n)}_{st})$ by Proposition \ref{Prop: pushforward F-eps representation}. Hence it remains to convince ourselves that $T^*_\varepsilon \Gamma^n_{st} T^{*,-1}_\varepsilon$ is actually the push-forward of $\Gamma^n_{st}$ under $\mathcal{S}^{-1}\mathcal{S}^\varepsilon$. To see this, note that the definition of the push-forward writes
    \begin{equation*}
        [(\mathcal{S}^{-1}\mathcal{S}^\varepsilon)_* \Gamma^n_{st}\Phi](z) = [\Gamma^n_{st}(\Phi\circ \mathcal{S}^{-1}\mathcal{S}^\varepsilon)]((\mathcal{S}^{\varepsilon})^{-1}\mathcal{S}(z)) = [\Gamma^n_{st}(\Phi\circ \mathcal{S}^{-1}\mathcal{S}^\varepsilon)](N^\varepsilon z) 
    \end{equation*}
    where $N^\varepsilon$ is precisely the coordinate change that describes $T^*_\varepsilon$ on scalar functions:
    \begin{equation*}
        T^*_\varepsilon \Phi (z) = \Phi(N^\varepsilon z).
    \end{equation*}
    On the other hand, we have
    \begin{equation*}
        T^*_\varepsilon \Gamma^n_{st} T^{*,-1}_\varepsilon\Phi = T^*_\varepsilon [\Gamma^n_{st}(T^{*,-1}_\varepsilon \Phi)](z) = [\Gamma^n_{st}(\Phi\circ \mathcal{S}^{-1}\mathcal{S}^\varepsilon)](N^\varepsilon z).
    \end{equation*}
    Comparing the expressions above concludes the proof.
\end{proof}
\begin{remark}
\label{Rem: non-restrictability of blow-up Gamma}
    Note that the blow-up transformed tensorization $\Gamma^\varepsilon$ is cast using casting operators that are not restricted. Indeed, the components of the $\varepsilon$-transformed tensor field rough path $\mathbf{X}^{\varepsilon}$ are not doubled in the sense of Definition \ref{Def: doubled tensor functions}, since each component is comprised of sums and differences of tensor fields with disparate variable structures. 
\end{remark}

At this point, it may be instructive to consider examples that unpack the structure of the $\varepsilon$-transformed tensor field rough path $\mathbf{X}^\varepsilon$. We have already seen the case $n=1$, so let us first consider the terms $\pi_{12}\mathrm{X}^{\varepsilon, (2)}_{st}$ and $\pi_{21}\mathrm{X}^{\varepsilon, (2)}_{st}$ that appear as two of the four terms comprising $\mathrm{X}^{\varepsilon,(2)}_{st}$: we first consider the term
\begin{align*}
    \pi_{12} \mathrm{X}^{\varepsilon, (2)}_{st} &= \frac{(\pi_{11}\bar{\mathrm{X}}^{(2)}_{st} - \pi_{12}\bar{\mathrm{X}}^{(2)}_{st}) + (\pi_{21}\bar{\mathrm{X}}^{(2)}_{st} - \pi_{22}\bar{\mathrm{X}}^{(2)}_{st})}{4\varepsilon} \\
    &= \frac{\mathrm{X}^{(2)}_{st}(x_++\varepsilon x_-, x_++\varepsilon x_-)-\mathrm{X}^{(2)}_{st}(x_++\varepsilon x_-, x_+-\varepsilon x_-)}{4\varepsilon} \\
    &\quad + \frac{\mathrm{X}^{(2)}_{st}(x_+-\varepsilon x_-, x_++\varepsilon x_-)-\mathrm{X}^{(2)}_{st}(x_+-\varepsilon x_-, x_+-\varepsilon x_-)}{4\varepsilon}
\end{align*}
being symmetrized and antisymmetrized in the first and second variables, respectively; the presence of antisymmetrized arguments yields the singular $\varepsilon^{-1}$ scaling. On the other hand, the term
\begin{align*}
    \pi_{21} \mathrm{X}^{\varepsilon, (2)}_{st} &= \frac{(\pi_{11}\bar{\mathrm{X}}^{(2)}_{st} + \pi_{12}\bar{\mathrm{X}}^{(2)}_{st}) - (\pi_{21}\bar{\mathrm{X}}^{(2)}_{st} + \pi_{22}\bar{\mathrm{X}}^{(2)}_{st})}{4\varepsilon} \\
    &= \frac{\mathrm{X}^{(2)}_{st}(x_++\varepsilon x_-, x_++\varepsilon x_-) + \mathrm{X}^{(2)}_{st}(x_++\varepsilon x_-, x_+-\varepsilon x_-)}{4\varepsilon} \\
    &\quad - \frac{\mathrm{X}^{(2)}_{st}(x_+-\varepsilon x_-, x_++\varepsilon x_-) + \mathrm{X}^{(2)}_{st}(x_+-\varepsilon x_-, x_+-\varepsilon x_-)}{4\varepsilon}
\end{align*}
has the same singular scaling, now emerging from the antisymmetrization with respect to the first variable slots of $\bar{\mathrm{X}}^{(2)}_{st}$. This pattern of variable symmetrization-antisymmetrization holds for the higher-level components as well. For example, when $n=3$, one of the eight terms comprising $\mathrm{X}^{\varepsilon,(3)}_{st}$ schematically writes
\begin{align*}
    \pi_{122} \mathrm{X}^{\varepsilon, (3)}_{st} &= \frac{([\pi_{111}\bar{\mathrm{X}}- \pi_{112}\bar{\mathrm{X}}] - [\pi_{121}\bar{\mathrm{X}} - \pi_{122}\bar{\mathrm{X}}])}{8\varepsilon^2} \\
    &\quad +\frac{([\pi_{211}\bar{\mathrm{X}} -\pi_{212}\bar{\mathrm{X}}] - [\pi_{221}\bar{\mathrm{X}} - \pi_{222}\bar{\mathrm{X}}])}{8\varepsilon^2}
\end{align*}
where we have suspended the transformed arguments and suggestively introduced brackets; it is clear from the binary projections on the left-hand side that the second and third variable slots are antisymmetrized, and the number of indices equal to $2$ also reflects the exponent for the total singular scaling $\varepsilon^{-1}$.

We recall the (abstract) concept of a renormalizable rough driver from \cite{BaiGub2017}. Let $\mathbf{A} = (\mathrm{A}^n)_{n=1}^N$ be an unbounded rough driver with tensorization $\Gamma = \{\Gamma^n\}_{n=1}^N$. We say that $\mathbf{A}$ is {\it renormalizable} if its blow-up transformed tensorization $ \Gamma^\varepsilon  = \{T^*_\varepsilon \Gamma^nT^{*,-1}_\varepsilon\}_{n=1}^N$ is a bounded family of unbounded rough drivers on an appropriate scale of spaces. The following result provides sufficient conditions for the renormalizability of drivers in our context. 

\begin{proposition}[Renormalizability of drivers]
\label{Prop: renormalizability of drivers}
    Assume that the weak geometric tensor field rough path $\mathbf{X} = \{\mathrm{X}^{(n)}\}_{n=1}^N$ has increased regularity at every level:
    \begin{equation*}
        \mathrm{X}^{(n)}_{st} \in W^{N+1,\infty}((\mathbb{R}^d)^n;(\mathbb{R}^d)^{\otimes n}),
    \end{equation*}
    and also suppose that the control $w_{\bX} (s,t)$ satisfies
    \begin{equation*}
        \|\mathrm{X}^{(n)}_{st}\|_{W^{N+1,\infty}(\mathbb{R}^d)} \leq w_{\bX}(s,t)^{n/\mathfrak{p}}
    \end{equation*}
    for every level $1\leq n \leq N$ and for all $(s,t)\in \Delta_T^{(2)}$.
    Then, for every $1\leq R <\infty$, the unbounded rough driver $\mathbf{A}=\{\mathrm{A}^n\}_{n=1}^N$ given by $\mathrm{A}^n = \mathcal{A}^{(n)}(\mathrm{X}^{(n)})$ is renormalizable on the scale of spaces $(\mathcal{E}_{n,R})_{0\leq n \leq N+1}$ given in Definition \ref{Def: the scales E and F}; the blow-up transformation of the $n$-th level tensorization $\Gamma^n$ satisfies the estimate
    \begin{equation}
    \label{Eqn: eps-transf-Gamma bound} 
    \|T^*_\varepsilon\Gamma^n_{st}T^{*,-1}_\varepsilon\|_{\mathcal{L}(\mathcal{E}_{-k,R},\mathcal{E}_{-n-k,R})} \lesssim w_{\bX}(s,t)^{n/\mathfrak{p}},\quad 0 \leq k \leq N+1-n,
    \end{equation}
    for every $1\leq n \leq N$, uniformly over $(s,t)\in \Delta_T^{(2)}$, with proportionality constants independent of both $\varepsilon$ and $R$.
\end{proposition}
\begin{proof}
    First, the identity from Proposition \ref{Prop: blow-up of tensorization},
    \begin{equation*}
        T^*_\varepsilon\Gamma^n_{st}T^{*,-1}_\varepsilon = \mathcal{A}^{(n)}_\pm (\mathrm{X}^{\varepsilon,(n)}_{st}),
    \end{equation*} 
    implies that it is sufficient to show the boundedness of the dual operator 
    \begin{equation*}
        \mathcal{A}^{(n),*}_{\pm}(\mathrm{X}^{\varepsilon,(n)}_{st})\colon \mathcal{E}_{n+k,R}\to \mathcal{E}_{k,R}
    \end{equation*}
    holds independently of $\varepsilon$ and $R$ with $0\leq k \leq N+1-n$. This is exactly the content of Proposition \ref{Prop: renorm result tensor fields} with the choice of scales $E^\zeta_l = \mathcal{E}_{l,R}$. Note that every function $\Phi \in \mathcal{E}_{l,R}$ satisfies, in particular, the support condition $\mathrm{supp}\,\Phi \subset \{|x_-|\leq 1\}$. The conclusion drawn from the statement of Proposition \ref{Prop: renorm result tensor fields} equivalently writes
    \begin{equation*}
        \|T^*_\varepsilon \Gamma^n_{st}T^{*,-1}_\varepsilon\|_{\mathcal{L}(\mathcal{E}_{-k,R},\mathcal{E}_{-n-k,R})} \lesssim_{n,k}  \|\mathrm{X}^{(n)}_{st}\|_{W^{N+1,\infty}(\mathbb{R}^d)} \leq w_{\mathbf{X}}(s,t)^{n/\mathfrak{p}}
    \end{equation*}
    for $0 \leq k \leq N+1-n$, thus concluding the proof. 
\end{proof}

\subsection{Renormalizability of the drift}

This subsection is concerned with the blow-up transformation of the tensorized drift term $M_t$ that appears in \eqref{Eqn: tensorized eqn dual blow up}. 

Following \cite{DGHT2019}, we introduce an intermediate norm that will be helpful in closing estimates that use Hölder's inequality, among other things. 
\begin{lemma}[\protect{\cite[Section 5]{DGHT2019}}]
\label{Lem: intermediate norm}
    Define the intermediate norm $\|\cdot\|_{L^1_-L^\infty_+}$ by
    \begin{equation}
        \|\Phi\|_{L^1_-L^\infty_+} := \int_{\mathbb{R}^d}\mathrm{d}x_- \sup_{x_+} |\Phi (x_+ + x_-, x_+-x_-)|.
    \end{equation}
    Let $\Phi\in \mathcal{E}_{n,R}$ for some $0\leq n\leq N+1$. Then the intermediate norm satisfies 
    \begin{equation}
        \| \mathcal{D} \Phi \|_{L^1_-L^\infty_+} \leq \|\Phi\|_{\mathcal{E}_{n-k,R}}
    \end{equation}
    whenever $\mathcal{D}$ (possibly equal to the identity) is a differential operator of order $k$ that does not increase support. Moreover, the intermediate norm is blow-up invariant:
    \begin{equation}
    \label{Eqn: blow-up norm invariant}
        \|T_\varepsilon \Phi\|_{L^1_-L^\infty_+} = \|\Phi\|_{L^1_-L^\infty_+}.
    \end{equation}
    Finally, given two functions $f\in L^p_{\mathrm{loc}}(\mathbb{R}^d)$, $g\in L^q_{\mathrm{loc}}(\mathbb{R}^d)$ with $1/p+1/q = 1$ and $1\leq p,q\leq \infty$, there is a localized version of Hölder's inequality that writes
    \begin{equation}
        |\langle f\ftensor g, \Phi \rangle| \leq \|f\|_{L^p(B_{R+1})} \|g\|_{L^q(B_{R+1})} \|\Phi\|_{L^1_- L^\infty_+}.
    \end{equation}
\end{lemma}

\begin{proposition}
\label{Prop: bound of epsilon-tensorized drift}
    Assume that $u\in \mathcal{B}_b([0,T];L^\infty_{\mathrm{loc}}(\mathbb{R}^d))$ and $b\in L^1([0,T];W^{1,1}_{\mathrm{loc}}(\mathbb{R}^d;\mathbb{R}^d))$. Then there exists a constant $C$ such that $M^\varepsilon_t = T^*_\varepsilon M_t$ satisfies the estimate
    \begin{equation*}
        \| \delta M^\varepsilon_{st}\|_{\mathcal{E}_{-1,R}} \leq C \|u\|_{\mathcal{B}_b([0,T];L^\infty(B_{R+1}))}^2 \int_s^t \|b_r\|_{W^{1,1}(B_{R+1})}\,\mathrm{d}r
    \end{equation*}
    for all $(s,t)\in \Delta^{(2)}_T$ with $C$ independent of both $\varepsilon$ and $R$.
\end{proposition}
\begin{proof}
    The blow-up transformation $T_\varepsilon$ commutes with gradients $\nabla_x$, $\nabla_y$ according to 
    \begin{equation*}
        \nabla_x T_\varepsilon 
        = \frac{1}{2}T_\varepsilon \nabla_+ + \frac{1}{2\varepsilon}T_\varepsilon\nabla_-, \quad 
        \nabla_y T_\varepsilon = \frac{1}{2}T_\varepsilon \nabla_+ - \frac{1}{2\varepsilon}T_\varepsilon \nabla_-.
    \end{equation*}
    Then, testing the increment $\delta  M^\varepsilon_{st}$ against $\Phi\in \mathcal{E}_{1,R}$ and using the (weak) divergence identity $\mathrm{div}(b_r u_r) = u_r\,\mathrm{div}\,b_r + b_r\cdot \nabla u_r$, we have
    \begin{align}
        -\langle \delta M^\varepsilon_{st},\Phi \rangle &= -\int_s^t \langle \nabla_x \cdot (u_r b_r \ftensor u_r) + \nabla_y\cdot (u_r\ftensor u_rb_r), T_\varepsilon \Phi\rangle \,\mathrm{d}r \notag \\
        &\qquad +\int_s^t \langle (u_r\,\mathrm{div}_x\,b_r)\ftensor u_r + u_r\ftensor (u_r\mathrm{div}_y\,b_r), T_\varepsilon\Phi \rangle \,\mathrm{d}r \notag \\
        &= \int_s^t \frac{1}{2}\langle (u_r b_r) \ftensor u_r + u_r\ftensor (u_rb_r), T_\varepsilon \nabla_+ \Phi \rangle \,\mathrm{d}r \label{Eqn: I-epsilon-1}\\
        &\qquad +  \int_s^t\frac{1}{2\varepsilon}\langle (u_r b_r) \ftensor u_r - u_r\ftensor (u_rb_r),T_\varepsilon\nabla_-\Phi\rangle \,\mathrm{d}r \label{Eqn: I-epsilon-2} \\
        &\qquad + \int_s^t \langle (u_r\,\mathrm{div}_x\, b_r)\ftensor u_r, T_\varepsilon \Phi \rangle \,\mathrm{d}r \label{Eqn: I-epsilon-3} \\
        &\qquad + \int_s^t \langle u_r \ftensor (u_r\mathrm{div}_y\,b_r),T_\varepsilon \Phi\rangle \,\mathrm{d}r  \label{Eqn: I-epsilon-4} \\
        &=: I^\varepsilon_1 + I^\varepsilon_2 + I^\varepsilon_3 + I^\varepsilon_4. \notag
    \end{align}
    To bound the integrals $I^\varepsilon_1$, $I^\varepsilon_3$, and $I^\varepsilon_4$, we use the intermediate norm $\|\cdot\|_{L^1_- L^\infty_+}$ combined with the integrability $b\in L^1([0,T];W^{1,1}_{\mathrm{loc}})$. By H\"older's inequality and Lemma \ref{Lem: intermediate norm}, we estimate $I_1^\varepsilon$ as follows
    \begin{align*}
        I^\varepsilon_1 &\lesssim \int_s^t \int_{B_R}\int_{B_1} |u_r(x_++\varepsilon x_-) u_r(x_--\varepsilon x_-)| |b_r(x_++\varepsilon x_-)+b(x_+-\varepsilon x_-)| \\
        &\qquad\qquad\,\,|\nabla_+\Phi(x_++x_-, x_+-x_-) |\,\mathrm{d}x_-\mathrm{d}x_+\mathrm{d}r \\
        &\leq \int_s^t \int_{B_1} \|u_r (\,\cdot\, + \varepsilon x_-)\|_{L^\infty(B_R)} \|u(\,\cdot\,-\varepsilon x_-)\|_{L^\infty(B_R)} (\|b_r(\cdot +\varepsilon x_-)\|_{L^{1}}+\|b_r(\cdot -\varepsilon x_-)\|_{L^{1}})\,\mathrm{d}x_- \\
        &\qquad\qquad\,\, \|\nabla_+\Phi\|_{L^1_-L^\infty_+}\,\mathrm{d}r \\
        &\lesssim \int_s^t \|u_r\|_{L^\infty(B_{R+1})}^{2} \|b_r\|_{L^1(B_{R+1})} \|\Phi\|_{\mathcal{E}_{1,R}}\,\mathrm{d}u \\
        &\lesssim \|u\|_{\mathcal{B}_b([0,T];L^\infty(B_{R+1}))}^2 \int_s^t\|b_r\|_{L^{1}(B_{R+1})}\,\mathrm{d}r \,\|\Phi\|_{\mathcal{E}_{1,R}},
    \end{align*}
    and, similarly, both $I_3^\varepsilon$ and $I_4^\varepsilon$ are bounded by
    \begin{align*}
        I^3_\varepsilon \eqsim I^4_\varepsilon &\lesssim \|u\|_{\mathcal{B}_b([0,T];L^\infty(B_{R+1}))}^2 \int_s^t\| \mathrm{div}\,b_r\|_{L^{1}(B_{R+1})}\,\mathrm{d}r \,\|\Phi\|_{\mathcal{E}_{0,R}}.
    \end{align*}
    Dealing with $I^\varepsilon_2$, one needs the Sobolev regularity of $b$ to cancel the singular $\varepsilon^{-1}$-scaling. For a.e. $r\in [0,T]$ and a.e. $x_+ \in B_R$ and $x_- \in B_1$, the integrand of $I^\varepsilon_2$ reads
    \begin{align*}
        &u_r(x_+ + \varepsilon x_-)\, u_r(x_+-\varepsilon x_-) \frac{b_r(x_++\varepsilon x_-) - b_r(x_+-\varepsilon x_-)}{2\varepsilon} \\& = u_r(x_+ + \varepsilon x_-)\, u_r(x_+-\varepsilon x_-)\left(\int_{-1/2}^{1/2} \mathrm{D} b_r \big(x_++ 2\theta \varepsilon x_-\big)x_-\,\mathrm{d}\theta\right).
    \end{align*}
    Since $\Phi \in \mathcal{E}_{R,1}$ one has that $\nabla_-\Phi$ is compactly supported on $|x_-|\leq 1$, and therefore
    \begin{align}
        I^\varepsilon_2 &\lesssim \int_s^t \int_{B_R}\int_{B_1} |u_r(x_+ + \varepsilon x_-)\, u_r(x_+-\varepsilon x_-)| \int_{-1/2}^{1/2} |\mathrm{D}b_r(x_++2\theta \varepsilon x_-)|\,\mathrm{d}\theta \,\mathrm{d}x_-\,\mathrm{d}x_+ \notag\\
        &\qquad\qquad\quad\|\nabla_-\Phi\|_{L^1_-L^\infty_+} \mathrm{d}r \notag \\
        & \lesssim\int_s^t \|u_r\|_{L^\infty(B_{R+1})}^2 \|\mathrm{D}b_r\|_{L^{1}(B_{R+1})} \label{Eqn: bound for I-epsilon-2}\|\Phi\|_{\mathcal{E}_{1,R}}\mathrm{d}r \\
        &\lesssim \|u\|_{\mathcal{B}_b([0,T];L^\infty(B_{R+1}))}^2 \int_s^t \|\mathrm{D}b_r\|_{L^{1}(B_{R+1})}\,\mathrm{d}r \,\|\Phi\|_{\mathcal{E}_{1,R}} \notag.
    \end{align}
    Combining all of the above estimates proves the claim.
\end{proof}

\subsection{Convergence and renormalizability of the transport equation}

Finally, we are in a position to prove the announced convergence results of the terms involved in \eqref{Eqn: tensorized eqn dual blow up}, which will allow us to prove that our transport equation is renormalizable. We start by proving convergence of the drivers and distributional tensor product.
\begin{proposition}
\label{Prop: convergence of tensors and drivers}
    Let $\Psi\in \mathcal{E}_{N+1,2R}$ be a distinguished test function as in Definition \ref{Def: Psi function}. Assume that $u\in \mathcal{B}_b([0,T];L^\infty_{\mathrm{loc}}(\mathbb{R}^d))$, and that the unbounded rough driver $\mathbf{A}$ is cast from $\mathbf{X}$ -- both fulfilling the assumptions of Proposition \ref{Prop: renormalizability of drivers}.
    Then, as $\varepsilon\to 0$ we have
    \begin{equation}
    \label{Eqn: convergence of formal tensors}
        \langle T^*_\varepsilon u^{\otimes 2}_{t}, \Psi \rangle \longrightarrow \langle u_t^2, \phi \rangle,
    \end{equation}
    for every $0\leq t \leq T$, and also 
    \begin{equation}
    \label{Eqn: convergence of drivers}
        \langle \Gamma^{\varepsilon,n}_{st}\, T_\varepsilon^*u^{\otimes 2}_s, \Psi \rangle \longrightarrow \langle \mathrm{A}^n_{st} u_s^2,\phi \rangle,
    \end{equation}
    for every $1\leq n \leq N$ and all $(s,t)\in \Delta_T^{(2)}$.
\end{proposition}
\begin{proof}
    Fix an arbitrary $t\in [0,T]$. Aiming to first prove \eqref{Eqn: convergence of formal tensors}, note that the local boundedness of $u$ ensures the existence of some smooth function $U\in C^\infty_c(\mathbb{R}^d)$ with 
    \begin{equation}
        \|u_t - U\|_{L^1(B_{R+1})} \leq \frac{\delta}{2}
    \end{equation}
    for some arbitrary but fixed $\delta >0$. Consider the absolute difference
    \begin{align*}
        &|\langle T^*_\varepsilon u^{\otimes 2}_t ,\Psi\rangle  - \langle u_t^2,  \phi\rangle | \\
        &\qquad \leq |\langle T^*_\varepsilon u^{\otimes 2}_t, \Psi\rangle  - \langle T^*_\varepsilon U^{\otimes 2}, \Psi\rangle | + |\langle T^*_\varepsilon U^{\otimes 2}, \Psi\rangle  - \langle U^2,\phi\rangle| + |\langle U^2, \phi\rangle -\langle u_t^2,\phi\rangle |.
    \end{align*}
    Dealing with the first of the latter three terms, we observe that $\|T_\varepsilon\Psi\|_{L^{1}_-L^\infty_+} \lesssim 1$ uniformly in $R$ and $\varepsilon$, and consequently estimate, using Lemma \ref{Lem: intermediate norm}, 
    \begin{align*}
        &|\langle T^*_\varepsilon u^{\otimes 2}_t, \Psi\rangle  - \langle T^*_\varepsilon U^{\otimes 2}, \Psi\rangle | = |\langle u_t^{\otimes 2}, T_\varepsilon \Psi\rangle - \langle U^{\otimes 2}, T_\varepsilon \Psi\rangle| \\
        &\qquad \leq |\langle u_t \ftensor(u_t-U),T_\varepsilon\Psi\rangle| + |\langle u_t-U)\ftensor U,T_\varepsilon\Psi\rangle| \\
        &\qquad \lesssim \| u_t\|_{L^\infty(B_{R+1})}\, \|u_t-U\|_{L^1(B_{R+1})} + \|u_t-U\|_{L^1(B_{R+1})} \|U\|_{L^\infty(B_{R+1})} \\
        &\qquad \lesssim \delta.
    \end{align*}
    For the second term, it suffices to note that $U$ is continuous, whence we can bound the difference $|\langle T^*_\varepsilon U^{\otimes 2}, \Psi\rangle - \langle U^2, \phi\rangle| \leq \delta$ whenever $\varepsilon$ is sufficiently small. Finally, the third term is bounded by a similar argument as for the first, yielding
    \begin{align*}
        |\langle U^2,\phi\rangle - \langle u_t^2,\phi\rangle | \lesssim 2 \|u_t-U\|_{L^1(B_{R+1})}  \leq \delta.
    \end{align*}
    Since $\delta$ and $t\in [0,T]$ were arbitrary, the convergence \eqref{Eqn: convergence of formal tensors} follows. 

    Turning to prove the convergence \eqref{Eqn: convergence of drivers}, we study the absolute difference
    \begin{equation*}
         |\langle \Gamma^{\varepsilon,n}_{st} T^*_\varepsilon u^{\otimes 2}_s ,\Psi\rangle - \langle \mathrm{A}^n_{st}u^2, \phi\rangle|= |\langle \mathcal{A}^{(n)}_\pm (\mathrm{X}^{\varepsilon,(n)}_{st}) T^*_\varepsilon u^{\otimes 2}_s, \Psi\rangle - \langle \mathrm{A}^n_{st}u^2, \phi\rangle|,
    \end{equation*}
    where we commute operators onto the dual 
    \begin{equation*}
        \langle \mathcal{A}^{(n)}_\pm (\mathrm{X}^{\varepsilon,(n)}_{st}) T^*_\varepsilon u^{\otimes 2}_s, \Psi\rangle = \langle u^{\otimes 2}_s , T_\varepsilon \mathcal{A}_\pm^{(n),*}(\mathrm{X}^{\varepsilon,(n)}_{st})\Psi\rangle
    \end{equation*}
    and observe from the renormalizability of $\mathbf{A}$ from Proposition \ref{Prop: renormalizability of drivers} and Lemma \ref{Lem: intermediate norm} that
    \begin{equation*}
        T_\varepsilon \mathcal{A}_\pm^{(n),*}(\mathrm{X}^{\varepsilon,(n)}_{st})\Psi = T_\varepsilon (T^*_\varepsilon\Gamma^n_{st}T^{*,-1}_\varepsilon)^* \Psi  = (\Gamma^n_{st})^* T_\varepsilon \Psi
    \end{equation*}
    satisfies the estimate
    \begin{align*}
        \|T_\varepsilon (T^*_\varepsilon\Gamma^n_{st}T^{*,-1}_\varepsilon)^* \Psi\|_{L^{1}_-L^\infty_+} & =  \|(T^*_\varepsilon\Gamma^n_{st}T^{*,-1}_\varepsilon)^* \Psi\|_{L^{1}_-L^\infty_+} = \|(T^*_\varepsilon\Gamma^n_{st}T^{*,-1}_\varepsilon)^*\Psi\|_{L^{1}_-(B_1; L^\infty_+(B_R))} \\
        & \leq \|(T^*_\varepsilon\Gamma^n_{st}T^{*,-1}_\varepsilon)^*\Psi\|_{\mathcal{E}_{0,2R}} \lesssim_{\mathbf{X}} \|\Psi\|_{\mathcal{E}_{n,2R}},    
        \end{align*}
        which follows from the $L^{1}_-L^\infty_+$-norm invariance \eqref{Eqn: blow-up norm invariant} and the compact support of $\Psi$ in the first inequality. This estimate holds uniformly in $\varepsilon$ for all $R\geq 1$. With this estimate at our disposal, the approximation argument which proves the convergence is essentially the same as the above; what remains is to see how one approximates $|\langle U^{\otimes 2} , (\Gamma^n_{st})^*\, T_\varepsilon \Psi \rangle -\langle U^2, \mathrm{A}^{n,*}_{st}\phi\rangle |$. To this end, recall the tensorization identity 
        \begin{equation*}
            \Gamma^n_{st} = \sum_{k+l = n} \mathrm{A}^k_{st}\ftensor \mathrm{A}^{l}_{st}
        \end{equation*}
        whence we ascertain by the smoothness of $U$ and the classical Leibniz rule that
        \begin{align*}
            \langle U^{\otimes 2}, (\Gamma^n_{st})^*\,T_\varepsilon \Psi \rangle = \sum_{k+l = n} \langle\mathrm{A}^k_{st}\ftensor \mathrm{A}^{l}_{st} U^{\otimes 2} ,T_\varepsilon \Psi\rangle = \langle \mathrm{A}^n_{st} U^2, T_\varepsilon \Psi\rangle.
        \end{align*}
        The smoothness of $U$ and the definition of $\Psi$ imply that the limit $\varepsilon\to 0$ of the right-hand side above converges to $\langle \mathrm{A}^n_{st} U^2, \phi\rangle$. The remainder of the approximation argument now follows in the exact same way as the proof of \eqref{Eqn: convergence of formal tensors}. 
\end{proof}

\begin{proposition}
\label{Prop: convergence of drift}
    Let $u\in \mathcal{B}_b([0,T];L^\infty_{\mathrm{loc}}(\mathbb{R}^d))$ and $b\in L^1([0,T];W^{1,1}_{\mathrm{loc}}(\R^d;\R^d))$. Then for any $(s,t)\in \Delta^{(2)}_T$, the blow-up of the tensorized drift $M^\varepsilon_t = T^*_\varepsilon M_t$ converges according to
    \begin{equation}
    \label{Eqn: convergence of tensorized drift}
        \lim_{\varepsilon \to 0} \langle \delta M^\varepsilon_{st}, \Psi \rangle = -\int_s^t \langle u_r^2\,b_r,\nabla \phi \rangle \,\mathrm{d}r + \int_s^t \langle u_r^2\,\mathrm{div}\, b_r, \phi\rangle \,\mathrm{d}r.
    \end{equation}
\end{proposition}
\begin{proof}
    Recall the setup provided by the proof of Proposition \ref{Prop: bound of epsilon-tensorized drift}; testing $\delta M^\varepsilon_{st}$ against $\Psi$, one obtains the same integrals $I^\varepsilon_k$ as in \eqref{Eqn: I-epsilon-1}--\eqref{Eqn: I-epsilon-4}. Moreover, the distinguished test function $\Psi$ satisfies 
    \begin{equation*}
        T_\varepsilon \nabla_+ \Psi (x,y) = \nabla_+ T_\varepsilon\Psi(x,y) = \nabla \phi\left(\frac{x+y}{2}\right)\chi_\varepsilon (x-y).
    \end{equation*}
    Using this identity along with the integrability $u^2\,b\in L^1([0,T];L^1_{\mathrm{loc}})$, sending $\varepsilon\to 0$, we obtain
    \begin{align*}
        I^\varepsilon_1 &= \int_s^t\int_{B_{2R}}\int_{B_{2R}} u_r(x)\, u_r(y)\frac{b_r(x) + b_r(y)}{2} \cdot\nabla \phi\left(\frac{x+y}{2}\right)\chi_\varepsilon(x-y)\,\mathrm{d}x\,\mathrm{d}y \,\mathrm{d}r \\
        &\to \int_s^t \int_{B_{2R}} u_r^2(x)\,b_r(x)\cdot\nabla \phi(x)\,\mathrm{d}x\,\mathrm{d}r
    \end{align*}
    from the properties of the mollifier $(\chi_\varepsilon)_{0\leq \varepsilon\leq 1}$. This produces the first term of \eqref{Eqn: convergence of tensorized drift}; for the second, we consider the limit of $I^\varepsilon_3$ as $\varepsilon\to 0$:
    \begin{align*}
        I^\varepsilon_3 &= \int_s^t \int_{B_{2R}}\int_{B_{2R}} u_r(x)\,u_r(y)\,\mathrm{div}_x\,b_r(x) \,\phi\left(\frac{x+y}{2}\right)\chi_\varepsilon(x-y)\,\mathrm{d}x\,\mathrm{d}y\,\mathrm{d}r \\
        &\to \int_s^t\int_{B_{2R}} u_r^2(x)\,\mathrm{div}_x\,b_r(x)\,\phi(x)\,\mathrm{d}x\,\mathrm{d}r.
    \end{align*}
    It remains to see that $I^\varepsilon_2 + I^\varepsilon_4 \to 0$ as $\varepsilon\to 0$. On this path, starting with $I^\varepsilon_2$ we have
    \begin{align*}
        I^\varepsilon_2 = 2^{d}\int_s^t \int_{B_R}\int_{B_1}\int_{-1/2}^{1/2}& u_r(x_++\varepsilon x_-)\,u_r(x_+-\varepsilon x_-)\phi(x_+)\\
        &\,  
        \mathrm{D}b_r(x_++2\varepsilon\theta x_-)\,x_-\cdot \nabla_-\chi(2x_-)\,\,\mathrm{d}\theta \,\mathrm{d}x_-\,\mathrm{d}x_+ \,\mathrm{d}r.
    \end{align*}
    Moreover, since $u^2 \,\mathrm{D}b \,x_- \in L^1([0,T];L^1_{\mathrm{loc}}(\mathbb{R}^d;\mathbb{R}^d))$ and $|x_-| \leq 1$, the continuity of translation operators in $L^1_{\mathrm{loc}}$ implies that
    \begin{equation*}
        \lim_{\varepsilon\to 0} u_r(\,\cdot\,+\varepsilon x_-) \,u_r(\,\cdot\,-\varepsilon x_-) \,\mathrm{D}b_r(\,\cdot + 2\theta \varepsilon x_-) x_- = u_r^2\,\mathrm{D}b_r\,x_-
    \end{equation*}
    in $L^1(B_{R+1})$ for all $x_-\in B_1$, $\theta\in [-1,1]$, and Lebesgue-a.e.~$r\in [s,t]$. It follows that for almost every $\theta$ and $x_-$, we have
    \begin{align*}
        &\lim_{\varepsilon\to 0} \int_{B_R}  u_r(x_+ + \varepsilon x_-) \,u_r(x_+-\varepsilon x_-)\phi(x_+) \,\mathrm{D}b_r(x_+ + 2\theta \varepsilon x_-) x_-\,\mathrm{d}x_+ \\
        &\quad= \int_{B_R} u_r^2(x_+)\,\phi(x_+)\mathrm{D}b_r(x_+)x_-\,\mathrm{d}x_+.
    \end{align*}
    On the other hand, dominated convergence in $L^1(B_1;L^1(B_R))$ from the bound \eqref{Eqn: bound for I-epsilon-2} implies
    \begin{align*}
        \lim_{\varepsilon\to 0} I^\varepsilon_2 = 2^d\int_s^t \int_{B_R} \int_{B_1} u_r^2(x_+) \,\phi(x_+) (\mathrm{D}b_r(x_+)x_-) \cdot \nabla_- \chi (2x_-)\,\mathrm{d}x_-\,\mathrm{d}x_+\,\mathrm{d}r. 
    \end{align*}
    From integrating by parts, we observe that for $1\leq i,j\leq d$ with $\tilde{x} = x_-$:
    \begin{equation*}
        \int_{B_1} \tilde{x}_i \frac{\partial}{\partial \tilde{x}_j}\chi (2\tilde{x})\,\mathrm{d}\tilde{x} = -\frac{\delta_{ij}}{2^d},
    \end{equation*}
    whence the above limit of $I^\varepsilon_2$ reads
    \begin{equation*}
        \lim_{\varepsilon\to 0}I^\varepsilon_2 = -\int_s^t \int_{\mathbb{R}^d} u_r^2(x_+)\phi(x_+)\,\mathrm{div}\,b_r(x_+) \,\mathrm{d}x_+\,\mathrm{d}r.
    \end{equation*}
    Finally, the local integrability of $u_r^2 \,\mathrm{div} \,b_r \in L^1([0,T];L^1_{\mathrm{loc}})$ immediately yields
    \begin{equation*}
        \lim_{\varepsilon \to 0} I^\varepsilon_4 = \int_{s}^t \langle u_r^2\,\mathrm{div}\,b_r, \phi\rangle \,\mathrm{d}r = - \lim_{\varepsilon\to 0} I^\varepsilon_2
    \end{equation*}
    which concludes the proof. 
\end{proof}

The following theorem summarizes the convergence results and furthermore proves that the limit equation is again a rough driver equation, thus proving the renormalizability of the rough transport equation, as outlined earlier.

\begin{theorem}[Renormalizability of the transport equation]
\label{Thm: renormalizability of transport equation}
    Assume that the drift vector field $b$ satisfies
    \begin{equation*}
        b\in L([0,T]; L^\infty \cap W^{1,1}_{\mathrm{loc}}(\mathbb{R}^d;\mathbb{R}^d)), \quad \mathrm{div}\, b\in L^\infty(\mathbb{R}^d).
    \end{equation*}
    Let $u\in \mathcal{B}_b([0,T];L^\infty_{\mathrm{loc}}(\mathbb{R}^d))$ solve \eqref{Eqn: rough solution def}. Assume that $\mathbf{X}$ is a weak geometric tensor field rough path with increased regularity $\mathrm{X}^{(n)}_{st}\in W^{N+1,\infty}(\mathbb{R}^d)$ such that
    \begin{equation*}
        \|\mathrm{X}^{(n)}_{st}\|_{W^{N+1,\infty}} \leq w_{\mathbf{X}}(s,t)^{n/\mathfrak{p}}.
    \end{equation*}
    Then $u$ solves the rough transport equation
    \begin{equation}
    \label{Eqn: thm squared transport eqn}
        \partial_t u^2 = \partial_t\tilde{\mu} + \dot{\mathrm{X}}\cdot \nabla u^2,\quad \tilde{\mu}_t = \int_0^t b_r\cdot \nabla u^2_r\,\mathrm{d}r
    \end{equation}
    in the rough sense of Definition \ref{def:main eq}.
    \end{theorem}
\begin{proof}
    Let $R \geq 1$ be fixed. Let $\Psi$ be a distinguished function as in Definition \ref{Def: Psi function}, and note that its factor $\phi\in \mathcal{F}_{N+1,R}$ was arbitrary. The convergence results of Propositions \ref{Prop: convergence of tensors and drivers} and \ref{Prop: convergence of drift} imply that we can pass to the limit of every term in \eqref{Eqn: tensorized eqn dual blow up} except the remainder; the right-hand side of
    \begin{equation}
    \label{Eqn: remainder-limit-thm}
        \langle T^*_\varepsilon u^{\otimes 2,\natural}_{st},\Psi\rangle  = \langle \delta ( u^{\otimes 2})_{st},T_\varepsilon\Psi\rangle - \langle\delta M_{st},T_\varepsilon \Psi \rangle - \sum_{n=1}^N \langle \Gamma^n_{st} \,u^{\otimes 2}_s,T_\varepsilon\Psi \rangle 
    \end{equation}
    converges as $\varepsilon\to 0$ to 
    \begin{equation*}
        \langle \delta u^2_{st},\phi \rangle - \langle \delta\tilde{\mu}_{st},\phi\rangle - \sum_{n=1}^N\langle \mathrm{A}^n_{st}\,u_s^2,\phi\rangle.
    \end{equation*}
    Therefore, the left-hand side of \eqref{Eqn: remainder-limit-thm} also converges to some limit, say $\langle u^{2,\natural}_{st},\phi\rangle$. Clearly, one has $u^{2,\natural}_{st} \in \mathcal{F}_{-(N+1),R}$, and thus it remains to verify that $u^{2,\natural}$ has finite $\mathfrak{p}/(N+1)$-variation. To this end, consider the blow-up transformed tensorized equation \eqref{Eqn: tensorized eqn dual blow up} on the scales $(\mathcal{E}_{l,2R})_{0\leq l \leq N+1}$ and recall that we have uniform-in-$\varepsilon$ estimates for $\|\delta M^\varepsilon_{st}\|_{\mathcal{C}^{1\mathrm{-var}}\mathcal{E}_{-1,2R}}$ and the operator norm of the unbounded rough driver $\Gamma^\varepsilon_{st} = \{\Gamma^{\varepsilon,n}_{st}\}_{n=1}^N$ by Propositions \ref{Prop: bound of epsilon-tensorized drift} and \ref{Prop: renormalizability of drivers}. It follows that we can use Proposition \ref{prop:a priori remainder} on the expansion \eqref{Eqn: tensorized eqn dual blow up} to obtain 
    \begin{equation*}
        \sup_{\varepsilon >0} \|T^*_\varepsilon u^{\otimes 2,\natural}_{st}\|_{\mathcal{E}_{-(N+1),R}} \leq w_{*,R}(s,t)^{(N+1)/\mathfrak{p}}
    \end{equation*}
    for some control $w_{*,R}$ independent of $\varepsilon$. Finally, using $\|\Psi\|_{\mathcal{E}_{N+1,2R}} \lesssim \|\phi\|_{\mathcal{F}_{N+1,R}}$, we compute
    \begin{equation*}
        |\langle u^{2,\natural}_{st},\phi\rangle | = \lim_{\varepsilon\to 0}|\langle T^*_\varepsilon u^{\otimes 2,\natural},\Psi\rangle| \lesssim w_{*,R}(s,t)^{N+1/\mathfrak{p}}\|\phi\|_{\mathcal{F}_{N+1,R}}
    \end{equation*}
    from which it follows that $u^{2,\natural}\in \mathcal{C}_2^{\mathfrak{p}/(N+1)}\mathcal{F}_{-(N+1),R}$, as desired. Compared with Definition \ref{def:main eq}, we have shown that the expansion equation \eqref{Eqn: expansion eqn squared transport} has remainder term $u^{2,\natural}$ of finite $\mathfrak{p}/(N+1)$-variation, and hence $u\in \mathcal{B}_b([0,T];\mathcal{F}_{-0,R})$ also solves \eqref{Eqn: thm squared transport eqn}.
\end{proof}

\subsection{Uniqueness and stability of solutions}
In order to prove uniqueness and stability properties, we need to take into consideration the initial conditions of the rough transport problem that have until now been omitted from our discussion. We begin by proving $L^2$-uniqueness.
\begin{theorem}
\label{thm:uniqueness}
    Assume $b \in L^1([0,T];L^{\infty}\cap W^{1,1}_{\loc}(\R^d;\R^d) )$ with $\Div\, b \in L^1([0,T];L^{\infty}(\R^d))$. Assume $\bX$ is a weak geometric tensor field rough path with increased regularity and
    $$
    \|\mathrm{X}^{(n)}_{st}\|_{W^{N+1,\infty}} \leq w_{\bX}(s,t)^{n/\mathfrak{p}}.
    $$
    Then, any solution $u \in \clB_b([0,T];L^2(\R^d))$ of \eqref{eq:main eq} satisfies the energy estimate
    \begin{equation} 
    \label{eq:L^2 growth}
        \sup_{t \in [0,T]} \|u_t\|_{L^2} \leq C \|u_0\|_{L^2}
    \end{equation}
    for some constant $C$ 
    that does not depend on the solution. In particular, square integrable solutions of \eqref{eq:main eq} are unique. 
\end{theorem}

\begin{proof}
    By Theorem \ref{Thm: renormalizability of transport equation} we find that $u^2$ is a solution of \eqref{eq:main eq} with initial condition $u^2_0$. The proof is completed by using positivity of $u^2$ and invoking Theorem \ref{thm:positive solutions are uqnique}.
\end{proof}


    

Using similar techniques as for proving existence of a solution, Theorem \ref{thm:existence}, we can now prove stability properties of the solution, which in particular gives continuity of the It\^{o}--Lyons map in the strong $L^2_{\loc}(\R^d)$-topology. 

\begin{theorem} \label{thm:stability}
    Retain the assumptions on $b$ and $\bX$ from Theorem \ref{thm:uniqueness}, let $u_0 \in L^2(\R^d)$ be given and denote by $u$ the corresponding solution. In addition, we make the following assumptions.
    \begin{enumerate}
        \item 
        $b^{\varepsilon}$ is a family of drifts with convergence $\lim_{\varepsilon \rightarrow 0}b^{\varepsilon} = b$ in $L^1([0,T];L^1_{\loc}(\R^d))$, satisfying the uniform bounds
        $$
        \sup_{\varepsilon>0} \int_0^T \|b_r^{\varepsilon}\|_{L^{\infty}} + \|\Div b_r^{\varepsilon}\|_{L^{\infty}} + \|b_r^{\varepsilon}\|_{W^{1,1}(B_R)} dr < \infty
        $$
        for all $R \geq 1$. 
        \item 
        $(\bX^{\varepsilon})_{\varepsilon}$ is a family of tensor field rough paths such that
        $$
        \lim_{\varepsilon \rightarrow 0} \, \llbracket \mathrm{X}^{(n)} - \mathrm{X}^{\varepsilon,(n)} \rrbracket_{\frac{\mathfrak{p}}{n};W^{N+1,\infty}}  = 0.
        $$
        
        \item $u^{\varepsilon}_0$ is a family of $L^2(\R^d)$ functions such that 
        $$
        \lim_{\varepsilon \rightarrow 0} \|u_0^{\varepsilon} - u_0 \|_{L^2} = 0
        $$       
    \end{enumerate}
    Denote by $u^{\varepsilon}$ the solution corresponding to $b^{\varepsilon}, \bX^{\varepsilon}$ and $u_0^{\varepsilon}$. Then we have, for any $\phi \in C^{\infty}_c(\R^d)$
    \begin{equation} \label{eq:weak stability}
    \lim_{\varepsilon \rightarrow 0}  \sup_{t \in [0,T]} | \langle u_t^{\varepsilon} - u_t, \phi \rangle  | = 0.    
    \end{equation}
    Moreover, for any fixed $t \in [0,T]$,
    \begin{equation} \label{eq:strong stability}
    \lim_{\varepsilon \rightarrow 0}  \|u_t^{\varepsilon} - u_t \|_{L^2_{\loc}} = 0.
    \end{equation}
\end{theorem}

\begin{proof}
    We begin by proving \eqref{eq:weak stability}; take any subsequence of $u^{\varepsilon}$ and follow the same steps as in the proof of Theorem \ref{thm:existence} to extract a further subsequence $\{u^{\varepsilon_k}\}_{k \geq 1}$ such that 
    $$
    \lim_{k \rightarrow \infty }  \sup_{t \in [0,T]} | \langle u_t^{\varepsilon_k} - u_t, \phi \rangle  | = 0.    
    $$
    By uniqueness, the limit point $u$ is the same for each such subsequence we find that the full family converge as $\varepsilon \rightarrow 0$.     
    
    Next, we prove \eqref{eq:strong stability}.
    Using the triangle inequality and \eqref{eq:L^2 growth} coupled with a cut-off procedure, there is no loss of generality proving the result when $\{u_0^{\varepsilon}\}_{\varepsilon > 0}$ is a bounded subset of $L^1(\mathbb{R}^d)\cap L^{\infty}(\R^d)$. Under this assumption, $\{(u_0^{\varepsilon})^2 \}_{\varepsilon  > 0}$ is bounded in $L^2(\R^d)$. From Theorem \ref{Thm: renormalizability of transport equation}, we get that $v_t^{\varepsilon} := (u_t^{\varepsilon})^2$ satisfies the same equation as $u^{\varepsilon}$ with initial condition $v^{\varepsilon}|_{t=0} = (u_0^{\varepsilon})^2$. Applying \eqref{eq:weak stability} to $v^{\varepsilon}$ yields
    $$
    \lim_{\varepsilon \rightarrow 0}  \sup_{t \in [0,T]} | \langle v_t^{\varepsilon} - v_t, \phi \rangle  | = 0,
    $$
    with $v_t = (u_t)^2$. Finally, we write
    \begin{align*}
    \int_{R^d} \phi(x) |u_t^{\varepsilon}(x) - u_t(x)|^2 \,\mathrm{d}x & = \int_{R^d} \phi(x)\, v_t^{\varepsilon}(x) \,\mathrm{d}x + \int_{R^d} \phi(x) v_t(x) \,\mathrm{d}x  \\
    & - 2 \int_{R^d} \phi(x) \,u_t^{\varepsilon}(x)\, u_t(x) \,\mathrm{d}x 
    \end{align*}
    which converges to 0 as $\varepsilon \rightarrow 0$ since $\phi\, u_t \in L^2(\R^d)$. 
\end{proof}

Until now, we have called a solution $u$ renormalized provided $u^2$ satisfies the same dynamics as $u$. In \cite{diperna1989ordinary}, the notion of a renormalized solution is used for solutions where $\beta(u)$ satisfies the same dynamics as $u$, where $\beta$ satisfies some admissibility conditions. We obtain a similar result as a corollary of the above stability.

\begin{corollary}
    Retain the assumptions on $b$ and $\bX$ from Theorem \ref{thm:uniqueness} and assume in addition that $\bX$ is strong geometric. Let $u_0 \in L^2(\R^d)$ be given and denote by $u$ the corresponding solution with $u|_{t=0} = u_0$. Assume $\beta \in C^1_b$ is such that $\beta(u_0) \in L^2(\R^d)$. Then $\beta(u)$ is the unique solution of \eqref{eq:main eq} with initial condition $\beta(u_0)$.
\end{corollary}
\begin{proof}
    By assumption, we may construct a sequence of smooth solutions $u^{\varepsilon}$ as in Theorem \ref{thm:existence}. We have that $\beta(u^{\varepsilon})$ satisfies the same dynamics as $u^{\varepsilon}$, and we find that $\beta(u^{\varepsilon})$ converges in the weak topology (as in Theorem \ref{thm:existence}) to a solution $\bar{u}$. 
    Using Theorem \ref{thm:stability} and that $\beta \in C^1_b$, we find that $\beta(u_t^{\varepsilon}) \rightarrow \beta(u_t)$ in $L^2_{\loc}(\R^d)$. By uniqueness we have $\beta(u_t) = \bar{u}_t$, thus proving the result. 
\end{proof}

\appendix

\section{The construction of smoothing operators}
\label{App: Construction of smoothing operators}

Our goal is to prove Proposition \ref{prop:smoothing}: there exists a family of smoothing operators on the scale $(E_k^{\rho})_{0\leq k \leq N+1}$ when $\rho$ satisfies Assumption \ref{assumption:psi}. Before we proceed, we prove a technical lemma.

\begin{lemma} \label{lemma:choose xbar}
    Under Assumption \ref{assumption:psi} on $\rho$, for each $x$ such that $1- 2\eta \leq \rho(x) \leq 1- \eta$ there exists a point $x_{\eta}$ such that $\rho(x_{\eta}) = 1 + \eta$ and $|x-x_{\eta}| \leq \frac{3}{C_{\rho}}\eta$.
\end{lemma}
\begin{proof}
Let $x$ be given as in the statement and define $x_t : = x + t \nabla \rho(x)$. By Assumption \ref{assumption:psi} (i), if we let $t_0 := \inf \{ t > 0 : \rho(x_t) \geq 1+\eta\}$ we have $t_0 < \infty$. We now choose $x_{\eta} := x_{t_0}$ and note that by convexity we have
$$
3 \eta \geq \rho(x_{\eta}) - \rho(x) \geq \nabla \rho(x)(x_{\eta} - x) = t_0 |\nabla \rho(x)|^2 \geq t_0 |\nabla \rho(x)| C_{\rho}.
$$
The result now follows since $t_0 \nabla \rho(x) = x_{\eta} - x$. 
\end{proof}

Consider the inhomogeneous Littlewood-Payley blocks $\{\Delta_j \}_{j \geq -1}$ and the frequency cut-off
$$
S_ju =  \sum_{i < j} \Delta_i u,
$$
we refer to \cite{BCD} for details. 
On the full space, i.e.~on the unrestricted Sobolev scale $(W^{k,\infty}(\R^d) )_{0\leq k \leq N+1}$, we get a smoothing operator that satisfies the Bernstein estimates of the form
\begin{equation*}
    \|S_j u \|_{W^{k,\infty}} \leq \eta^{-l} \| u \|_{W^{k-l,\infty}}, \quad \|S_j u - u \|_{W^{k,\infty}} \leq \eta^{l}\|u\|_{W^{k+l,\infty}}    
\end{equation*}
when we choose $j = j(\eta) = \lfloor - \ln_2(\eta) \rfloor$; see \cite[Lemma 2.1]{BCD}. Note, however, that $S_j$ is not a suitable smoothing operator on the support-restricted scale $(E^{\rho}_k)_{0\leq k \leq N+1}$; for $u \in E^{\rho}_k$ we do not have $S_j u \in E^{\rho}_k$ since the support of $S_ju$ is necessarily the full space $\R^d$. 

To work around this, we introduce a suitable cut-off function as follows. 
Take a family $\{\theta_{\eta}\}_{\eta \in (0,1)}$ of smooth mappings $\theta_{\eta}\colon \R_+ \rightarrow \R_+$ such that 
$$
\theta_{\eta}|_{[0,1-2\eta)} \equiv 1, \quad \textrm{supp}(\theta_{\eta}) \subset [0,1-\eta), \quad  |\nabla^l \theta_{\eta}| \lesssim \eta^{-l},
$$
and let $\Theta_{\eta}(x) := \theta_{\eta}(\rho(x))$. Next, define
\begin{equation} \label{eq:smoothing definition}
    J^{\eta}u(x) = \Theta_{\eta}(x)\, S_ju(x)    
\end{equation}
and note that $J^{\eta}$ maps $W^{k,\infty}(\R^d)$ into $E_{l}^{\rho}$ for any $l =1,2,\dots$. We are ready to prove the main goal of this section, which constitutes the proof of Proposition \ref{prop:smoothing}.


\begin{proposition}
    Let $\rho$ satisfy Assumption \ref{assumption:psi} and define $(J^{\eta})_{\eta \in (0,1)}$ as in \eqref{eq:smoothing definition}. Then
    \begin{equation} \label{eq:smoothing}
        \|J^{\eta} u\|_{W^{k,\infty}} \lesssim \eta^{-m}\|u\|_{W^{k-m,\infty}}, \quad \text{ for } m\leq k \leq N+1,
    \end{equation}
    and 
    \begin{equation} \label{eq:desmoothing}
        \|(I-J^{\eta}) u\|_{W^{k,\infty}} \lesssim \eta^{m}\|u\|_{W^{k+m,\infty}}, \quad \text{ for } 0\leq k \leq N+1-m.
    \end{equation}
    It follows that $(J^{\eta})_{\eta \in (0,1)}$ is a smoothing operator on the scales $(E_k^{\rho})_{0\leq k \leq N+1}$. 
\end{proposition}
\begin{proof}
Write $\|\cdot\|_{k,\infty} = \|\cdot \|_{W^{k,\infty}}$. To show \eqref{eq:smoothing} it is enough to show that 
\begin{equation} \label{eq:smoothing on k infty}
\| \nabla^l J^{\eta} u\|_{0,\infty} \lesssim \eta^{-m} \|u\|_{k-m,\infty}.    
\end{equation}
for each $l \leq k$. We focus on the case $l = k$; the cases $l < k$ are shown similarly. 
First, we write 
$$
\nabla^k J^{\eta} u = \sum_{l=0}^k \nabla^{k-l} \Theta_{\eta} \nabla^l S_j u 
$$
and bound each term in the above sum. 
When $l = k$ we use that $\Theta_{\eta} \leq 1$ to get 
$$
|\Theta_{\eta} \nabla^kS_j  u | \leq | \nabla^kS_j  u | \lesssim \eta^{-m} \|u\|_{k-m,\infty}.
$$
For $l<k$ we write
$$
\nabla^{k-l} \Theta_{\eta} \nabla^l S_j u  = \nabla^{k-l} \Theta_{\eta} \nabla^l  u + \nabla^{k-l} \Theta_{\eta} \nabla^l ( S_j u -u).
$$
Note that $\nabla^{k-l} \Theta_{\eta}(x)$ is 0 except when $x$ is such that $1- 2 \eta \leq \rho(x) \leq 1- \eta$. Given any such $x$ we now use Lemma \ref{lemma:choose xbar} to find $x_{\eta}$ such that $|x - x_{\eta}| \lesssim \eta$ and $1<\rho(x_{\eta})$. Then, $\nabla^q u(x_{\eta}) = 0$ for $q = 0,\dots, k-m -1$ so that performing a Taylor expansion around $x_{\eta}$ we find
\begin{equation} \label{eq:derivatives via Taylor}
\nabla^lu(x) = \nabla^lu(x) - \nabla^l u(x_{\eta}) - \sum_{q=l}^{k-m - 1} \nabla^q u(x_{\eta}) \frac{(x-x_{\eta})^{\otimes q}}{q!} \lesssim \eta^{k+l-m} \|u\|_{k-m,\infty}
\end{equation}
which gives 
$$
|\nabla^{k-l} \Theta_{\eta} \nabla^l  u | \lesssim \eta^{-m} \|u\|_{k-m,\infty}.    
$$
Using $\nabla^l ( S_j u -u) \lesssim \eta^{k-l-m}\|u\|_{k-l-m,\infty}$ and $\nabla^{k-l} \Theta_{\eta} \lesssim \eta^{-(k-l)}$ we obtain the bound in \eqref{eq:smoothing on k infty}.

Next, we show \eqref{eq:desmoothing}. First write
$$
(I - J^{\eta}) \,u = (1 - \Theta_{\eta}) \,u + \Theta_{\eta}(I-S_j) \,u .
$$
For the first term we write 
$$
\nabla^k \big((1 - \Theta_{\eta}) \,u \big)= \sum_{l=0}^k \nabla^{k-l} (1 - \Theta_{\eta})  \nabla^l u.
$$
Note that $(1 - \Theta_{\eta}(x)) u (x) $ is non-zero only when $x$ is such that $1-2 \eta \leq \rho(x) \leq 1$. Using a \eqref{eq:derivatives via Taylor}  as above gives the bound
$$
\nabla^{k-l} (1 - \Theta_{\eta})  \nabla^l u \lesssim \eta^{-(k-l)} \eta^{m+k-l}\|u\|_{m+k} = \eta^{m}\|u\|_{m+k}.
$$
For the second term we write again
$$
\nabla^k\big( \Theta_{\eta} (I - S_j)\,u \big) = \sum_{l=0}^k \nabla^{k-l}  \Theta_{\eta}  \nabla^l(I-S_j) u.
$$
When $l=0$ we get
$$
 \Theta_{\eta}  \nabla^k(I-S_j) u \lesssim \|(I-S_j)u\|_{k,\infty} \lesssim \eta^{m}\|u\|_{m+k,\infty}
$$
and for $l>0$ we get
$$
\nabla^{k-l}  \Theta_{\eta}  \nabla^l(I-S_j) u \lesssim \eta^{-(k-l)}\|(I-S_j)u\|_{l,\infty} \lesssim \eta^m \|u\|_{m+k,\infty}.
$$
\end{proof}

\section{Weak compactness}\label{app:compactness}


The following result is a simplified version of \cite[Proposition B.4]{GLN}, recalled here for the convenience of the reader. 

\begin{proposition}\label{prop:weak compactness}
    Assume $\{u^{n}\}_{n \geq 1}$ is a sequence of maps $u^n\colon [0,T] \rightarrow L^2(\R^d)$  such that 
    $$
    \sup_{n \geq 1} \sup_{t \in [0,T]} \|u^n_t\|_{L^2} < \infty
    $$
    and for every $\phi \in C^{\infty}_c(\R^d)$ we have that the family of maps
    $$
    t \mapsto \langle u_t^n,\phi \rangle
    $$
    are equicontinuous. Then there exists a bounded map $u\in \mathcal{B}_b([0,T] ; L^2(\R^d))$ and a subsequence $\{u^{n_k}\}_{k \geq 1}$ converging to $u$ in the sense that for all $\phi \in C^{\infty}_c(\R^d)$ 
    $$
    \sup_{t \in [0,T]} | \langle u_t - u_t^{n_k}, \phi \rangle | \rightarrow 0
    $$
    as $k \rightarrow \infty$. 


\end{proposition}

\bibliographystyle{alpha}
\bibliography{bibliography}

\end{document}

%% file: preamble.tex
\usepackage[scale=0.73, hmarginratio=1:1, vmarginratio=2:3,top=1.6cm,headsep=20pt,includehead, bottom = 4cm, footskip=40pt]{geometry}%
\usepackage{amsmath}%
\usepackage{graphicx}
\usepackage{amsthm}
\usepackage[indent=18pt]{parskip}
\usepackage{fancyhdr}
\usepackage{etoolbox}
\usepackage{hyperref}

\usepackage{shuffle}

\usepackage{accents}

\usepackage{nccmath}

\usepackage{tcolorbox}
\usepackage{xcolor}
\definecolor{lightred}{rgb}{0.9, 0.5, 0.45}
\definecolor{lightgreen}{rgb}{0.61 0.97 0.72}

\numberwithin{equation}{section}

\providecommand{\U}[1]{\protect\rule{.1in}{.1in}}
\newtheorem{theorem}{Theorem}[section]

\newtheorem{corollary}[theorem]{Corollary}

\newtheorem{lemma}[theorem]{Lemma}
\newtheorem{assumption}[theorem]{Assumption}

\newtheorem{proposition}[theorem]{Proposition}

\theoremstyle{definition}
\newtheorem{definition}[theorem]{Definition}
\newtheorem{remark}[theorem]{Remark}
\newtheorem{notation}[theorem]{Notation}
\newtheorem{example}[theorem]{Example}

\AtEndEnvironment{definition}{\null\hfill$\diamond$}
\AtEndEnvironment{remark}{\null\hfill$\diamond$}
\AtEndEnvironment{notation}{\null\hfill$\diamond$}
\AtEndEnvironment{example}{\null\hfill$\diamond$}

\DeclareMathAlphabet\mathbfcal{OMS}{cmsy}{b}{n}
\makeatletter
\@tfor\next:=abcdefghijklmnopqrstuvwxyzABCDEFGHIJKLMNOPQRSTUVWXYZ\do{%
\def\command@factory#1{%
\expandafter\def\csname b#1\endcsname{\mathbf{#1}}
\expandafter\def\csname fk#1\endcsname{\mathfrak{#1}}
\expandafter\def\csname bb#1\endcsname{\mathbb{#1}}
\expandafter\def\csname cl#1\endcsname{\mathcal{#1}}
\expandafter\def\csname bcl#1\endcsname{\mathbfcal{#1}}
}
\expandafter\command@factory\next
}

\usepackage[charter,expert,cal=cmcal]{mathdesign}
\usepackage{XCharter}  

\newcommand{\R}{\mathbb{R}}
\newcommand{\loc}{\mathrm{loc}}
\newcommand{\Div}{\mathrm{div}}

\newcommand{\ftensor}{\mathbin{\accentset{\rule{1.0ex}{0.1ex}}{\otimes}}}
\newcommand{\ubar}[1]{\underaccent{\bar}{#1}}
\newcommand{\intprod}{\mathbin{\raisebox{\depth}{\scalebox{1}[-1]{$\lnot$}}}}

\newcommand{\vertiii}[1]{{\left\vert\kern-0.25ex\left\vert\kern-0.25ex\left\vert #1 
    \right\vert\kern-0.25ex\right\vert\kern-0.25ex\right\vert}}

\newcommand{\la}{[\![}
\newcommand{\ra}{]\!]}